\documentclass[review]{elsarticle}

\makeatletter
\def\ps@pprintTitle{%
  \let\@oddhead\@empty
  \let\@evenhead\@empty
  \let\@oddfoot\@empty
  \let\@evenfoot\@oddfoot
}
\makeatother

\usepackage{lineno,hyperref}
\modulolinenumbers[5]

\usepackage[flushleft]{threeparttable}
\usepackage[labelsep=space]{subcaption}
\usepackage{amsmath}
\usepackage{amssymb}
\usepackage{amsthm}
\usepackage{algorithm,algcompatible}
\usepackage[titletoc,title]{appendix} 
\usepackage[labelfont=bf, labelsep=space]{caption} 
\usepackage{comment}
\usepackage{diagbox}

\newtheorem{theorem}{Theorem}[section]

\journal{}

\usepackage{xcolor}
\usepackage{numcompress}
\biboptions{sort&compress}

\begin{document}

\begin{frontmatter}

\title{
Scalar auxiliary variable (SAV) stabilization of implicit-explicit (IMEX) time integration schemes for nonlinear structural dynamics
}

\author[1]{Sun-Beom Kwon\corref{mycorrespondingauthor}}
\ead{sbkwon@purdue.edu}
\author[1]{Arun Prakash}
\cortext[mycorrespondingauthor]{Corresponding author}
\ead{aprakas@purdue.edu}

\address[1]{Lyles School of Civil Engineering, Purdue University, West Lafayette, IN 47907, USA}

\begin{abstract}
Implicit-explicit (IMEX) time integration schemes are well suited for nonlinear structural dynamics because of their low computational cost and high accuracy. 
However, stability of IMEX schemes cannot be guaranteed for general nonlinear problems. 
In this article, we present a scalar auxiliary variable (SAV) stabilization of high-order IMEX time integration schemes that leads to unconditional stability.
The proposed IMEX-BDF$k$-SAV schemes treat linear terms implicitly using $k^{th}$-order backward difference formulas (BDF$k$) and nonlinear terms explicitly.
This eliminates the need for iterations in nonlinear problems and leads to low computational cost.
Truncation error analysis of the proposed IMEX-BDF$k$-SAV schemes confirms that up to $k^{th}$-order accuracy can be achieved and this is verified through a series of convergence tests.
Unlike existing SAV schemes for first-order ordinary differential equations (ODEs), we introduce a novel SAV for the proposed schemes that allows direct solution of the second-order ODEs without transforming them to a system of first-order ODEs.
Finally, we demonstrate the performance of the proposed schemes by solving several nonlinear problems in structural dynamics and show that the proposed schemes can achieve high accuracy at a low computational cost while maintaining unconditional stability.
\end{abstract}


\end{frontmatter}

\section{Introduction}
\label{sec:4th1}

Direct time integration is often used to solve the second-order ordinary differential equations (ODEs) governing the dynamic response of structures and wave propagation in solids \cite{chopra2007dynamics,zienkiewicz2005finite,hughes2012finite,bathe2016finite}.
Time integration is conventionally categorized into explicit and implicit schemes.
Explicit schemes generally have a low computational cost per time-step and are usually used for impact, crash, and short-time wave propagation problems.
However, they are almost always conditionally stable and are associated with a maximum critical time-step given by the Courant-Friedrichs-Lewy (CFL) condition \cite{courant1928partiellen}.
On the other hand, implicit schemes permit a greater range of time-steps and several are unconditionally stable.
However, this flexibility comes with a higher computational cost per time-step which is exacerbated for nonlinear problems due to the need for computing the solution iteratively at each time step. 
Nevertheless, implicit schemes are usually more suitable for large problems where the primary interest is in simulating the overall dynamics of a structure rather than capturing local transients.
Implicit-explicit (IMEX) schemes combine implicit and explicit schemes in such a way that exploits the benefits of both types of schemes.
IMEX schemes are well suited for nonlinear problems because they do not require iterations at each time-step and have a low computational cost.
However, stability of IMEX schemes cannot be guaranteed, especially for nonlinear problems \cite{ascher1995implicit}.
In this article, we propose a stabilization scheme for IMEX schemes that guarantees unconditional stability.
First, we provide a review of relevant literature on time integration schemes, especially focusing on stability for nonlinear problems.

In general, time integration schemes can be classified into two groups: multi-stage methods and multi-step methods.
In multi-stage methods, the time-step is divided into multiple intermediate stages and the intermediate state vectors are used to advance the solution to the next time-step.
Multi-step methods, on the other hand, advance the solution by using known state vectors from several previous time-steps.
Most commonly, however, one finds single-step (SS) time integration methods in the literature which lie at the intersection of multi-stage methods and multi-step methods.
Examples of SS methods include the Newmark-$\beta$ \cite{newmark1959method}, Wilson-$\theta$ \cite{wilson1968computer}, HHT-$\alpha$ \cite{hilber1977improved}, WBZ-$\alpha$ \cite{wood1980alpha}, CH-$\alpha$ \cite{chung1993time}, CL methods \cite{chung1994new} and numerous others.
Stability, accuracy and computational cost characteristics of multi-stage and multi-step methods are discussed next to provide context for the present work. 

In the class of multi-stage methods, the Runge-Kutta (RK) family of integrators is widely used~\cite{wang2023overview,butcher1996history,butcher2000numerical,butcher2016numerical}. 
RK methods may be explicit \cite{runge1895numerische,heun1900neue,kutta1901beitrag,nystrom1925numerische,butcher1964runge} or implicit \cite{kuntzmann1961neuere,butcher1964implicit,alexander1977diagonally,kennedy2016diagonally,ababneh2009design}.
In addition, multi-stage methods also include composite time integration schemes, which use different time integration schemes in different intermediate stages of each time-step \cite{bathe2005composite,bathe2007conserving,noh2013explicit,soares2016novel,kwon2017non,kim2017new,kim2018improved,kim2018improved2,malakiyeh2019bathe,malakiyeh2021new,noh2019bathe,kwon2020analysis,kwon2021selecting,kim2019higher,ji2020optimized}.
In general, implicit multi-stage methods can achieve high-order accuracy with unconditional stability~\cite{iserles2009first,wanner1996solving}.
On the other hand, explicit and IMEX multi-stage methods cannot be guaranteed to be unconditionally stable for general nonlinear problems \cite{verwer1996explicit, tan2022stability}.
In addition, multi-stage methods usually have a high computational cost per time-step which, for an $s$-stage method, is roughly equivalent to $s$ times the computational cost of a 1-step method.

Multi-step methods that use a linear combination of the known state vectors from previous time-steps are called linear multi-step (LMS) methods.
Examples of LMS methods include Adams-Bashforth \cite{bashforth1883attempt}, Adams-Moulton \cite{moulton1926new}, and BDF \cite{gear1971numerical}.
In LMS methods, two important results are noteworthy:
\begin{enumerate}
    \item[(1)] The first Dahlquist barrier \cite{dahlquist1956convergence,dahlquist1958stability} states that the order of accuracy of a stable linear $k$-step method cannot exceed $k+1$ if $k$ is odd, or $k+2$ if $k$ is even.
    \item[(2)] The second Dahlquist barrier \cite{dahlquist1963special} states that there is no explicit $A$-stable LMS method; the highest order of accuracy of an $A$-stable LMS method is 2; and the trapezoidal rule (TR) has the smallest error constant among the second-order $A$-stable LMS methods.
\end{enumerate}
From the second Dahlquist barrier, it is evident that high-order unconditionally stable LMS methods cannot be designed for linear or nonlinear structural dynamics.
Erlicher et al. \cite{erlicher2002analysis} conducted stability analysis of the generalized-$\alpha$ method, a second-order LMS scheme, for nonlinear structural dynamics and showed that the method is energy stable in the high-frequency range.

The preceding discussion shows that ensuring unconditional stability of high-order IMEX LMS schemes has not been achieved for general nonlinear problems. 
However, in recent years, the scalar auxiliary variable (SAV) approach, introduced by Shen and co-workers \cite{shen2018scalar,shen2019new,shen2018convergence} for gradient flows, has been used to develop efficient high-order time integration schemes with unconditional energy stability across multiple application domains.
With the SAV approach, one first defines a scalar variable $\Phi$ and augments the governing differential equations of the problem at hand with an evolution equation for this auxiliary variable which, when discretized, guarantees unconditional stability.
Shen and co-workers formulated several time integrators, such as $k^{\textrm{th}}$-order BDF (BDF$k$) for $1\le k \le 4$ and second-order Crank-Nicolson methods, based on the SAV approach, and proved their unconditional energy stability.
Note that BDF$k$ schemes without the SAV approach are only conditionally stable when $k\ge3$.
Lin et al.~\cite{lin2019numerical} developed BDF$k$ SAV schemes to solve incompressible Navier-Stokes equations.
Alternative approaches using different SAVs for Navier-Stokes equations have also been formulated \cite{li2020error, huang2021stability, li2022new, zhang2022unconditional}. 
Furthermore, Huang and Shen \cite{huang2022new} developed a class of implicit-explicit (IMEX) BDF$k$ SAV schemes for general dissipative systems, with $1\le k \le 5$.
They proved that the IMEX-BDF$k$-SAV scheme is $k^{\textrm{th}}$-order accurate and unconditionally energy stable for dissipative systems.
Other applications of the SAV approach include Schr{\"o}dinger equation \cite{deng2021second,poulain2022convergence}, magneto-hydrodynamics \cite{li2022stability}, electro-hydrodynamics \cite{he2023stability}, and modified phase-field crystal equation \cite{qi2023error}.
However, these SAV schemes are primarily designed to solve only first-order ordinary differential equations (ODEs).
While higher-order ODEs can certainly be transformed to a system of first-order ODEs and solved with the SAV approach, doing so may be computationally inefficient.

In this paper, we present a SAV stabilization of high-order IMEX BDF$k$ time integration schemes to solve general nonlinear problems in structural dynamics without converting the governing second-order ODEs to a set of first-order ODEs.
Note that, there is no existing high-order IMEX time integration method in the literature that can ensure unconditional stability for general nonlinear problems in structural dynamics. 
In section~\ref{sec:4th2}, we define a scalar auxiliary variable for structural dynamics which is crucial for ensuring unconditional stability and high-order accuracy.
Further, the proposed IMEX-BDF$k$-SAV schemes discretize linear terms using implicit time integration while treating nonlinear terms explicitly, thereby avoiding the need for iterations at each time-step.
We prove unconditional energy stability for linear as well as nonlinear structural dynamics.
We also carry out a truncation error analysis and convergence tests to show that the proposed IMEX-BDF$k$-SAV schemes achieve $k^{\textrm{th}}$-order accuracy.
Finally, in section {\ref{sec:4th3}}, we present several numerical examples
to illustrate the performance of the proposed IMEX-BDF$k$-SAV schemes over existing time integration methods in terms of stability, accuracy and computational cost.

\section{IMEX BDF$k$ SAV time integration schemes}
\label{sec:4th2}
The semi-discrete second-order ordinary differential equations governing the dynamic response of a nonlinear structural system can be written as:
\begin{equation}
    \begin{array}{lll}
        \boldsymbol{M}\ddot{\boldsymbol{u}}(t)+\boldsymbol{f}^{int}(\boldsymbol{u}(t),\dot{\boldsymbol{u}}(t))=\boldsymbol{f}^{ext}(t) & \textrm{in  } \, \Omega & \forall \, t \in [0,T] \\
        \boldsymbol{u}(t) = \boldsymbol{u}_D(t) & \textrm{on  } \, \Gamma_D & \forall \, t \in (0,T] \\
        \boldsymbol{u}(0) = \boldsymbol{u}_0, \quad \dot{\boldsymbol{u}}(0) = \boldsymbol{v}_0 & \textrm{in  } \, \Omega & \forall \, t = 0
    \end{array}
    \label{eqn:4thGDE}
\end{equation}
where $\boldsymbol{M}$, $\boldsymbol{u}$, $\boldsymbol{f}^{int}$, and $\boldsymbol{f}^{ext}$ represent the mass matrix, nodal displacement vector, internal force vector, and external force vector, respectively.
A dot over a quantity represents its time derivative.
Note that the mass matrix $\boldsymbol{M}$ is symmetric and positive definite.
The vector $\boldsymbol{u}_D$ is the prescribed displacement vector on the Dirichlet boundary $\Gamma_D$ over the time duration $(0,T]$, and $\boldsymbol{u}_0$ and $\boldsymbol{v}_0$ denote the initial conditions for the displacement and velocity vectors, respectively.
This governing differential equation may be solved approximately using a time integration scheme that discretizes the time of simulation $(0, T]$ into a number of time-steps $\Delta t$ and advances the solution one instant of time $t_n$ to the next $t_{n+1} = t_n + \Delta t$ using known kinematic quantities from previous time-steps.
The choice of time-step $\Delta t$ is usually problem dependent, but as a general guideline, one may pick the time-step as one-tenth of the time period corresponding to the largest frequency of interest.  
Note that for some time integration schemes, a special starting procedure may be needed at $t_0$ to compute the solution for the first few time-steps. 

To facilitate the development of the proposed IMEX-BDF$k$-SAV schemes, the internal force vector $\boldsymbol{f}^{int}$ can be decomposed, without loss of generality, into a non-negative linear part and a general nonlinear part:
\begin{equation}
    \boldsymbol{f}^{int}(\boldsymbol{u}(t),\dot{\boldsymbol{u}}(t)) = \underbrace{\boldsymbol{C}\dot{\boldsymbol{u}}(t) + \boldsymbol{K}\boldsymbol{u}(t)}_{\text{Non-negative linear part}} + \underbrace{\boldsymbol{f}^{NL}(\boldsymbol{u}(t),\dot{\boldsymbol{u}}(t))}_{\text{Nonlinear part}}
\end{equation}
where $\boldsymbol{C}$, $\boldsymbol{K}$, and $\boldsymbol{f}^{NL}$ denote the linear damping matrix, linear stiffness matrix, and the nonlinear internal force vector, respectively.
Note that matrices $\boldsymbol{C}$ and $\boldsymbol{K}$ are required to be  symmetric and positive semi-definite.
If, for certain unstable physical systems, the instantaneous tangent matrices $\boldsymbol{C}$ and/or $\boldsymbol{K}$ are negative definite, then they would be taken as $\boldsymbol{0}$.

\subsection{Definition of the scalar auxiliary variable (SAV) for structural dynamics}
In structural dynamics, we first define a pseudo-energy of a general nonlinear system as
\begin{equation}
    \Psi(t) = \frac{1}{2}\dot{\boldsymbol{u}}(t)^T\boldsymbol{M}\dot{\boldsymbol{u}}(t) + \frac{1}{2}\boldsymbol{u}(t)^T\boldsymbol{K}\boldsymbol{u}(t)
    \label{eqn:4thEnergy}
\end{equation}
Note that this pseudo-energy would be identical to the real energy for linear structural dynamics.
Next, we define the scalar auxiliary variable (SAV) for general nonlinear structural dynamics as:
\begin{equation}
    \Phi(t)=\Psi(t)+\psi
    \label{eqn:4thSAV_definition}
\end{equation}
where $\psi$ is a user-defined scalar value that ensures $\Phi(t)>0 \quad \forall t\in[0,T]$.
Note that, since psuedo-energy is always non-negative, one may choose a small non-zero value for $\psi$ to ensure that $\Phi(t)>0$.
While such a choice is valid for maintaining unconditional stability, in general, one should choose a sufficiently large value of $\psi$ to obtain high accuracy.  
As we show later in section~\ref{subsec:4th24} and section~\ref{sec:4th3}, a good rule of thumb for the value for $\psi$ is to choose it to be at least $100$ times the maximum value of psuedo-energy $\Psi(t)$ expected during the entire simulation.  

To formulate the IMEX-BDF$k$-SAV schemes, we define an evolution equation for $\Phi$ by computing its time derivative as:
\begin{equation}
    \frac{d\Phi(t)}{dt} = \frac{d(\Psi(t)+\psi)}{dt} = -\Theta(t)
    \label{eqn:4thSAV_ODE}
\end{equation}
where 
\begin{equation}
    \Theta=-\dot{\boldsymbol{u}}^T \left( \boldsymbol{M}\ddot{\boldsymbol{u}} + \boldsymbol{K}\boldsymbol{u} \right)=\dot{\boldsymbol{u}}^T \left( \boldsymbol{C}\dot{\boldsymbol{u}}-\boldsymbol{f}^{ext}+\boldsymbol{f}^{NL} \right)
    \label{eqn:4thSAV_theta}
\end{equation}
Note that for linear elasto-dynamics without external forces (i.e., $\boldsymbol{f}^{NL}=\boldsymbol{f}^{ext}=\boldsymbol{0}$), the energy in the system is dissipated (i.e., $\Theta \ge 0$) since $\boldsymbol{C}$ is positive semi-definite.

We now construct new IMEX-BDF$k$-SAV schemes by discretizing the equation of motion, Eq.~(\ref{eqn:4thGDE}), in time using an implicit-explicit approach, where nonlinear terms are treated using explicit integration (denoted with superscript EX) and linear terms are kept implicit (superscript IM):
\begin{align}
    &\boldsymbol{M}\boldsymbol{a}_{n+1}^{IM} + \boldsymbol{C}\boldsymbol{v}_{n+1}^{IM} + \boldsymbol{K}\boldsymbol{u}_{n+1}^{IM} +\boldsymbol{f}^{NL}\left(\boldsymbol{u}_{n+1}^{EX} ,\boldsymbol{v}_{n+1}^{EX} \right)=\boldsymbol{f}^{ext}_{n+1} \label{eqn:4thGE} \\
    &\boldsymbol{v}_{n+1}^{IM} = \frac{1}{\Delta t} \left\{ H^{(k)}\boldsymbol{u}_{n+1}^{IM} - \sum_{j=0}^{k-1}(-1)^j\frac{1}{j+1}{k \choose j+1} \boldsymbol{u}_{n-j} \right\} \label{eqn:4thBDFvel} \\
    &\boldsymbol{a}_{n+1}^{IM} = \frac{1}{\Delta t} \left\{ H^{(k)}\boldsymbol{v}_{n+1}^{IM} - \sum_{j=0}^{k-1}(-1)^j\frac{1}{j+1}{k \choose j+1} \boldsymbol{v}_{n-j} \right\} \label{eqn:4thBDFacc}
\end{align}
where $H^{(k)}$ is the $k^{\textrm{th}}$ harmonic number ($=\sum_{j=1}^k\frac{1}{k}$), and ${k \choose j+1}$ denotes the binomial coefficient ($=k! /[(k-j-1)!(j+1)!]$).
Note that the explicit state vectors $\boldsymbol{u}_{n+1}^{EX}$ and $\boldsymbol{v}_{n+1}^{EX}$ are obtained using a $k^{\textrm{th}}$-order extrapolation:
\begin{align}
    \boldsymbol{u}_{n+1}^{EX} = \sum_{j=0}^{k-1} (-1)^{j} {k \choose j+1}\boldsymbol{u}_{n-j}, \quad
    \boldsymbol{v}_{n+1}^{EX} = \sum_{j=0}^{k-1} (-1)^{j} {k \choose j+1}\boldsymbol{v}_{n-j}
    \label{eqn:4thEX}
\end{align}
Using Eqs.~(\ref{eqn:4thGE})-(\ref{eqn:4thEX}), one can compute the implicit state vectors $\boldsymbol{u}_{n+1}^{IM}$ and $\boldsymbol{v}_{n+1}^{IM}$.
Note that these implicit state vectors would be the solution for a conventional implicit-explicit (IMEX) $k^\textrm{th}$-order backward difference formulas (BDF$k$) scheme.
However, as mentioned before, stability of the IMEX-BDF$k$ schemes is not guaranteed for general nonlinear problems. 

To ensure unconditional stability, we compute an updated value of the SAV $\Phi_{n+1}$ by discretizing its evolution equation (Eq.~(\ref{eqn:4thSAV_ODE})) as follows:
\begin{align}
    & \frac{\Phi_{n+1}-\Phi_n}{\Delta t} = -\frac{\Phi_{n+1}}{\Psi[\boldsymbol{u}_{n+1}^{IM},\boldsymbol{v}_{n+1}^{IM}]+\psi}\Theta[\boldsymbol{u}_{n+1}^{IM},\boldsymbol{v}_{n+1}^{IM}] \label{eqn:4thBDFSAV}
\end{align}
and then update for the IMEX-BDF$k$ solution using the computed value of the SAV $\Phi_{n+1}$ as:
\begin{align}
    \boldsymbol{u}_{n+1} = \Upsilon^{(k)}_{n+1} \boldsymbol{u}_{n+1}^{IM}, \quad \boldsymbol{v}_{n+1} = \Upsilon^{(k)}_{n+1} \boldsymbol{v}_{n+1}^{IM} \label{eqn:4thupdated}
\end{align}
where
\begin{align}
    & \Upsilon^{(k)}_{n+1}=1-(1-\Xi_{n+1})^{\beta^{(k)}},
    \quad \textrm{and} \label{eqn:4thupsilon} \\
    & \Xi_{n+1}=\frac{\Phi_{n+1}}{\Psi[\boldsymbol{u}_{n+1}^{IM},\boldsymbol{v}_{n+1}^{IM}]+\psi} \label{eqn:4thxi}
\end{align}

To ensure $k^{\textrm{th}}$-order accuracy of the scheme, parameter $\beta^{(k)}$ is chosen as:
\begin{equation}
    \beta^{(k)} = \begin{cases}
        1+(k+1)/2& \textrm{for odd numbers of } k \\
        1+k/2& \textrm{for even numbers of  } k
    \end{cases}
    \label{eqn:4thbetaPara}
\end{equation}
Truncation error analysis (see section~\ref{subsec:4th23} for details) confirms $k^{\textrm{th}}$-order accuracy of the proposed IMEX-BDF$k$-SAV schemes. 
Note that the displacement and velocity state vectors ($\boldsymbol{u}_{n+1}$ and $\boldsymbol{v}_{n+1}$) are updated using the same scalar coefficient $\Upsilon^{(k)}_{n+1}$ which helps in keeping the computational overhead of the SAV approach low and also leads to an unconditionally stable scheme (see section~\ref{subsec:4th22} for details).
Finally, if needed, the acceleration $\boldsymbol{a}_{n+1}$ may be updated from the governing equation of motion. 

\subsection{Solution recovery from instability after temporal locking}

Even though the proposed IMEX-BDF$k$-SAV method, as presented above, is unconditionally stable, it suffers from a drawback that when the underlying IMEX-BDF$k$ scheme becomes unstable (i.e. $||\boldsymbol{u}_{n+1}^{IM}|| \rightarrow \infty$ and $||\boldsymbol{v}_{n+1}^{IM}|| \rightarrow \infty$), the computed solution from the IMEX-BDF$k$-SAV scheme tends to zero (i.e. $||\boldsymbol{u}_{n+1}|| \rightarrow 0$ and $||\boldsymbol{v}_{n+1}|| \rightarrow 0$).
This occurs because, when $||\boldsymbol{u}_{n+1}^{IM}|| \rightarrow \infty$ and $||\boldsymbol{v}_{n+1}^{IM}|| \rightarrow \infty$, the definition of $\Psi$ in Eq.~(\ref{eqn:4thEnergy}) leads to
\begin{align}
    \frac{||\boldsymbol{u}_{n+1}^{IM}||}{\Psi[\boldsymbol{u}_{n+1}^{IM},\boldsymbol{v}_{n+1}^{IM}]} \rightarrow 0, \quad
    \frac{||\boldsymbol{v}_{n+1}^{IM}||}{\Psi[\boldsymbol{u}_{n+1}^{IM},\boldsymbol{v}_{n+1}^{IM}]} \rightarrow 0 \label{eqn:4thUV_Phi}
\end{align}
This leads to $\Xi_{n+1}\rightarrow 0 $ and $\Upsilon_{n+1}^{(k)}\rightarrow 0$ and consequently, $||\boldsymbol{u}_{n+1}||, ||\boldsymbol{v}_{n+1}|| \rightarrow 0$.

The fact that $||\boldsymbol{u}_{n+1}||, ||\boldsymbol{v}_{n+1}|| \rightarrow 0$ by itself is not an issue, but for some problems, when $\Phi_{n+1}$ is also small, then the computed solution for the following time-steps continues to remain close to zero (i.e. $||\boldsymbol{u}_{n+m}||, ||\boldsymbol{v}_{n+m}|| \rightarrow 0$ for any value of $m$).
We call this phenomenon \emph{temporal locking} where the structure becomes almost rigid and gives an almost zero response even under the action of external forces.  
To understand why temporal locking occurs, we consider the solution at the next time-step, $t_{n+2}$, after $\boldsymbol{u}_{n+1},\boldsymbol{v}_{n+1},\Phi_{n+1}$ have all been computed.
Following Eqs.~(\ref{eqn:4thGE})-(\ref{eqn:4thBDFSAV}), $\boldsymbol{u}^{IM}_{n+2}, \boldsymbol{v}^{IM}_{n+2}, \Phi_{n+2}$ are obtained.
Next, using Eq.~(\ref{eqn:4thxi}), we can show that: 
\begin{equation}
    \Xi_{n+2} = \frac{\Phi_{n+2}}{\Psi[\boldsymbol{u}_{n+2}^{IM},\boldsymbol{v}_{n+2}^{IM}]+\psi} \le \frac{\Phi_{n+2}}{\psi} \le \frac{2\Phi_{n+1}}{\psi} \approx 0
    \label{eqn:4thXinm}
\end{equation}
In the above, we make use of the Gronwall Lemma which states that $\Phi_{n+2} \le 2 \Phi_{n+1}$ for general nonlinear systems (see Eq.~(\ref{eqn:4thPhiRelation})) and since $\Phi_{n+1} \approx 0$, thus $\Xi_{n+2} \rightarrow 0$ and $\Upsilon_{n+2}^{(k)}\rightarrow 0$.
Thus, from Eqs.~(\ref{eqn:4thupdated}), (\ref{eqn:4thupsilon}), and (\ref{eqn:4thXinm}), $||\boldsymbol{u}_{n+2}|| \approx 0$ and $||\boldsymbol{v}_{n+2}|| \approx 0$.
This phenomenon continues over the following time-steps as well making $||\boldsymbol{u}_{n+m}|| \approx 0$ and $||\boldsymbol{v}_{n+m}|| \approx 0$ for any value of $m$.

To avoid the temporal locking, for the proposed IMEX-BDF$k$-SAV schemes, we set a minimum value of $\Phi_{n+1}$ as $\varepsilon_{tol} \psi$, where $\varepsilon_{tol}$ is a small number (in this study, $\varepsilon_{tol}$ is chosen as $10^{-7}$).
Note that the minimum value of $\Phi_{n+1}$ is problem dependent and that is why it is scaled with the variable $\psi$.
Thus, when $||\boldsymbol{u}_{n+1}|| \approx 0$, $||\boldsymbol{v}_{n+1}|| \approx 0$ and $\Phi_{n+1}/\psi < \varepsilon  _{tol}$, then we initiate a recovery of the SAV, $\Phi$, by replacing the value of $\Phi_{n+1}$ with the following expression:
\begin{equation}
    \Phi_{n+1} = \Psi[\boldsymbol{u}_{n+1},\boldsymbol{v}_{n+1}] +\psi = \left( \Upsilon_{n+1}^{(k)} \right)^2 \Psi[\boldsymbol{u}^{IM}_{n+1},\boldsymbol{v}^{IM}_{n+1}] +\psi
    \label{eqn:4thRecoverning}
\end{equation}
Note that, since $\Upsilon_{n+1}^{(k)}$ and $\Psi[\boldsymbol{u}^{IM}_{n+1},\boldsymbol{v}^{IM}_{n+1}]$ is already computed at each time step, the computational cost for this recovery is negligible. 
Further, since the updated SAV $\Phi_{n+1}$ is always bounded (because $\boldsymbol{u}_{n+1}$ and $\boldsymbol{v}_{n+1}$ are bounded), the stability analysis presented in section~\ref{subsec:4th22} remains valid.
Recovering the value of $\Phi_{n+1}$, as shown above, simply amounts to a modification of the state of the system for the following time-steps. 
One can imagine this recovery as dividing the time interval of simulation into two segments, one before and one after the recovery.
The part of the simulation after the recovery can then be simply thought of as a new simulation with an updated initial condition.
Lastly, the local truncation error (LTE) analysis presented in section~\ref{subsec:4th23} also remains valid because we conduct the LTE analysis only going from $t_{n}$ to $t_{n+1}$.

\subsection{Starting procedure and algorithm}

A special starting procedure is needed for the proposed IMEX-BDF$k$-SAV schemes with $k\ge2$ to achieve $k^{\textrm{th}}$-order accuracy.
At the initial time $t_0$, the state vectors for the first $(k-1)$ steps must be computed with a method that is at least $(k-1)^{\textrm{th}}$-order accurate. 
In this study, we use the $(k-1)^{\textrm{th}}$-order Runge-Kutta method for the first $(k-1)$ steps.

We summarize the procedure of the proposed IMEX-BDF$k$-SAV schemes for structural dynamics in Algorithm \ref{alg:4thProcedure}.

\begin{algorithm}[!htbp]
    \caption{\label{alg:4thProcedure} Time stepping for the IMEX-BDF$k$-SAV schemes.}
    \centering 
    \begin{tabular}{l}
        1. Initial calculation    \\
        \quad a) Construct the mass, stiffness, and damping matrices $\boldsymbol{M}$, $\boldsymbol{K}$, and $\boldsymbol{C}$  \\
        \quad b) Initialize displacement and velocity vectors, and SAV: \\
        \qquad $\boldsymbol{u}_0=\boldsymbol{u}^{EX}_0=\boldsymbol{u}(0), \quad \boldsymbol{v}_0=\boldsymbol{v}_0^{EX}=\dot{\boldsymbol{u}}(0), \quad \Phi_0 = \Psi[\boldsymbol{u}_0, \boldsymbol{v}_0] + \psi$ \\
        2. Starting procedure only when $k\ge2$ ($0 \le {n} < {k-1}$)  \\
        \quad a) Compute $\boldsymbol{u}_{n+1}$ and $\boldsymbol{v}_{n+1}$ \\
        \qquad using $(k-1)^\textrm{th}$-order time integration schemes (e.g. Runge-Kutta method) \\
        \quad b) If $n=k-2$, then compute the SAV: \\
        \qquad $\Phi_{n+1} = \Psi[\boldsymbol{u}_{n+1}, \boldsymbol{v}_{n+1}] + \psi$ \\
        3. For each cycle of the time-step (${n} \ge {k-1}$) \\
        \quad a) Compute the implicit state vectors $\boldsymbol{u}_{n+1}^{IM}$ and $\boldsymbol{v}_{n+1}^{IM}$ from Eqs.~(\ref{eqn:4thGE})-(\ref{eqn:4thEX}). \\
        \quad b) Compute the SAV $\Phi_{n+1}$ from Eq.~(\ref{eqn:4thBDFSAV}). \\
        \quad c) Update the state vectors $\boldsymbol{u}_{n+1}$ and $\boldsymbol{v}_{n+1}$ from Eqs.~(\ref{eqn:4thupdated})-(\ref{eqn:4thxi}). \\
        \quad d) If $\Phi_{n+1} <  \varepsilon_{tol} \psi$, then update the SAV $\Phi_{n+1}$ from Eq.~(\ref{eqn:4thRecoverning})
    \end{tabular}
\end{algorithm}

\subsection{Stability analysis}
\label{subsec:4th22}
In this section, we prove that the proposed IMEX-BDF$k$-SAV schemes are unconditionally energy stable, first for linear structural dynamics in Theorem~\ref{theorem:4th1} and subsequently for nonlinear structural dynamics in Theorem~\ref{theorem:4th2}. 
\begin{theorem}
    \label{theorem:4th1}
  For linear structural dynamics, in the absence of external forces $\boldsymbol{f}^{ext}$, the IMEX-BDF$k$-SAV schemes for any $k \ge 1$ are unconditionally energy stable such that
  \begin{equation}
      \Phi_{n+1} \le \Phi_n \label{eqn:4thLTheorem1}
  \end{equation}
  Furthermore, there exists $U^{(k)}>0$ such that
  \begin{equation}
      \Psi[\boldsymbol{u}_n,\boldsymbol{v}_n] \le \left(U^{(k)}\right)^2 \quad  \forall \, n \in [0, T/\Delta t] \label{eqn:4thLTheorem2}
  \end{equation}
\end{theorem}

\begin{proof}
  For linear structural dynamics in the absence of external forces, the definition of $\Theta$ in Eq.~(\ref{eqn:4thSAV_theta}) ensures that $\Theta[\boldsymbol{u}_{n+1}^{IM},\boldsymbol{v}_{n+1}^{IM}] \ge 0$ at all times. 
  Further, from the definition of the SAV $\Phi$ in Eq.~(\ref{eqn:4thSAV_definition}), $\Psi[\boldsymbol{u}_{n+1}^{IM},\boldsymbol{v}_{n+1}^{IM}]+\psi > 0$.
  Thus, given $\Phi_n\ge0$, Eq.~(\ref{eqn:4thBDFSAV}) may be rewritten as:
  \begin{equation}
      \Phi_{n+1} = \dfrac{\Phi_n}{ 1+\Delta t\dfrac{\Theta[\boldsymbol{u}_{n+1}^{IM},\boldsymbol{v}_{n+1}^{IM}]}{\Psi[\boldsymbol{u}_{n+1}^{IM},\boldsymbol{v}_{n+1}^{IM}]+\psi} } 
      \label{eqn:4thTimestepping2}
  \end{equation}  
  ensuring that $\Phi_{n+1} \ge 0$.
  Next, multiplying Eq.~(\ref{eqn:4thBDFSAV}) by $\Delta t$ and noting that the factor on the right-hand side of Eq.~(\ref{eqn:4thBDFSAV}) is simply $\Xi_{n+1}$, as defined in Eq.~(\ref{eqn:4thxi}), we conclude that:
  \begin{equation}
      \Phi_{n+1}-\Phi_n = -\Delta t\Xi_{n+1}\Theta[\boldsymbol{u}_{n+1}^{IM},\boldsymbol{v}_{n+1}^{IM}] \le 0
  \end{equation}
  because, from Eq.~(\ref{eqn:4thxi}), $\Xi_{n+1} \ge 0$.
  This proves Eq.~(\ref{eqn:4thLTheorem1}) of this theorem.
  
  To prove Eq.~(\ref{eqn:4thLTheorem2}), let $U:=\Phi_0$.
  Then, Eq.~(\ref{eqn:4thLTheorem1}) ensures that $\Phi_{n+1} \le U$.
  From Eq.~(\ref{eqn:4thxi}),
  \begin{equation}
      \Xi_{n+1} = \dfrac{\Phi_{n+1}}{\Psi[\boldsymbol{u}_{n+1}^{IM},\boldsymbol{v}_{n+1}^{IM}]+\psi} \le \dfrac{U}{\Psi[\boldsymbol{u}_{n+1}^{IM},\boldsymbol{v}_{n+1}^{IM}]+\psi}
      \label{eqn:4thxi2}
  \end{equation}  
  From Eq.~(\ref{eqn:4thupsilon}), the scalar coefficient $\Upsilon^{(k)}_{n+1}$ can be represented as $\Upsilon^{(k)}_{n+1}=\Xi_{n+1} P_{\beta^{(k)}-1}(\Xi_{n+1})$, where $P_{\beta^{(k)}-1}$ is a polynomial of degree $\beta^{(k)}-1$.
  Note that the absolute value of the polynomial $P_{\beta^{(k)}-1}$ is bounded. 
  Further, since $U>0$ and $\psi>0$, there exists $U^{(k)}>0$ such that 
  \begin{equation}
      |P_{\beta^{(k)}-1}(\Xi^{n+1})| \le \dfrac{2\sqrt{\psi}}{U}U^{(k)}
      \label{eqn:4thPfunction}
  \end{equation} 
  From Eqs.~(\ref{eqn:4thxi2}) and (\ref{eqn:4thPfunction}), 
  \begin{equation}
      |\Upsilon^{(k)}_{n+1}|=|\Xi_{n+1} P_{\beta^{(k)}-1}(\Xi^{n+1})|\le \dfrac{2\sqrt{\psi}U^{(k)}}{\Psi[\boldsymbol{u}_{n+1}^{IM},\boldsymbol{v}_{n+1}^{IM}]+\psi}
      \label{eqn:4thProof1_eta}
  \end{equation}  
  From Eqs.~(\ref{eqn:4thEnergy}) and (\ref{eqn:4thupdated}), the pseudo-energy $\Psi$ at $t_{n+1}$ may be written as:
  \begin{equation}
       \Psi[\boldsymbol{u}_{n+1},\boldsymbol{v}_{n+1}] = \left(\Upsilon^{(k)}_{n+1}\right)^2 \Psi[\boldsymbol{u}_{n+1}^{IM},\boldsymbol{v}_{n+1}^{IM}] 
       \label{eqn:4thPsinextstep}
  \end{equation}
  Substituting Eq.~(\ref{eqn:4thProof1_eta}) into Eq.~(\ref{eqn:4thPsinextstep}), the following inequality is obtained:
  \begin{equation}
      \begin{array}{ll}
           \Psi[\boldsymbol{u}_{n+1},\boldsymbol{v}_{n+1}] 
           &\le \left(\dfrac{2\sqrt{\psi}U^{(k)}}{\Psi[\boldsymbol{u}_{n+1}^{IM},\boldsymbol{v}_{n+1}^{IM}]+\psi}\right)^2 \Psi[\boldsymbol{u}_{n+1}^{IM},\boldsymbol{v}_{n+1}^{IM}] \\
           &\le \left(U^{(k)}\right)^2
      \end{array}        
  \end{equation}
    Note that we use the following identity in the last step above:
    \begin{equation}
      \begin{array}{ll}
            &\dfrac{4 \psi \Psi[\boldsymbol{u}_{n+1}^{IM},\boldsymbol{v}_{n+1}^{IM}] }{\left(\Psi[\boldsymbol{u}_{n+1}^{IM},\boldsymbol{v}_{n+1}^{IM}]+\psi\right)^2} \le 1 \\
      \end{array} 
      \label{eqn:4thIdentity1}
    \end{equation}
    which in turn is obtained from the following statement:
    \begin{equation}
        {\left(\Psi[\boldsymbol{u}_{n+1}^{IM},\boldsymbol{v}_{n+1}^{IM}]-\psi\right)^2} = {\left(\Psi[\boldsymbol{u}_{n+1}^{IM},\boldsymbol{v}_{n+1}^{IM}]+\psi\right)^2} - 4 \psi \Psi[\boldsymbol{u}_{n+1}^{IM},\boldsymbol{v}_{n+1}^{IM}] \ge 0
    \end{equation}
  The proof is complete.
\end{proof}

\begin{theorem}
\label{theorem:4th2}
  For nonlinear structural dynamics, in the absence of external forces $\boldsymbol{f}^{ext}$,
  given $\psi \ge 2\Delta t^2||\boldsymbol{L}^{-1}\boldsymbol{f}^{NL}||^2$ for all $t$, where $\boldsymbol{M}=\boldsymbol{L}\boldsymbol{L}^T$ (Cholesky decomposition), there exists $U^{(k)}>0$ such that
  \begin{equation}
      \Psi[\boldsymbol{u}_n,\boldsymbol{v}_n] \le \left(U^{(k)}\right)^2 \quad  \forall \, n \in [0, T/\Delta t] \label{eqn:4thNLTheorem2}
  \end{equation}
\end{theorem}

\begin{proof}
    For nonlinear structural dynamics in the absence of external forces, Eq.~(\ref{eqn:4thSAV_theta}) states:
    \begin{equation}
        \Theta [\boldsymbol{u}_{n+1}^{IM}, \boldsymbol{v}_{n+1}^{IM}] = \boldsymbol{v}_{n+1}^{IM^T} \left( \boldsymbol{C}\boldsymbol{v}_{n+1}^{IM} + \boldsymbol{f}^{NL}(\boldsymbol{u}_{n+1}^{IM}, \boldsymbol{v}_{n+1}^{IM}) \right) 
    \end{equation}
    Thus,
    \begin{equation}
        \begin{array}{ll}
        \left| \Theta[\boldsymbol{u}_{n+1}^{IM},\boldsymbol{v}_{n+1}^{IM}] \right| 
        &\le \left| \boldsymbol{v}_{n+1}^{IM^T} \boldsymbol{f}^{NL}(\boldsymbol{u}_{n+1}^{IM},\boldsymbol{v}_{n+1}^{IM}) \right| \\
        &\le \left| \boldsymbol{v}_{n+1}^{IM^T} \boldsymbol{L} \boldsymbol{L}^{-1} \boldsymbol{f}^{NL}(\boldsymbol{u}_{n+1}^{IM},\boldsymbol{v}_{n+1}^{IM}) \right| \\
        &\le ||\boldsymbol{L}^T\boldsymbol{v}_{n+1}^{IM} || 
            \, ||\boldsymbol{L}^{-1}\boldsymbol{f}^{NL}(\boldsymbol{u}_{n+1}^{IM},\boldsymbol{v}_{n+1}^{IM})|| 
        \end{array} 
        \label{eqn:4thNLforceEnum}
    \end{equation}
    From Eq.~(\ref{eqn:4thEnergy}) and choosing $\psi \ge 2\Delta t^2||\boldsymbol{L}^{-1}\boldsymbol{f}^{NL}||^2$,
    \begin{equation}
        \begin{array}{ll}
        \Psi[\boldsymbol{u}_{n+1}^{IM},\boldsymbol{v}_{n+1}^{IM}]+\psi
         &\ge \frac{1}{2}\boldsymbol{v}_{n+1}^{IM^T}\boldsymbol{M}\boldsymbol{v}_{n+1}^{IM}+ \psi \\
         &\ge \frac{1}{2}\boldsymbol{v}_{n+1}^{IM^T}\boldsymbol{M}\boldsymbol{v}_{n+1}^{IM}+2 \Delta t^2||\boldsymbol{L}^{-1}\boldsymbol{f}^{NL}||^2 \\
         &\ge \frac{1}{2}\boldsymbol{v}_{n+1}^{IM^T}\boldsymbol{L}\boldsymbol{L}^T\boldsymbol{v}_{n+1}^{IM}+2 \Delta t^2||\boldsymbol{L}^{-1}\boldsymbol{f}^{NL}||^2
        \label{eqn:4thNLforceEden}
        \end{array} 
    \end{equation}     
   Dividing inequality (\ref{eqn:4thNLforceEnum}) by inequality (\ref{eqn:4thNLforceEden}), we obtain:
    \begin{equation} \resizebox{0.91\linewidth}{!}{$
        \begin{array}{ll}
            \left| \dfrac{\Theta[\boldsymbol{u}_{n+1}^{IM},\boldsymbol{v}_{n+1}^{IM}]}  {\Psi[\boldsymbol{u}_{n+1}^{IM},\boldsymbol{v}_{n+1}^{IM}]+\psi} \right|
            &\le \dfrac{||\boldsymbol{L}^T\boldsymbol{v}_{n+1}^{IM} || 
            \, ||\boldsymbol{L}^{-1}\boldsymbol{f}^{NL}(\boldsymbol{u}_{n+1}^{IM},\boldsymbol{v}_{n+1}^{IM})|| } {\frac{1}{2} || \boldsymbol{L}^T\boldsymbol{v}_{n+1}^{IM} ||^2 +2\Delta t^2 ||\boldsymbol{L}^{-1}\boldsymbol{f}^{NL}(\boldsymbol{u}_{n+1}^{IM},\boldsymbol{v}_{n+1}^{IM})||^2} \\
            &\le \dfrac{1}{2\Delta t}
        \end{array} $}
        \label{eqn:4thNLforceE}
    \end{equation}    
    where, in the last step above, we use the fact:
    \begin{equation}
            {\left(|| \boldsymbol{L}^T\boldsymbol{v}_{n+1}^{IM} || - 2 \Delta t  ||\boldsymbol{L}^{-1}\boldsymbol{f}^{NL}(\boldsymbol{u}_{n+1}^{IM},\boldsymbol{v}_{n+1}^{IM})||\right)^2} \ge 0 \\
      \label{eqn:4thIdentity2}
    \end{equation}    
    Thus, given $\Phi_n>0$, Eq.~(\ref{eqn:4thTimestepping2}) and inequality~(\ref{eqn:4thNLforceE}) lead to:
    \begin{equation}
        \Phi_{n+1}=\Phi_n\left( 1 + \Delta t \dfrac{\Theta[\boldsymbol{u}_{n+1}^{IM},\boldsymbol{v}_{n+1}^{IM}]}{\Psi[\boldsymbol{u}_{n+1}^{IM},\boldsymbol{v}_{n+1}^{IM}]+\psi}  \right)^{-1} > 0
    \end{equation}    
    Then, summing up Eq.~(\ref{eqn:4thBDFSAV}) from $n=0$ to $m$, 
    \begin{equation}
        \begin{array}{ll}
             \Phi_{m+1}
             &= \Phi_0 - \Delta t \displaystyle\sum_{j=0}^{m}\dfrac{\Phi_{j+1}}  {\Psi[\boldsymbol{u}_{j+1}^{IM},\boldsymbol{v}_{j+1}^{IM}]+\psi} \Theta[\boldsymbol{u}_{j+1}^{IM},\boldsymbol{v}_{j+1}^{IM}] \\
             &\le \Phi_0 + \Delta t \displaystyle\sum_{j=0}^{m} \left| \dfrac{\Theta[\boldsymbol{u}_{j+1}^{IM},\boldsymbol{v}_{j+1}^{IM}]}  {\Psi[\boldsymbol{u}_{j+1}^{IM},\boldsymbol{v}_{j+1}^{IM}]+\psi} \right| \Phi_{j+1} \\
            &\le \Phi_0 + \dfrac{1}{2}\displaystyle\sum_{j=0}^{m}\Phi_{j+1}
        \end{array} 
    \end{equation}
    Applying the discrete Gronwall lemma (see Ref.~\cite{clark1987short}), it can be proved that
    \begin{equation}
        \Phi_{m} \le 2^m \Phi_0 \quad \forall \quad m \in [0,T/\Delta t]
        \label{eqn:4thPhiRelation}
    \end{equation}  
    which implies that
    \begin{equation}
      \Xi_{n+1} = \dfrac{\Phi_{n+1}}{\Psi[\boldsymbol{u}_{n+1}^{IM},\boldsymbol{v}_{n+1}^{IM}]+\psi} \le \dfrac{2^{n+1}\Phi_0}{\Psi[\boldsymbol{u}_{n+1}^{IM},\boldsymbol{v}_{n+1}^{IM}]+\psi}
      \label{eqn:4thZetaNL}
    \end{equation}  
    Recall $\Upsilon^{(k)}_{n+1}=\Xi_{n+1} P_{\beta^{(k)}-1}(\Xi_{n+1})$.
    Since $\Psi_0>0$ and $\psi>0$, there exists $U^{(k)}>0$ such that 
      \begin{equation}
          |P_{\beta^{(k)}-1}(\Xi^{n+1})| \le \dfrac{2\sqrt{\psi}}{2^{n+1}\Phi_0}U^{(k)}
          \label{eqn:4thPfunctionNL}
      \end{equation} 
    From Eqs.~(\ref{eqn:4thZetaNL}) and (\ref{eqn:4thPfunctionNL}), 
    \begin{equation}
      |\Upsilon^{(k)}_{n+1}|=|\Xi_{n+1} P_{\beta^{(k)}-1}(\Xi^{n+1})|\le \dfrac{2\sqrt{\psi}U^{(k)}}{\Psi[\boldsymbol{u}_{n+1}^{IM},\boldsymbol{v}_{n+1}^{IM}]+\psi}
      \label{eqn:4thProof2_eta}
    \end{equation}  
    From Eq.~(\ref{eqn:4thupdated}), Eq.~(\ref{eqn:4thProof2_eta}) implies
    \begin{equation}
      \begin{array}{ll}
           \Psi[\boldsymbol{u}_{n+1},\boldsymbol{v}_{n+1}] &= \left(\Upsilon^{(k)}_{n+1}\right)^2 \Psi[\boldsymbol{u}_{n+1}^{IM},\boldsymbol{v}_{n+1}^{IM}]  \\
           &\le \left(\dfrac{2\sqrt{\psi}U^{(k)}}{\Psi[\boldsymbol{u}_{n+1}^{IM},\boldsymbol{v}_{n+1}^{IM}]+\psi}\right)^2 \Psi[\boldsymbol{u}_{n+1}^{IM},\boldsymbol{v}_{n+1}^{IM}] \\
           &\le \left( U^{(k)} \right)^2
      \end{array}        
    \end{equation}
    where, in the final step above, Eq.~(\ref{eqn:4thIdentity1}) is used.
    The proof is complete.
\end{proof}

\subsection{Local truncation error analysis}
\label{subsec:4th23}
In this section, we analyze the local truncation error (LTE) of the proposed IMEX-BDF$k$-SAV schemes in linear structural dynamics.
Note that for linear structural dynamics, the governing second-order ODEs are obtained from Eq.~(\ref{eqn:4thGDE})  where:
\begin{align}
    \boldsymbol{f}^{int} = \boldsymbol{C}\dot{\boldsymbol{u}}(t)  + \boldsymbol{K}\boldsymbol{u}(t) 
    \label{eqn:4thEOMLinHomo}
\end{align}
In this case, the definition of $\Theta$ in Eq.~(\ref{eqn:4thSAV_theta}) simplifies to:
 \begin{align}
     \Theta=\dot{\boldsymbol{u}}^T (\boldsymbol{C}\dot{\boldsymbol{u}} -  \boldsymbol{f}^{ext})
     \label{eqn:4thSAV_thetaL}
 \end{align}
As the solutions for displacement and velocity vectors, $\boldsymbol{u}_{n+1}$ and $\boldsymbol{v}_{n+1}$ respectively, are computed by advancing from $t_n$ to $t_{n+1}$, the LTEs in these quantities can be obtained as:
\begin{align} 
    \boldsymbol{\tau}_{\boldsymbol{u}} = \boldsymbol{u}_{n+1}-\boldsymbol{u}(t_{n+1}), \quad
   \boldsymbol{\tau}_{\boldsymbol{v}} = \boldsymbol{v}_{n+1}-\dot{\boldsymbol{u}}(t_{n+1})
\end{align}
where ${\boldsymbol{u}}(t_{n+1})$ and $\dot{\boldsymbol{u}}(t_{n+1})$ denote the exact displacement and velocity vectors.
The exact solution can be obtained from a Taylor series expansion about $t_n$ as follows:
\begin{equation}
    \boldsymbol{u}(t_{n+1}) = \boldsymbol{u}(t_n) + \sum_{k=1}^{\infty}\frac{\Delta t^k}{k!}\frac{d^k\boldsymbol{u}(t_n)}{dt^k}, \quad
    \dot{\boldsymbol{u}}(t_{n+1}) = \dot{\boldsymbol{u}}(t_{n}) + \sum_{k=1}^{\infty}\frac{\Delta t^k}{k!}\frac{d^{k}\dot{\boldsymbol{u}}(t_n)}{dt^{k}}
    \label{eqn:4thTaylor}
\end{equation}
Note that, for LTE, we assume that $\boldsymbol{u}(t_n)=\boldsymbol{u}_n$ and $\dot{\boldsymbol{u}}(t_n)=\boldsymbol{v}_n$ are exact displacement and velocity vectors at $t_n$.
In this section, the external force vector is assumed to be continuous in time, then $\boldsymbol{f}^{ext}(t_{n+1})$ can be expressed as follows:
\begin{equation}
    \boldsymbol{f}^{ext}(t_{n+1}) = \boldsymbol{f}^{ext}(t_n) + \sum_{k=1}^{\infty}\frac{\Delta t^k}{k!}\frac{d^k\boldsymbol{f}^{ext}(t_n)}{dt^k}
    \label{eqn:4thTaylorForce}
\end{equation}

Using Eqs.~(\ref{eqn:4thTaylor}) and (\ref{eqn:4thTaylorForce}), it can be shown that the conventional BDF$k$ schemes are $k^{\textrm{th}}$-order accurate (see Ref~ \cite{iserles2009first} for details): 
\begin{equation}
    \boldsymbol{u}^{IM}_{n+1} - \boldsymbol{u}(t_{n+1}) \propto  C_{u} \Delta t^{k+1} + O(\Delta t^{k+2}), \quad \boldsymbol{v}^{IM}_{n+1} - \dot{\boldsymbol{u}}(t_{n+1}) \propto C_{v} \Delta t^{k+1} + O(\Delta t^{k+2})
    \label{eqn:4thBDFLTE}
\end{equation}
where $C_{u}$ and $C_{v}$ are constant coefficients of the leading-order terms in the LTEs in displacement and velocity vectors, respectively.
To compute the LTEs for the proposed IMEX-BDF$k$-SAV schemes, $\Xi_{n+1}$ is obtained by substituting Eq.~(\ref{eqn:4thTimestepping2}) into Eq.~(\ref{eqn:4thxi}) and using the definition of $\Theta$ in Eq.~(\ref{eqn:4thSAV_thetaL}):
\begin{equation}
    \begin{array}{ll}
        \Xi_{n+1}&=\dfrac{\Phi_n}{ \Psi[\boldsymbol{u}_{n+1}^{IM},\boldsymbol{v}_{n+1}^{IM}]+\psi + \Delta t \Theta[\boldsymbol{u}_{n+1}^{IM},\boldsymbol{v}_{n+1}^{IM}]} \\
        &=\dfrac{ \boldsymbol{v}_n^T\boldsymbol{M}\boldsymbol{v}_n + \boldsymbol{u}_n^T\boldsymbol{K}\boldsymbol{u}_n + 2\psi}{ \boldsymbol{v}^{IM^T}_{n+1}\boldsymbol{M}\boldsymbol{v}^{IM}_{n+1}+ \boldsymbol{u}^{IM^T}_{n+1}\boldsymbol{K}\boldsymbol{u}^{IM}_{n+1} + 2\psi + 2\Delta t \boldsymbol{v}_{n+1}^{IM^T} (\boldsymbol{C} \boldsymbol{v}_{n+1}^{IM} - \boldsymbol{f}^{ext}(t_{n+1})) }
    \end{array}    
    \label{eqn:4thXi1}
\end{equation}
Using Eq.~(\ref{eqn:4thBDFLTE}), Eq.~(\ref{eqn:4thXi1}) is rewritten as:
\begin{equation}
    \Xi_{n+1} \propto \dfrac{ \boldsymbol{v}_n^T\boldsymbol{M}\boldsymbol{v}_n + \boldsymbol{u}_n^T\boldsymbol{K}\boldsymbol{u}_n + 2\psi}{ 
    \left(\begin{array}{l}
        \dot{\boldsymbol{u}}(t_{n+1})^T\boldsymbol{M}\dot{\boldsymbol{u}}(t_{n+1}) + \boldsymbol{u}(t_{n+1})^T\boldsymbol{K}\boldsymbol{u}(t_{n+1}) + 2\psi \\
         + 2\Delta t\dot{\boldsymbol{u}}(t_{n+1})^T (\boldsymbol{C} \dot{\boldsymbol{u}}(t_{n+1}) - \boldsymbol{f}^{ext}(t_{n+1})) + O(\Delta t^{k+1})
    \end{array}\right)
    } \\
    \label{eqn:4thXi2}
\end{equation}
Using Eqs.~(\ref{eqn:4thTaylor}) and (\ref{eqn:4thTaylorForce}), the following terms in denominator of Eq.~(\ref{eqn:4thXi2}) can be written as:
\begin{align}
    \begin{array}{l}
    \dot{\boldsymbol{u}}(t_{n+1})^T\boldsymbol{M}\dot{\boldsymbol{u}}(t_{n+1}) + \boldsymbol{u}(t_{n+1})^T\boldsymbol{K}\boldsymbol{u}(t_{n+1}) + 2\Delta t \dot{\boldsymbol{u}}(t_{n+1})^T (\boldsymbol{C} \dot{\boldsymbol{u}}(t_{n+1})-\boldsymbol{f}^{ext}(t_{n+1})) \\
    \qquad \propto \boldsymbol{v}_n^T\boldsymbol{M}\boldsymbol{v}_n + \boldsymbol{u}_n^T\boldsymbol{K}\boldsymbol{u}_n + 2\Delta t \boldsymbol{v}_n^T ( \boldsymbol{M}\boldsymbol{a}_n + \boldsymbol{C} \boldsymbol{v}_n + \boldsymbol{K} \boldsymbol{u}_n - \boldsymbol{f}^{ext}(t_{n}))  + C_E \Delta t^2 + O(\Delta t^3)
    \\
    \qquad \propto \boldsymbol{v}_n^T\boldsymbol{M}\boldsymbol{v}_n + \boldsymbol{u}_n^T\boldsymbol{K}\boldsymbol{u}_n + C_E \Delta t^2 + O(\Delta t^3)
    \end{array}  \label{eqn:4thTaylorEnergy}
\end{align}
where $C_E$ is the coefficient of the leading-order term and $\boldsymbol{M}\boldsymbol{a}_n + \boldsymbol{C} \boldsymbol{v}_n + \boldsymbol{K} \boldsymbol{u}_n = \boldsymbol{f}^{ext}(t_{n})$.
Thus, Eq.~(\ref{eqn:4thXi2}) can be simplified to:
\begin{equation}
    \begin{array}{ll}
        \Xi_{n+1} &\propto \dfrac{ \boldsymbol{v}_n^T\boldsymbol{M}\boldsymbol{v}_n + \boldsymbol{u}_n^T\boldsymbol{K}\boldsymbol{u}_n + 2\psi}{ \boldsymbol{v}_n^T\boldsymbol{M}\boldsymbol{v}_n + \boldsymbol{u}_n^T\boldsymbol{K}\boldsymbol{u}_n + 2\psi + C_E \Delta t^{2} + O(\Delta t^{3})} \\
        & \propto 1 + \dfrac{C_E}{2\Phi_n} \Delta t^2 + O(\Delta t^3)
    \end{array}    
    \label{eqn:4thXiFinal}
\end{equation}
where $\Phi_n = \frac{1}{2}\left(\boldsymbol{v}_n^T\boldsymbol{M}\boldsymbol{v}_n + \boldsymbol{u}_n^T\boldsymbol{K}\boldsymbol{u}_n \right) + \psi$.
Since $\Upsilon^{(k)}_{n+1}=1-(1-\Xi_{n+1})^{\beta^{(k)}}$, Eq.~(\ref{eqn:4thXiFinal}) implies:
\begin{equation}
    \begin{array}{ll}
    \Upsilon^{(k)}_{n+1} &\propto 1 - \left(- \dfrac{C_E}{2\Phi_0}\right)^{\beta^{(k)}} \Delta t^{2\beta^{(k)}}
   \end{array}
   \label{eqn:4thUpsilonLTE}
\end{equation}
Finally, from Eqs.~(\ref{eqn:4thupdated}), (\ref{eqn:4thBDFLTE}), and (\ref{eqn:4thUpsilonLTE}), LTEs in displacement and velocity vectors are obtained as:
\begin{align} 
    \begin{array}{ll}
        \boldsymbol{\tau}_{\boldsymbol{u}} & = \boldsymbol{u}_{n+1}-\boldsymbol{u}(t_{n+1}) = \Upsilon^{(k)}_{n+1}\boldsymbol{u}^{IM}_{n+1}-\boldsymbol{u}(t_{n+1}) \\
        & \propto  C_u \Delta t^{k+1} - \left(- \dfrac{C_E}{2\Phi_0}\right)^{\beta^{(k)}}  \boldsymbol{u}(t_{n+1}) \Delta t^{2\beta^{(k)}}
    \end{array}    \\
    \begin{array}{ll}
       \boldsymbol{\tau}_{\boldsymbol{v}} & = \boldsymbol{v}_{n+1}-\dot{\boldsymbol{u}}(t_{n+1}) = \Upsilon^{(k)}_{n+1}\boldsymbol{v}^{IM}_{n+1}-\dot{\boldsymbol{u}}(t_{n+1}) \\
       & \propto  C_v \Delta t^{k+1} - \left(- \dfrac{C_E}{2\Phi_0}\right)^{\beta^{(k)}} \dot{\boldsymbol{u}}(t_{n+1}) \Delta t^{2\beta^{(k)}}
   \end{array}
\end{align}
Since $2\beta^{(k)} > k+1$, the proposed IMEX-BDF$k$-SAV schemes leads to $k^{\textrm{th}}$-order accuracy.
Thus, the LTEs for the underlying BDF$k$ schemes are preserved:
\begin{equation}
    \boldsymbol{\tau}_{\boldsymbol{u}} \propto  C_{u} \Delta t^{k+1} + O(\Delta t^{k+2}), \quad
    \boldsymbol{\tau}_{\boldsymbol{v}} \propto C_{v} \Delta t^{k+1} + O(\Delta t^{k+2})
\end{equation}
It is important to note that with an increase in $\beta^{(k)}$, the energy bound $\left(U^{(k)}\right)^2$ also increases (see Eqs.~(\ref{eqn:4thPfunction}) and (\ref{eqn:4thPfunctionNL})).
Therefore, as given in Eq.~(\ref{eqn:4thbetaPara}), we recommend selecting the smallest $\beta^{(k)}$ such that $2\beta^{(k)} > k+1$.
This is in contrast to the value $\beta^{(k)} = k+1$ suggested by Huang and Shen \cite{huang2022new} for general dissipative systems because, in their derivation, $\Xi_{n+1}$ is a first-order approximation to 1. 
As shown in Eq.~(\ref{eqn:4thXiFinal}), $\Xi_{n+1}$ for the proposed IMEX-BDF$k$-SAV schemes is a second-order approximation to 1 for any value of $k$.
Thus, for structural dynamics, we can use a value of $\beta^{(k)}$ that is about half of the value used for general dissipative systems.
This leads to a lower energy bound $\left( U^{(k)} \right)^2$ and consequently better performance of the proposed IMEX-BDF$k$-SAV schemes.

To verify the theoretical LTEs found above, we compute the LTEs for a typical modal equation of motion with linear internal forces:
\begin{align}
    & \ddot{u}(t) + 2\zeta \omega_0 \dot{u}(t) + \omega_0^2 u(t) = f^{ext}(t) \label{eqn:4thSDOFEqn} \\
    & u(0) = u_0, \quad \dot{u}(0) = v_0
\end{align}
where $\omega_0$, $\zeta$, $u$, and $f^{ext}$ are the natural frequency, damping ratio, modal displacement, and modal external force, respectively.
We consider a sinusoidal external forcing function: $f^{ext}(t)=p_0 \sin{\omega_f t}$, where $p_0$ and $\omega_f$ are the amplitude and excitation frequency, respectively.
Note that the exact solution to this problem is given as: 
\begin{equation}
    \begin{array}{ll}
         u(t)=&e^{-\zeta \omega_0 t} (A_f \cos(\omega_d t) + B_f \sin(\omega_d t))  \\
         &+ \dfrac{p_0}{\eta} \left( (\omega_0^2-\omega_f^2) \sin(\omega_f t) - 2 \omega_0 \omega_f \zeta \cos(\omega_f t)\right)
    \end{array}
\end{equation}
where
\begin{align}
    &A_f = u_0 + 2 \frac{p_0}{\eta} \zeta \omega_0 \omega_f, \quad
    B_f = \frac{v_0+\zeta \omega_0 u_0}{\omega_d} + \frac{p_0}{\eta}\frac{\omega_f}{\omega_d}(\omega_f^2-\omega_0^2+2 \zeta^2 \omega_0^2), \\
    &\eta = (\omega_0^2-\omega_f^2)^2+(2 \zeta \omega_0 \omega_f)^2
\end{align}
Using the exact solution above and the analysis in section~\ref{subsec:4th23}, the LTEs in displacement and velocity for the modal equation of motion are computed and reported in Appendix {\ref{ap:4thLTEforced}}.
Note that the LTEs in displacements and velocities are found to be of the order of $\Delta t^{k+1}$ for $1 \le k \le 5$.
Thus, the global errors in displacements and velocities can be expected to be of the order of $\Delta t^{k}$ for $1 \le k \le 5$.

\subsection{Convergence tests}
\label{subsec:4th24}

To verify the theoretical findings in section~\ref{subsec:4th23}, we compute global errors for the proposed IMEX-BDF$k$-SAV schemes when solving linear and nonlinear problems.
To quantify global errors, we define the following error norm:
\begin{align}
    \begin{array}{c}
        \mathcal{E}(\boldsymbol{z}) = \sqrt{ \dfrac{\sum_{n=0}^{T/\Delta t}{\left( z_n-\tilde{z}_n\right)^2}} {\sum_{n=0}^{T/\Delta t}{\left( \tilde{z}_n \right)^2}}  }
    \end{array}
    \label{eqn:4thErrorNorm}
\end{align}
where $\boldsymbol{z}$ may be displacement, velocity or acceleration and $\tilde{\boldsymbol{z}}$ denotes the exact solution or a reference solution, when an exact solution is not available.

\subsubsection{Linear structural dynamics}
We consider a damped single-degree-of-freedom (SDOF) problem, as given in Eq.~({\ref{eqn:4thSDOFEqn}}), with damping ratio $\zeta=0.2$, natural frequency $\omega_0=2\pi$, and the corresponding natural period $T_0=1$.
The duration of the simulation is taken to be $[0,T]$, where $T=10$.

We apply zero initial displacement and velocity and measure global error in the forced response under a load given as $f^{ext}(t)= \sin{2\omega_0 t}$.
For the proposed IMEX-BDF$k$-SAV schemes, we need an estimate of the maximum energy $\Psi$ for this problem, which can be obtained by assuming that the maximum displacement will be twice the static response i.e. $u_{max} \approx 2 p_0/\omega_0^2$.
Thus, the maximum energy stored in the underlying undamped system would be $\Psi_{max} \approx \frac{1}{2} {\omega_0}^2 u_{max}^2 \approx 0.05$.
Following the guideline for the choice of SAV in Eq.~(\ref{eqn:4thSAV_definition}), we choose $\psi = 100\Psi_{max} = 5$.
Fig. \ref{fig:4thConvergenceForced} shows the global error in all kinematic quantities using the proposed IMEX-BDF$k$-SAV schemes with $\psi = 5$.
For comparison, the global errors from the standard Bathe method with $\gamma=0.5$ \cite{bathe2005composite,bathe2007conserving}, Newmark trapezoidal rule (TR) \cite{newmark1959method}, and the generalized-$\alpha$ method with $\rho_\infty=0$ \cite{chung1993time} are also shown.
From these plots, we can verify that the proposed IMEX-BDF$k$-SAV schemes are $k^{\textrm{th}}$-order accurate in displacement, velocity, and acceleration for linear problems in structural dynamics.

To investigate the effect of $\psi$, we plot the global errors for different values of $\psi/\Psi_{max}$ for two values of $\Delta t/T_0$ in Fig.~\ref{fig:4thConvergenceForcedpsi}.
We observe that global errors do depend on the value of $\psi$, however, 
as for $\psi/\Psi_{max}$ greater than $100$, the global errors stabilize.

\begin{figure}[!htbp]
    \centering
    \includegraphics[width=0.8\linewidth]{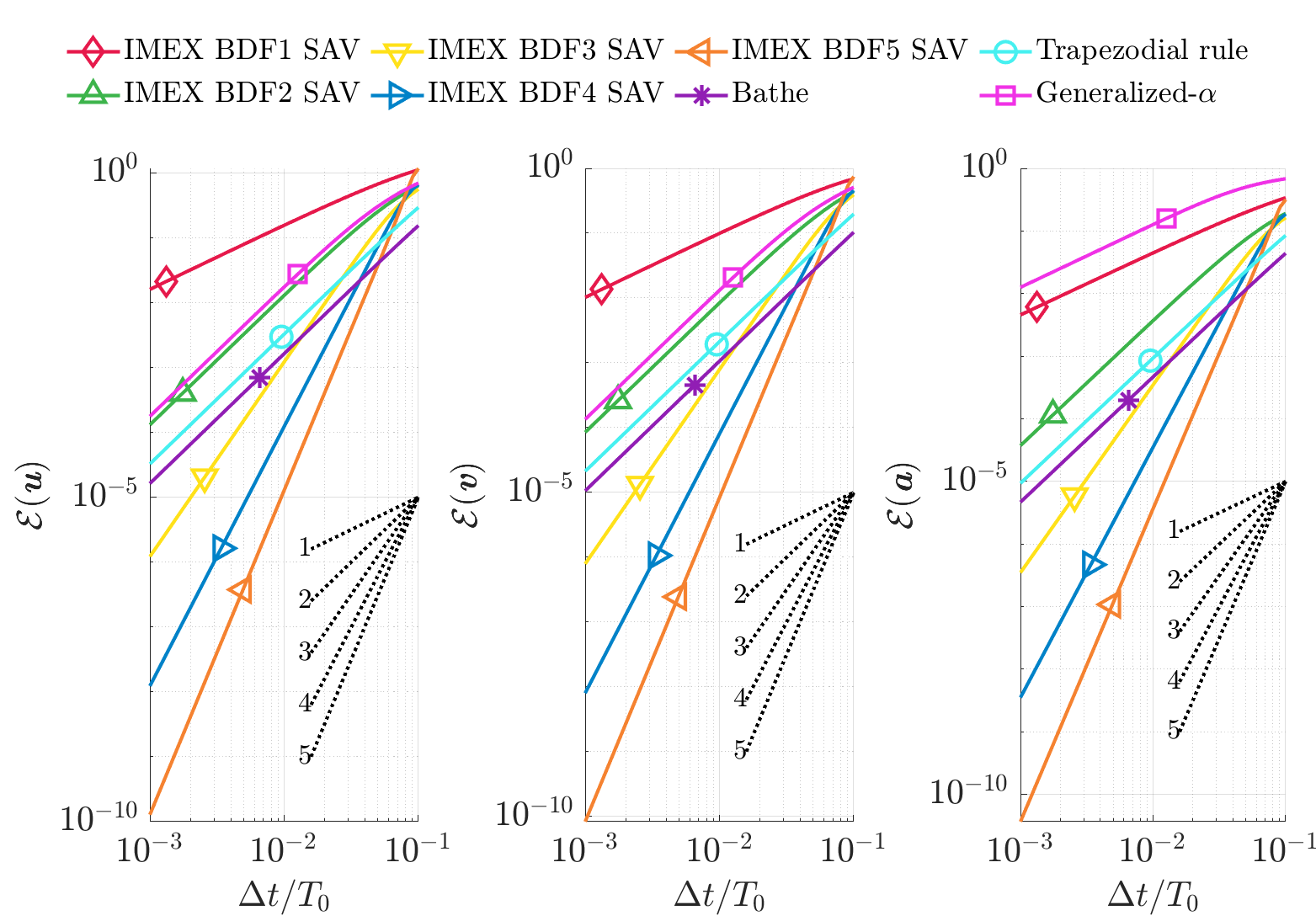}
    \caption{Error norms for the damped SDOF problem with $f^{ext}=\sin{2\omega_0 t}$ and $u_0=v_0=0$.
    }
    \label{fig:4thConvergenceForced}
\end{figure}

\begin{figure}[!htbp]
    \centering
    {\includegraphics[width=0.85\textwidth]{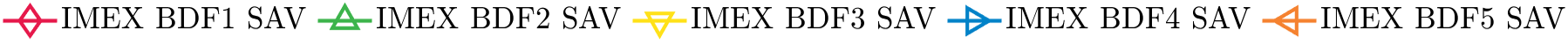}}
    \subcaptionbox{$\Delta t/T_0=0.1$}{\includegraphics[width=0.4\textwidth]{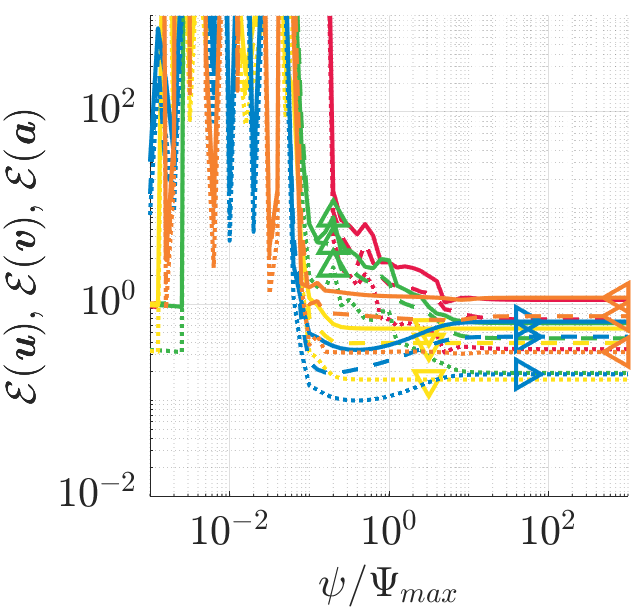}}
    \subcaptionbox{$\Delta t/T_0=0.001$}{\includegraphics[width=0.4\textwidth]{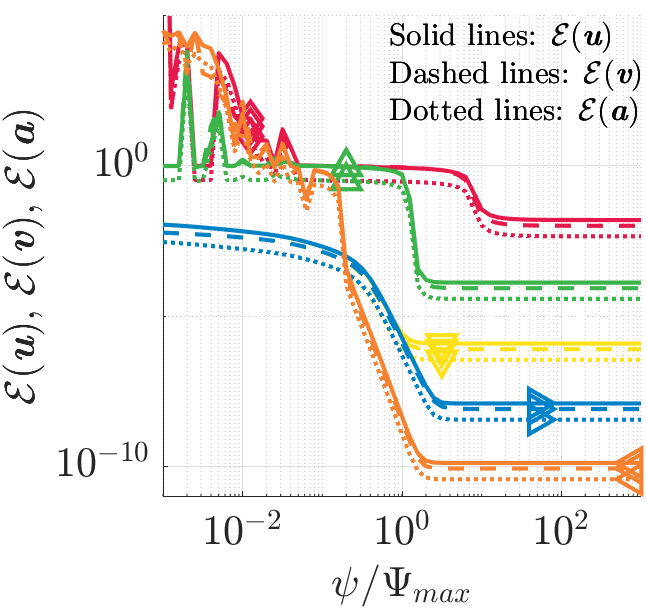}}
    \caption{Error norms for the damped SDOF problem with $f^{ext}=\sin{2\omega_0 t}$ and $u_0=v_0=0$ using the proposed BDF$k$ SAV schemes for various values of $\psi/\Psi_{max} \in [10^{-3},10^3]$ at intervals of $10^{0.1}$.}
    \label{fig:4thConvergenceForcedpsi}
\end{figure}

\subsubsection{Nonlinear structural dynamics}
\label{subsubsec:4th242}
Here, we consider two types of nonlinear SDOF problems: a Van der Pol oscillator and a Duffing oscillator.
The governing equations of these problems are expressed as
\begin{align}
    &\text{Van der Pol oscillator:} \quad m \ddot{u} + k_1 u - \mu (1-u^2)\dot{u} = f^{ext} \label{eqn:4thEqnVanderPol} \\
    &\text{Duffing oscillator:} \quad \quad \enspace \; \: m \ddot{u} + c \dot{u} + k_1 u + k_3 u^3 = f^{ext} \label{eqn:4thEqnDuffing}
\end{align}

First, we consider the Van der Pol oscillator with $m=1$, $k_1=1$, $\mu=2$, and $f^{ext}=0$ over the time duration $[0,T]$, where $T=15$.
A nonzero initial displacement and a zero initial velocity is applied: $u_0=2$ and $v_0=0$.
As estimate of the maximum pseudo-energy $\Psi$ is obtained from an approximate solution $\breve{u}$ and $\breve{v}$ of a limit cycle for any value of $\mu>0$ (see Ref.~\cite{lopez2000limit} for details):
\begin{equation}
    \breve{v}(\breve{u}) = 
    \begin{cases}
        \sqrt{3-\breve{u}^2} & \textrm{if} \: -\sqrt{3}<\breve{u}<-1 \\   
        \sqrt{3-\breve{u}^2} + \mu \left( - \frac{1}{3}\breve{u}^3+\breve{u}+\frac{2}{3}\right) & \textrm{if} \: -1<\breve{u}<\sqrt{3}
    \end{cases}    
    \label{eqn:4thVanApp}
\end{equation}
This approximate solution is compared to a reference solution for the Van der Pol oscillator in Fig.~\ref{fig:4thVanderPolRef}. 
Using Eq.~(\ref{eqn:4thVanApp}), the pseudo-energy $\Psi$ is approximated as $\frac{1}{2}m \breve{v}^2 + \frac{1}{2}k_1 \breve{u}^2$ and the maximum pseudo-energy is estimated to be $\Psi_{max} \approx 8.9$.

\begin{figure}[!htbp]
    \centering
    \includegraphics[width=0.5\linewidth]{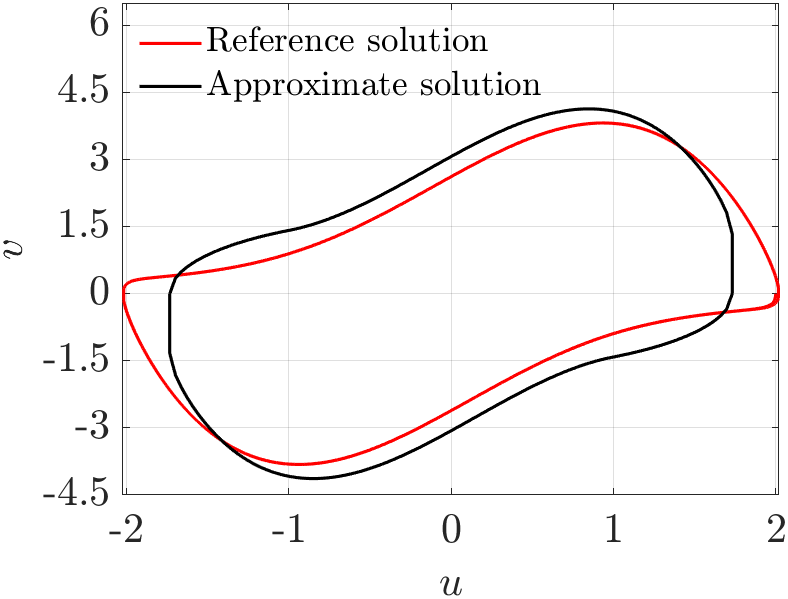}
    \caption{Phase portrait of the SDOF Van der Pol problem with $f^{ext}(t)=0$, $u_0=2$, and $v_0=0$.
    }
    \label{fig:4thVanderPolRef}
\end{figure}

Next, we consider the Duffing oscillator with $m=1$, $c=1$, $k_1=1$, $k_3=20$, and $f^{ext}(t)=p_0\cos(2\pi t)$ over the time duration of $[0,T]$.
For this problem, we choose the load intensity $p_0 = 500$ and time duration $T=1$ and initial conditions as $u_0=5$ and $v_0=-10$.
If we neglect the initial velocity and nonlinear effects, then we can estimate the maximum displacement of the oscillator as $u_{max} \approx u_0 + 2p_0/k_1$.
Thus, the maximum pseudo-energy stored in the system may be approximated as $\Psi_{max} \approx \frac{1}{2} m v_0^2 + \frac{1}{2} k_1 u_{max}^2 \approx 5 \times 10^5$.

Reference solutions for both problems, the Van der Pol oscillator and Duffing oscillator, are obtained by solving them with the standard Bathe scheme \cite{bathe2005composite,bathe2007conserving} and the $4^\textrm{th}$-order Runge-Kutta (RK4) scheme using a time-step of $\Delta t = 10^{-7}$ for both. 
The difference in the solutions obtained from these schemes is quantified using Eq.~(\ref{eqn:4thErrorNorm}) where $\boldsymbol{z}$ is taken to be the RK4 solution and $\boldsymbol{\tilde{z}}$ is taken to be the solution obtained from the Bathe scheme.    
The error (or difference) between these solutions is found to be of the order of $10^{-11}$ for the Van der Pol oscillator and of the order of $10^{-13}$ for the Duffing oscillator.

Next, both problems are solved using the proposed IMEX-BDF$k$-SAV schemes and other time integration methods such as the Bathe, TR, and generalized-$\alpha$ schemes.
For implicit schemes such as the Bathe, TR and generalized-$\alpha$, Netwon-Raphson iterations are performed at each time-step until the absolute residual reaches a tolerance of $10^{-7}$.
Note that the proposed IMEX-BDF$k$-SAV schemes do not require iterations because nonlinear terms are treated explicitly.
Figs.~\ref{fig:4thConvergenceHomoNL} and \ref{fig:4thConvergenceForcedNL} show global errors in kinematic quantities using the proposed IMEX-BDF$k$-SAV schemes with $\psi = 100\Psi_{max}$, standard Bathe scheme, TR, and generalized-$\alpha$ scheme for the unforced Van der Pol oscillator and the forced Duffing oscillator, respectively.
It can be observed that the proposed IMEX-BDF$k$-SAV schemes give $k^{\textrm{th}}$-order accurate solutions for both nonlinear problems.
Lastly, note that errors for all time-steps considered are larger than about $10^{-9}$ for the Van der Pol oscillator and $10^{-11}$ for the and Duffing oscillator, which are much larger than the difference between the Bathe and RK4 solutions that were used as a reference for these problems.

Figs.~\ref{fig:4thConvergenceHomoNLpsi} and \ref{fig:4thConvergenceForcedNLpsi} show the variation of global errors for the proposed IMEX-BDF$k$-SAV schemes for different values of $\psi/\Psi_{max}$.
As with linear problems, the errors depend upon the choice of $\psi$, but for values of $\psi/\Psi_{max} > 10^{2}$, the errors stabilize.
Thus, we suggest choosing a value of $\psi = 100\Psi_{max}$ for nonlinear problems as well.

\begin{figure}[!htbp]
    \centering
    \includegraphics[width=0.8\linewidth]{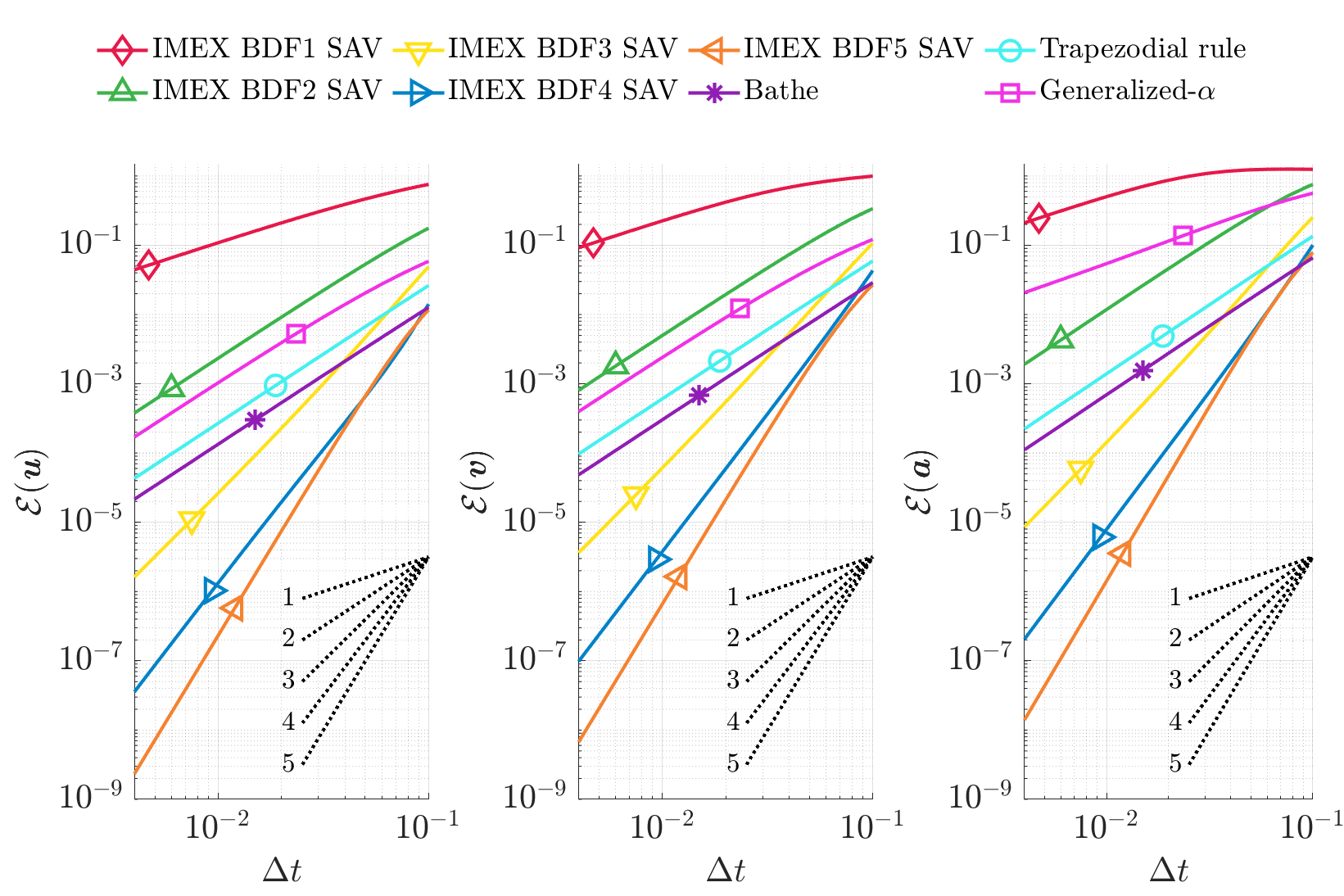}
    \caption{Error norms for the SDOF Van der Pol problem with $f^{ext}(t)=0$, $u_0=2$, and $v_0=0$ for 20 different time-steps between $4\times10^{-3}$ to $10^{-1}$.
    }
    \label{fig:4thConvergenceHomoNL}
\end{figure}

\begin{figure}[!htbp]
    \centering
    \includegraphics[width=0.8\linewidth]{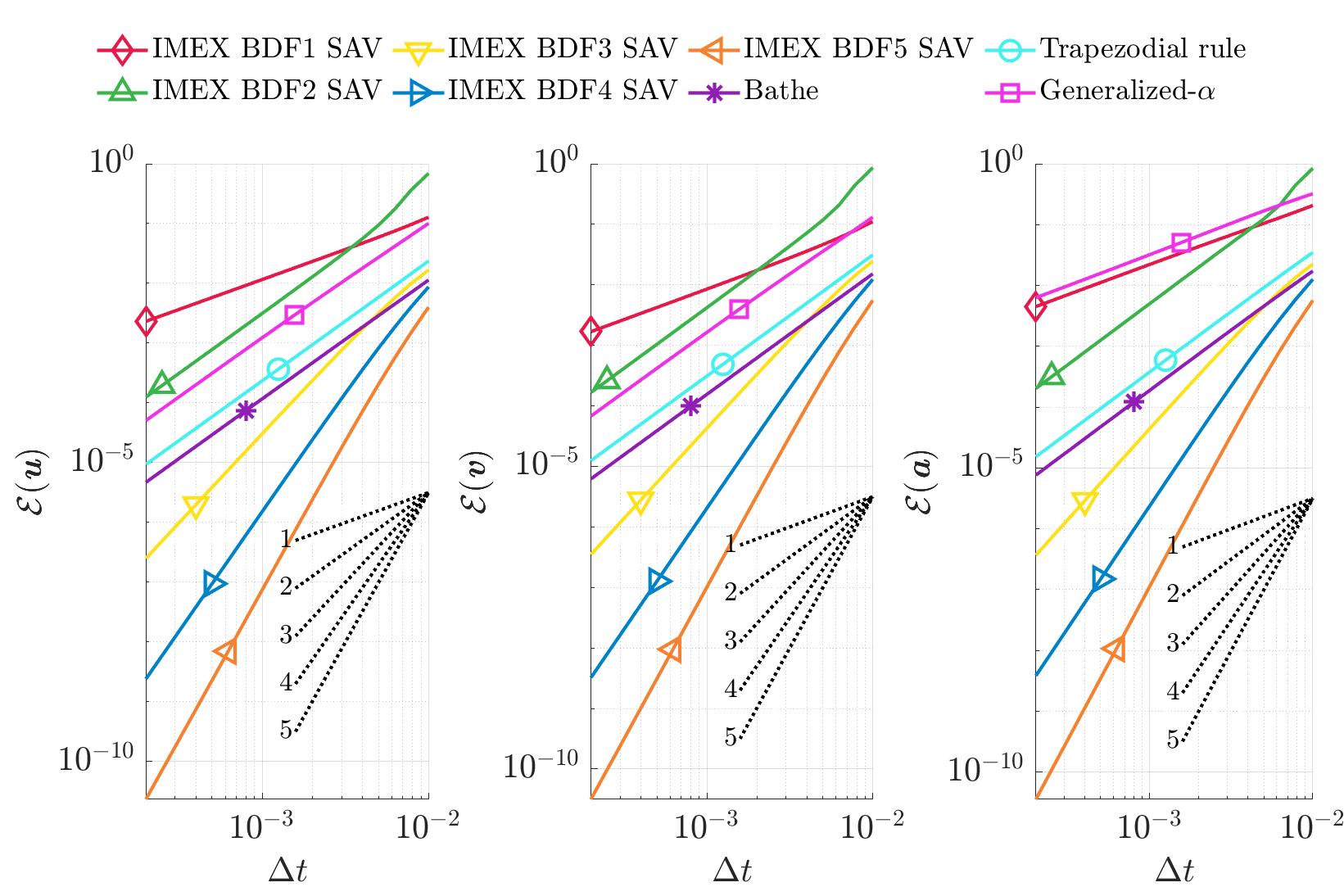}
    \caption{Error norms for the SDOF Duffing problem with $f^{ext}(t)=500\cos{2\pi t}$, $u_0=5$, and $v_0=-10$ for 20 different time-steps between $2\times10^{-4}$ to $10^{-2}$.
    }
    \label{fig:4thConvergenceForcedNL}
\end{figure}

\begin{figure}[!htbp]
    \centering
    {\includegraphics[width=\textwidth]{pics/4thObjective/Convergence/legend_IMEX}}
    \subcaptionbox{$\Delta t=0.1$}{\includegraphics[width=0.4\textwidth]{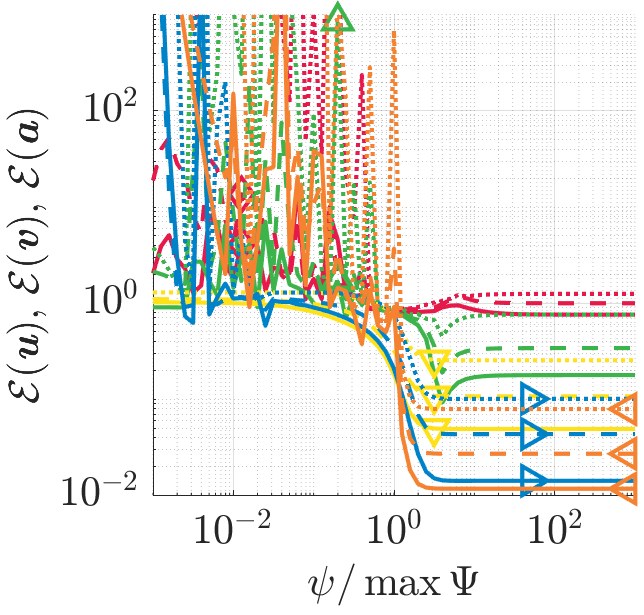}}
    \subcaptionbox{$\Delta t=0.004$}{\includegraphics[width=0.4\textwidth]{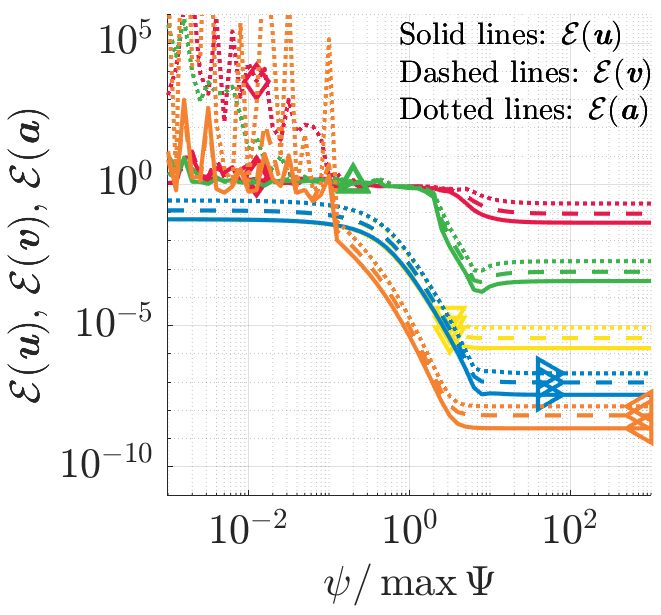}}
    \caption{Error norms for the SDOF Van der Pol problem with $f^{ext}(t)=0$, $u_0=2$, and $v_0=0$, using the proposed IMEX-BDF$k$-SAV schemes for various values of $\psi/\Psi_{max} \in [10^{-3},10^3]$ at intervals of $10^{0.1}$.}
    \label{fig:4thConvergenceHomoNLpsi}
\end{figure}

\begin{figure}[!htbp]
    \centering
    {\includegraphics[width=\textwidth]{pics/4thObjective/Convergence/legend_IMEX}}
    \subcaptionbox{$\Delta t=0.01$}{\includegraphics[width=0.4\textwidth]{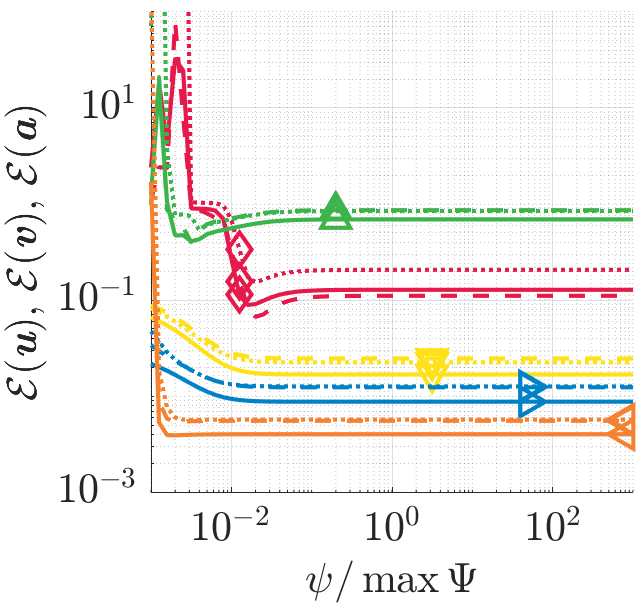}}
    \subcaptionbox{$\Delta t=0.0002$}{\includegraphics[width=0.4\textwidth]{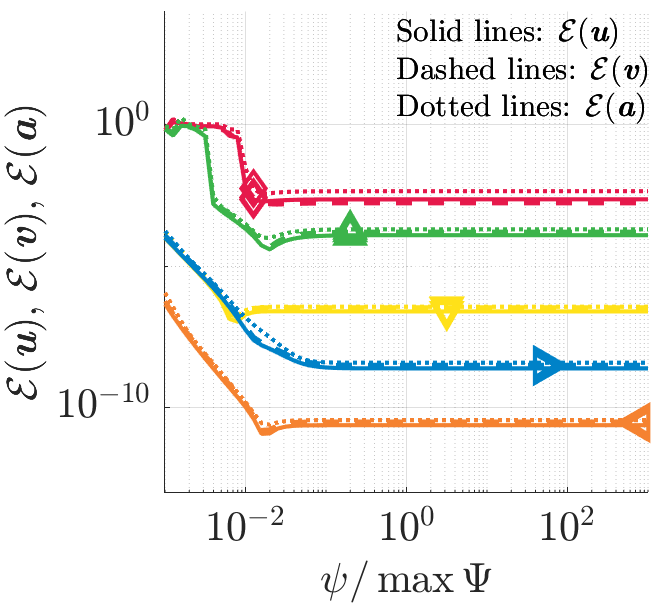}}
    \caption{Error norms for the SDOF Duffing problem with $f^{ext}(t)=500\cos{2\pi t}$, $u_0=5$, and $v_0=-10$, using the proposed IMEX-BDF$k$-SAV schemes for various values of $\psi/\Psi_{max} \in [10^{-3},10^{3}]$ at intervals of $10^{0.1}$.}
    \label{fig:4thConvergenceForcedNLpsi}
\end{figure}

\newpage
\section{Numerical examples}
\label{sec:4th3}
In this section, several numerical examples for nonlinear structural dynamics are presented to demonstrate the performance of the proposed IMEX-BDF$k$-SAV schemes: 1) a simple pendulum model, 2) a spring-pendulum model, 3) a multi-degree-of-freedom (MDOF) Duffing oscillator, and 4) a 9-story building.
Solutions obtained from the proposed IMEX-BDF$k$-SAV schemes with $\psi = 100\Psi_{max}$, Bathe scheme \cite{bathe2005composite,bathe2007conserving}, TR \cite{newmark1959method}, Generalized-$\alpha$ scheme \cite{chung1993time}, Kim's $4^{\textrm{th}}$-order explicit scheme \cite{kim2019higher}, and central difference (CD) scheme are compared. 
Note that when nonlinear internal forces depend on velocities, we employ the central difference method proposed by Park and Underwood \cite{park1980variable}, to avoid the need for iterations at each time-step.
To compare accuracy, we define three measures of error: the local instantaneous error $\epsilon^i_n$, maximum instantaneous error $\epsilon_n$, and maximum error $\epsilon$:
\begin{align}
    \epsilon^i_n(\boldsymbol{z}) = \frac{|z^i_n - \tilde{z}^i_n|}{\displaystyle\max_{i,n}\tilde{z}^i_n - \displaystyle\min_{i,n}\tilde{z}^i_n}, \quad
    \epsilon_n(\boldsymbol{z}) = \displaystyle\max_{i}\epsilon^i_n(\boldsymbol{z}), \quad
    \epsilon(\boldsymbol{z}) = \displaystyle\max_{n}\epsilon_n(\boldsymbol{z})
    \label{eqn:4thErr_definition}
\end{align}
where $i$ denotes the degree of freedom ($i=1,2,..,N_i)$, where $N_i$ is the total number of degree of freedom.
The variables $z^i_n$ and $\tilde{z}^i_n$ represent the numerical and reference solutions, respectively, at a degree of freedom $i$ at time $t_n$.

\subsection{Simple pendulum model}
We consider a simple pendulum -- a benchmark problem for demonstrating the performance of nonlinear systems \cite{kim2017new,kim2018improved,fung2001solving}. 
The equation of motion for this problem is given as: 
\begin{align}
    &\ddot{\theta}(t) + \dfrac{g}{L}\sin{(\theta(t))}=0 \\
    &\theta(t=0)=\theta_0, \quad \dot{\theta}(t=0)=v_0
\end{align}
where $\theta$, $g$, and $L$ are the angle of swing, gravitational constant, and pendulum length, respectively.
The exact solution of this problem can be obtained from the analytical relation between $\theta(t)$ and $t$ \cite{fung2001solving}:
\begin{equation}
    t = \pm \int_{\theta_0}^{\theta(t)}\dfrac{1}{\sqrt{\left(v_0\right)^2+2\dfrac{g}{L}(\cos{z}-\cos{\theta_0})}}  dz
    \label{eqn:4thPendulumExact}
\end{equation}
Note that in the pendulum problem, we can modify the level of nonlinearity by controlling $\theta_0$ and $v_0$.
As $\theta_0$ and/or $v_0$ increase, the system becomes more nonlinear.
Further details of this simple pendulum problem and comments on the importance of the problem are given in Refs.~\cite{kim2018improved,fung2001solving}.

The parameters chosen for this pendulum problem are
$L = g = 1$, and initial conditions are $\theta_0=0$ and $v_0=1.95$ for which the maximum angle of swing of the pendulum computed from Eq.~(\ref{eqn:4thPendulumExact}), $\theta_{max}=2.6934$.
The time duration of the simulation is $[0,2T_1]$, where $T_1=11.6576$ is the period of system obtained from Eq.~(\ref{eqn:4thPendulumExact}).
During the simulation, the maximum pseudo-energy is approximated as $\Psi_{max} \approx \frac{1}{2}v_0^2 \approx 1.9$ because the system is conservative.
Thus, we choose $\psi=100\Psi_{max}=190$ for the proposed IMEX-BDF$k$-SAV schemes.

Fig.~\ref{fig:4thSimplePendulum_history} shows numerical results for the angle of swing for different time integration schemes.
The time-step $\Delta t$ for the proposed IMEX-BDF$k$-SAV schemes is chosen as $T_1/20$ and $T_1/100$. 
For computational cost parity, the time-steps of the comparative schemes are chosen as $n_{sub} \times \Delta t$, where $n_{sub}$ is the number of sub-stages associated with the time integration scheme.
For instance, $n_{sub}$ for the Bathe scheme is 2 and for the Kim's $4^\textrm{th}$-order explicit scheme $n_{sub}$ is 4.
Note that for the TR scheme, generalized-$\alpha$ scheme, and CD scheme, number of sub-stages $n_{sub}=1$.
For sake of clarity, we do not include the conventional IMEX-BDF$k$ schemes (without SAV) in this figure because their response is almost identical to that of the corresponding IMEX-BDF$k$-SAV schemes for this problem.
We observe that when $\Delta t = T_1/20$, several schemes drift away from the exact solution because when $|\theta|$ exceeds $\pi$ then the pendulum makes more than a half-revolution and consequently cannot maintain oscillations around the equilibrium position.
With a finer time-step of $\Delta t = T_1/100$, none of the methods suffer from this problem, however, IMEX-BDF2 and IMEX-BDF2-SAV have an uncharacteristically large period elongation.
IMEX-BDF2 and IMEX-BDF2-SAV schemes seem to exhibit this behavior specifically for this problem and across different values of $\Delta t$.

\begin{figure}[!htbp]
    \centering
    {\includegraphics[width=0.52\textwidth]{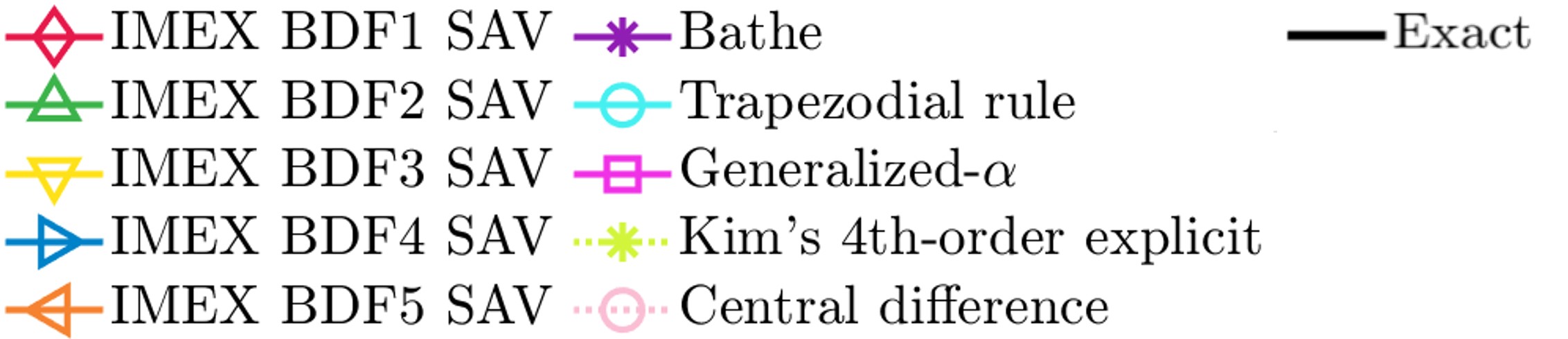}}
    \subcaptionbox{$\Delta t=T_1/20$}{\includegraphics[width=0.49\textwidth]{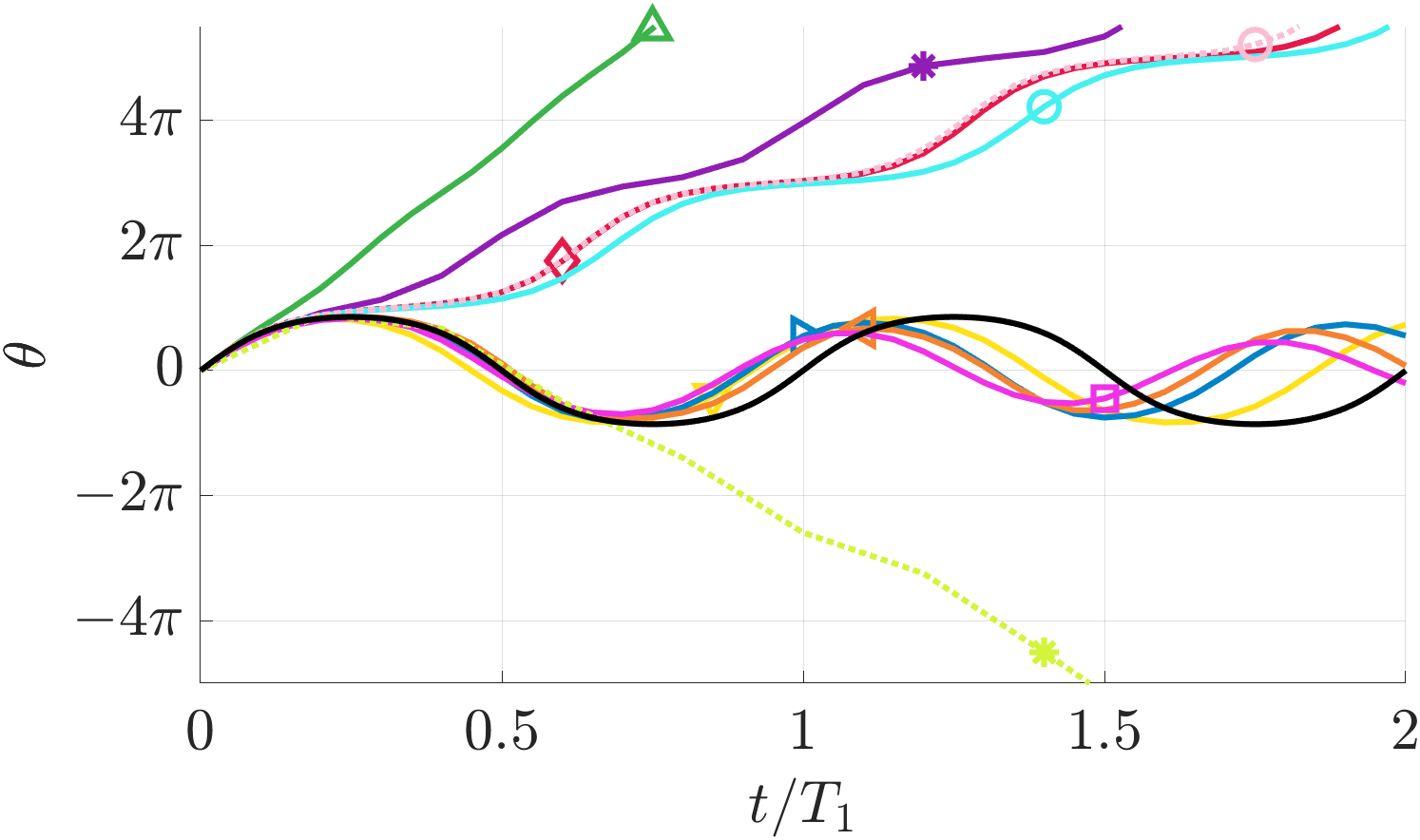}}
    \subcaptionbox{$\Delta t=T_1/100$}{\includegraphics[width=0.49\textwidth]{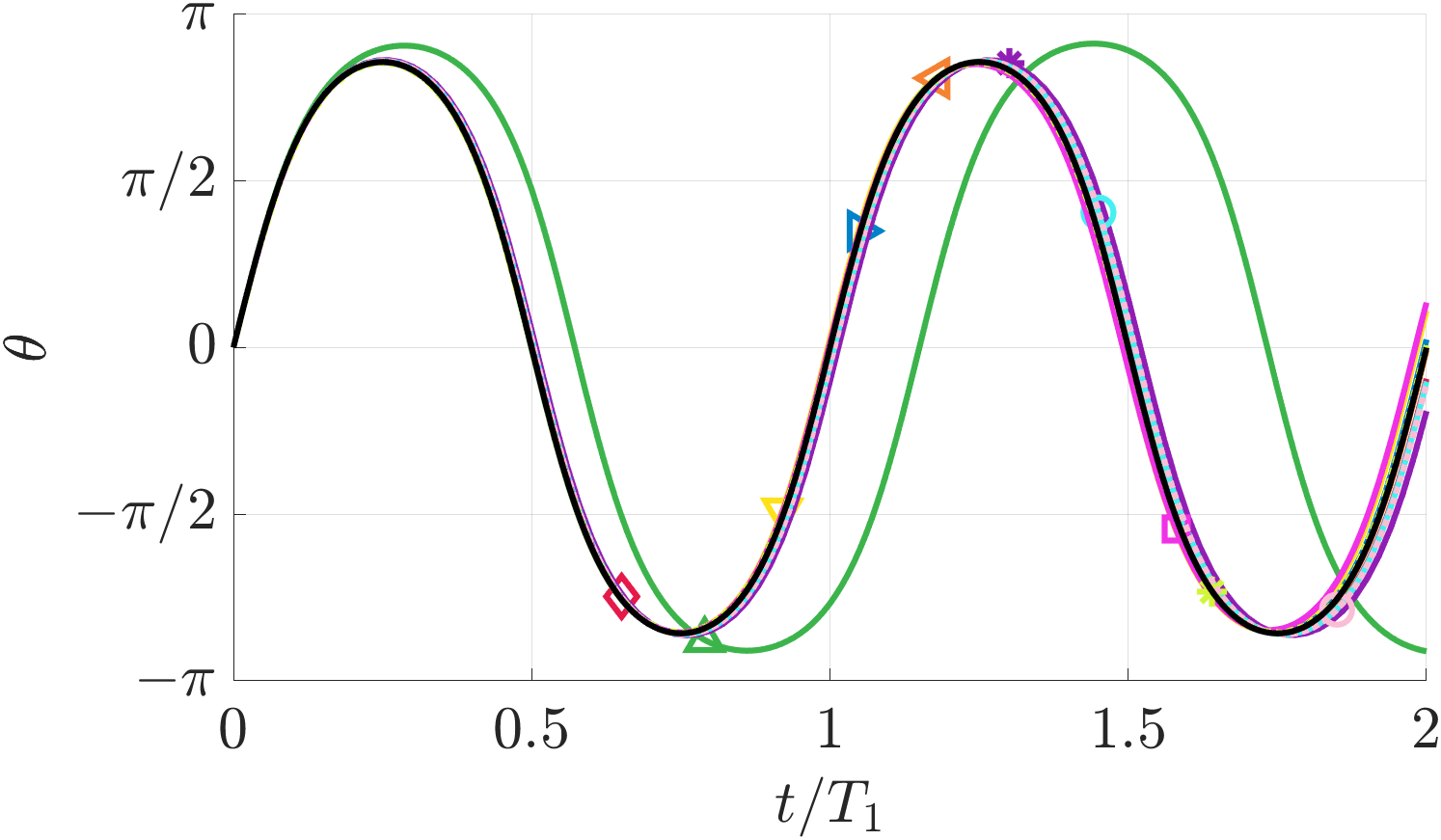}}
    \caption{Time history of angles for various schemes in the simple pendulum model.}
    \label{fig:4thSimplePendulum_history}
\end{figure}

\begin{figure}[!htbp]
    \centering
    {\includegraphics[width=0.6\textwidth]{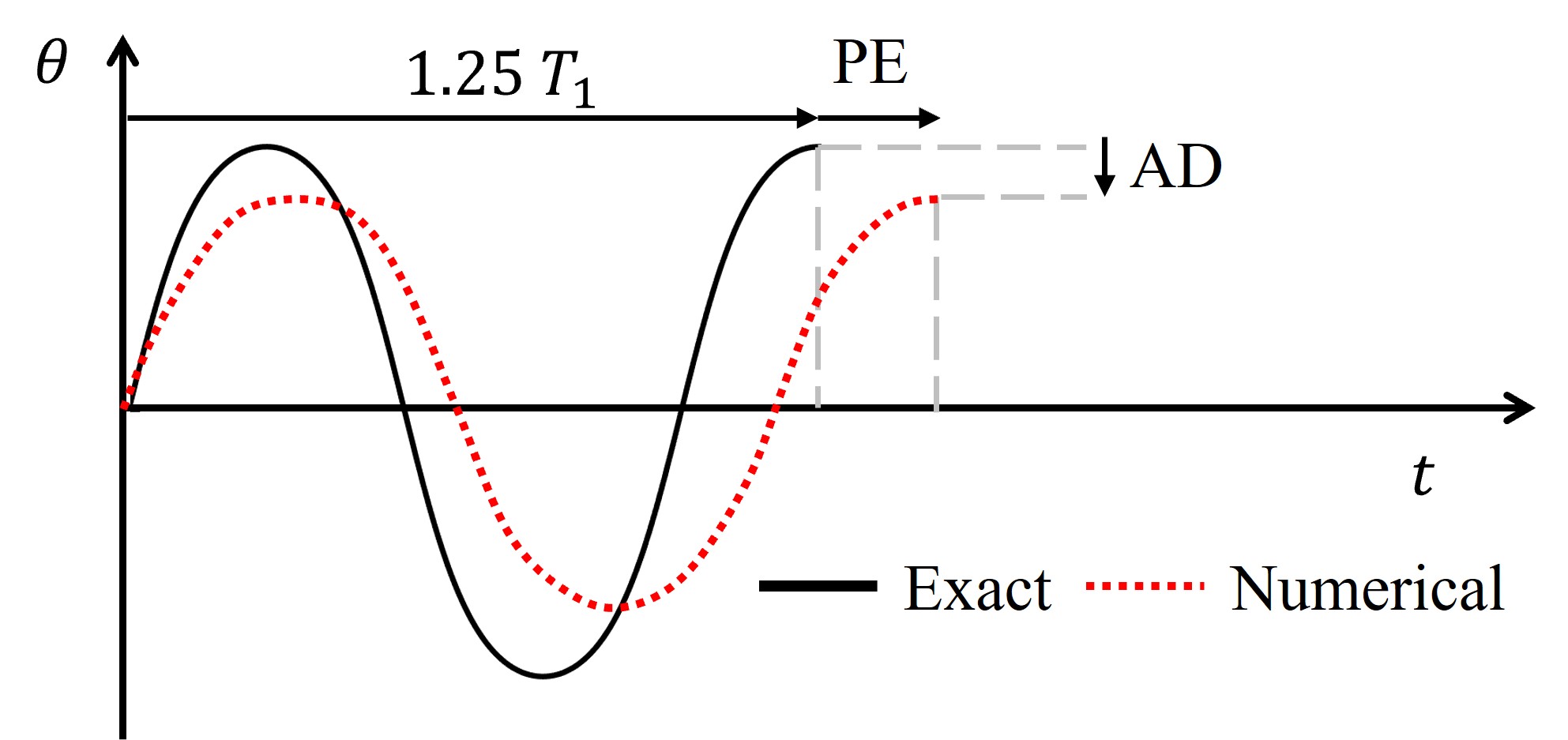}}
    \caption{Definition of period elongation (PE) and amplitude decay (AD) in the simple pendulum model.}
    \label{fig:4thSimplePendulum_PEAD}
\end{figure}

\begin{table}[!htbp]
    \caption{\label{tab:4thSimplePendulumSummary} Period elongation (PE), amplitude decay (AD), and computational costs for various schemes over the time interval $[0,\,2T_1]$ in the simple pendulum problem.}
    \centering 
    \resizebox{\linewidth}{!}{
    \begin{tabular}{l|c c |r r |c r | c  r}
         & \multicolumn{2}{c|}{Avg. \# of iter.}& \multicolumn{2}{c|}{Comput. cost ($\mu s$)}&  \multicolumn{2}{c|}{$PE/1.25T_1$ (\%)}&\multicolumn{2}{c}{$AD/\theta_{max}$ (\%)}\\
         \multicolumn{1}{r|}{$\Delta t$} & $T_1/20$&$T_1/100$& $T_1/20$&$T_1/100$& $T_1/20$&$T_1/100$& $T_1/20$&$T_1/100$\\
         Scheme & && && && &\\
         \hline
         IMEX-BDF1-SAV & 1\phantom{.00} 
 &1\phantom{.000}& 14 &65&  * &\phantom{0}0.7150&* &-0.3242\\
         IMEX-BDF2-SAV & 1\phantom{.00} 
 &1\phantom{.000}& 19 &80&  * &15.3303&* &-6.4429\\
         IMEX-BDF3-SAV & 1\phantom{.00} 
 &1\phantom{.000}& 25 &120&  \phantom{0}-8.2323 &-0.5145&\phantom{0}3.0809 &\phantom{-}0.4998\\
         IMEX-BDF4-SAV & 1\phantom{.00} 
 &1\phantom{.000}& 25 &122&  -11.7333 &-0.1366&10.3890 &\phantom{-}0.1019\\
         IMEX-BDF5-SAV & 1\phantom{.00} 
 &1\phantom{.000}& 25 &129&  -10.9731 &\phantom{0}0.0373&20.7512 &-0.0304\\
         Bathe         & 2\phantom{.00} 
 &1.115& 140 &345&  * &\phantom{0}1.0886&* &-0.6543\\
         TR & 2\phantom{.00} 
 &1.045& 146 &356&  * &\phantom{0}0.7056&* &-0.3551\\
         Generalized-$\alpha$ & 2.35
 &1.555& 150 &408&  -13.8742 &-0.4199&29.1999 &\phantom{-}0.7268\\
         Kim's $4^\textrm{th}$-order & 1\phantom{.00}
 &1\phantom{.000}& 13 &79&  *  &-0.6060&* &\phantom{-}0.0318\\
         CD & 1\phantom{.00}  &1\phantom{.000}& 45 &176&  * &\phantom{0}0.7561&* &-0.3831\end{tabular}
    }
    \raggedright
    {\scriptsize
     * The pendulum cannot maintain oscillations around the equilibrium position.
     }
\end{table}

To compare accuracy of the different schemes, we measure errors in terms of period elongation and amplitude decay using the approach presented by Bathe and Wilson \cite{bathe1972stability}.
These measures can be obtained by comparing the exact solution in Eq.~(\ref{eqn:4thPendulumExact}) with the numerical solutions (see Fig.~\ref{fig:4thSimplePendulum_PEAD}).
In Table~\ref{tab:4thSimplePendulumSummary}, we summarize the average number of iterations per time step, computational cost, percentage period elongation ($PE/1.25T_1$), and amplitude decay ($AD/\theta_{max}$) for several time integration schemes.
We observe that the computational cost of the proposed IMEX-BDF$k$-SAV schemes is at par or lower than existing explicit methods with the added benefit of unconditional stability over the explicit methods.
Further, at the smaller time-step ($\Delta t = T_1/100$), period elongation and amplitude decay improve with the order of the proposed IMEX-BDF$k$-SAV schemes for $k$ = 3, 4 and 5.
This trend is found to continue for progressively smaller time-steps as well that are not shown in this table.
This numerical example demonstrates the main benefits of the IMEX-BDF$k$-SAV schemes i.e. unconditional stability, low computational cost, no need for iterations for nonlinear problems, and high-order accuracy.

\subsection{Spring-pendulum model}

We solve an elastic spring-pendulum problem \cite{chung1994new,kim2018improved,kim2018improved2} as shown in Fig.~\ref{fig:4thSpringPendulum} .
The governing differential equation of the problem can be expressed as:
\begin{equation}
    \begin{bmatrix}
        \ddot{x} \\ \ddot{\theta}
    \end{bmatrix}
    +
    \begin{bmatrix}
        k/m & 0 \\ 0 & 0
    \end{bmatrix}
    \begin{bmatrix}
        x \\ \theta
    \end{bmatrix}
    +
    \begin{bmatrix}
        -(L_0+x)(\dot{\theta})^2-g\cos{\theta} \\
        (2\dot{x}\dot{\theta}+g\sin{\theta})/(L_0+x)
    \end{bmatrix}
    =
    \begin{bmatrix}
        f_x \\ f_\theta
    \end{bmatrix}
    \label{eqn:4thSpringPenGE}
\end{equation}
where $x$ and $\theta$ denote the radial and circumferential displacements, respectively, $m$ is the mass of the pendulum, $k$ is the spring constant, $L_0$ is the initial length of the pendulum, and $g$ is the gravitational constant.
The parameters for this problem are selected as $m=1$, $k=98.1$, $L_0=0.5$, and $g=9.81$.
In this study, we consider a manufactured analytical solution in time from $t=0$ to $t=2$:
\begin{equation}
    x(t) = \theta(t) = 0.1 \sin\left(2\pi t\right)
    \label{eqn:4thSpringPenManSol}
\end{equation}
Substituting Eq.~(\ref{eqn:4thSpringPenManSol}) into Eq.~(\ref{eqn:4thSpringPenGE}), the required forcing functions $f_x$ and $f_\theta$ are obtained.
For this nonlinear system, the maximum pseudo-energy is $\Psi_{max} = 0.49$.
Thus, for the proposed IMEX-BDF$k$-SAV schemes, we select $\psi=100\Psi_{max}=49$.

\begin{figure}[!htbp]
    \centering
    {\includegraphics[width=0.5\textwidth]{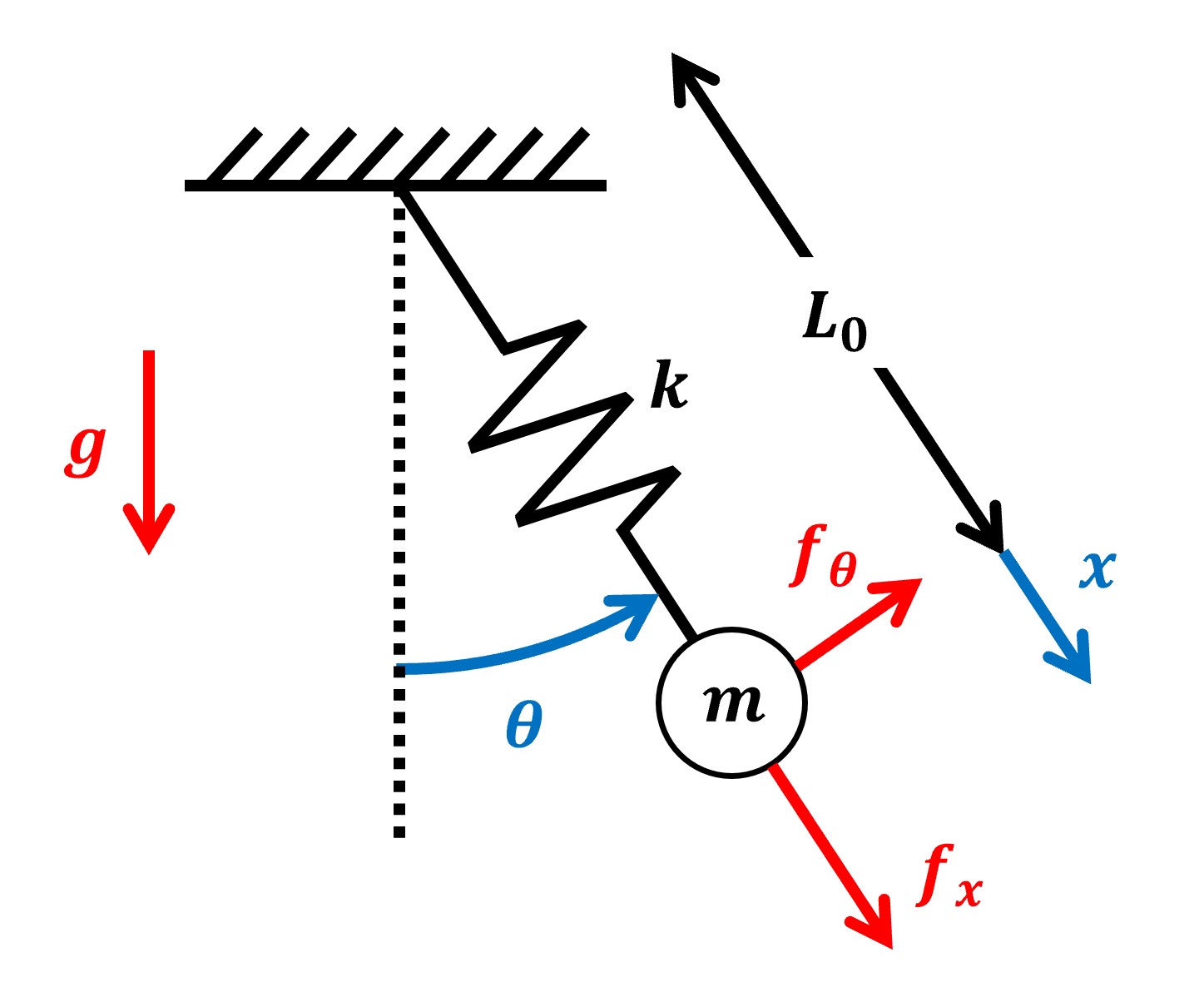}}
    \caption{A 2DOF spring-pendulum model.}
    \label{fig:4thSpringPendulum}
\end{figure}

Fig.~\ref{fig:4thSpringPendulum_d} shows time histories of displacements for various time integration schemes with $\Delta t=T_1/10$ and $T_1/20$ for the proposed IMEX-BDF$k$-SAV schemes.
Similar to previous example, the time-steps of the comparative schemes are selected as $n_{sub} \times \Delta t$.
In this problem, since the internal force depends upon velocities, we employ the central difference method proposed by Park and Underwood \cite{park1980variable}, i.e. CD(PU), to avoid the need for iterations at each time-step.
Note that, at the larger time-step ($n_{sub} \times \Delta t=4 T_1/10$), Kim's $4^\textrm{th}$-order scheme is unstable.
We also observe that the IMEX-BDF1-SAV scheme, being a low-order scheme, has high errors for both the time-steps considered here. 
The performance of different time integration schemes for this problem is summarized in Table~\ref{tab:4thSpringPendulumSummary}.
We note that, at the smaller time-step, the IMEX-BDF4-SAV and IMEX-BDF5-SAV schemes lead to the lowest error, have low computational cost and are guaranteed to be unconditionally stable.

\begin{figure}[!htbp]
    \centering
    {\includegraphics[width=0.52\textwidth]{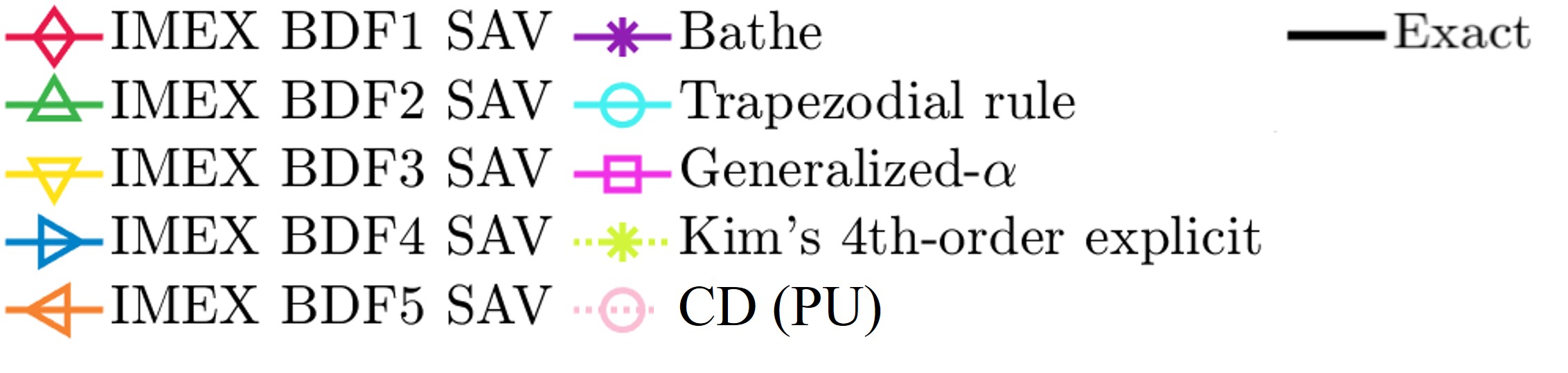}}
    \subcaptionbox{$x$ when $\Delta t=T_1/10$}{\includegraphics[width=0.49\textwidth]{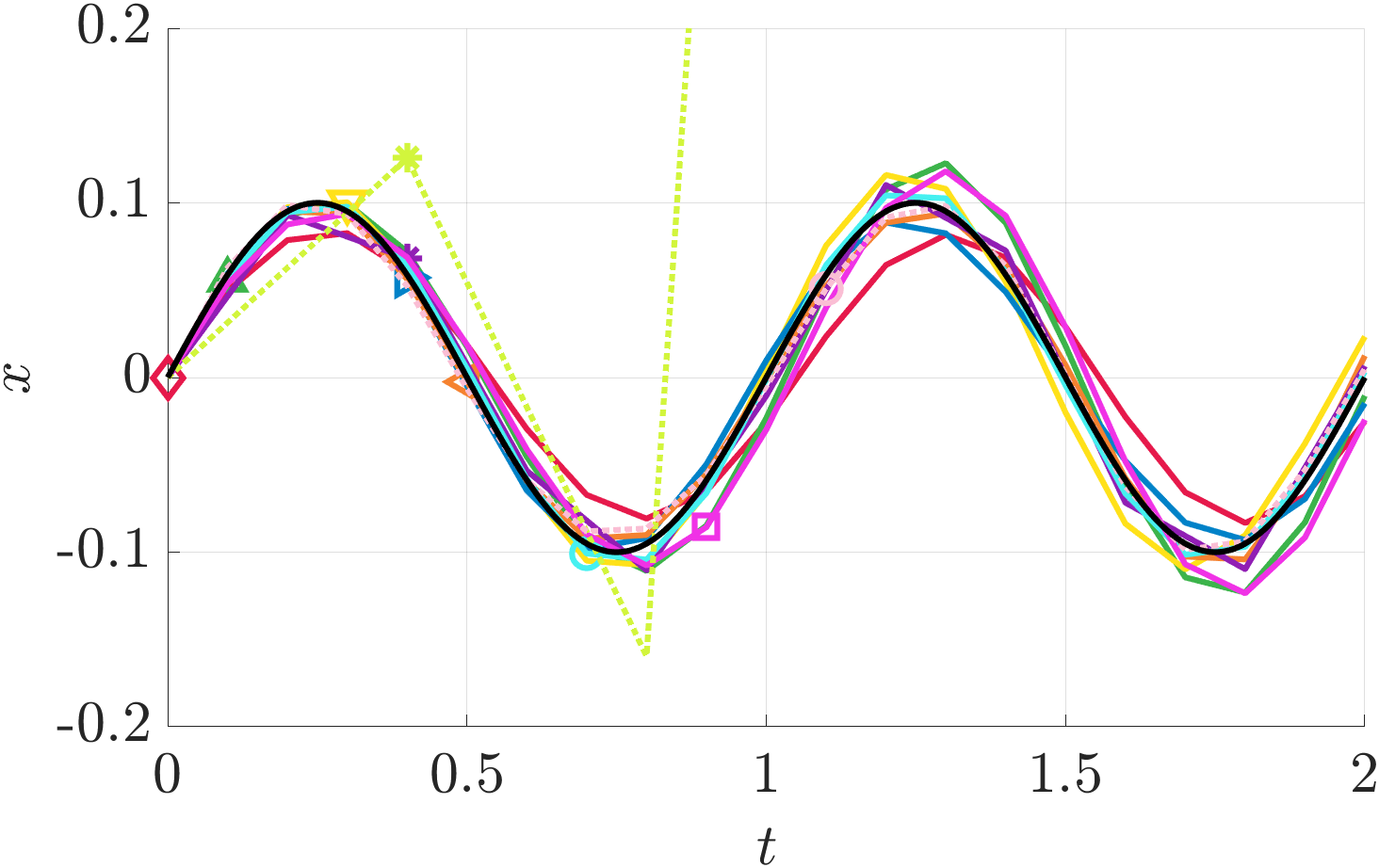}}
    \subcaptionbox{$\theta$ when $\Delta t=T_1/10$}{\includegraphics[width=0.49\textwidth]{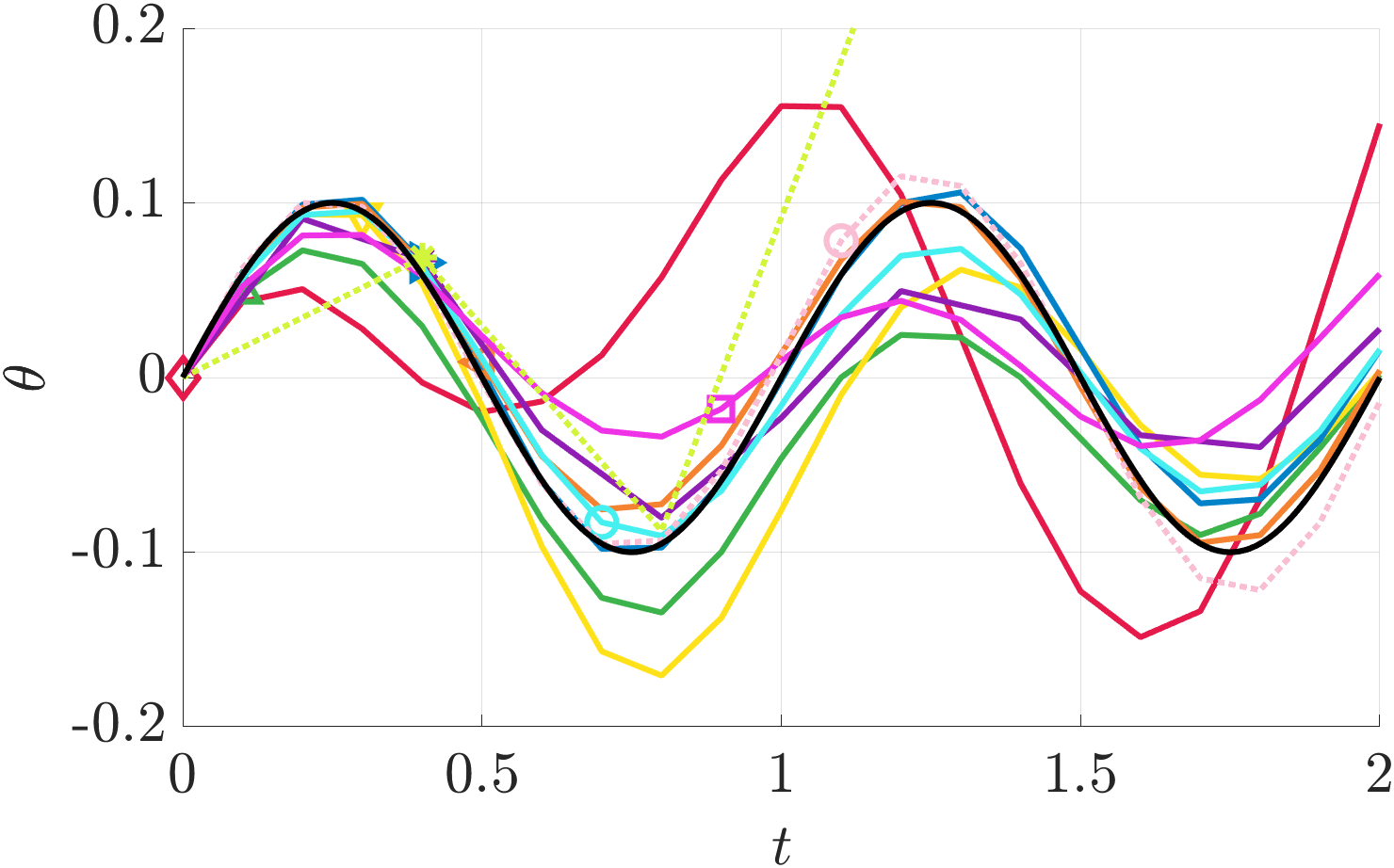}}
    \subcaptionbox{$x$ when $\Delta t=T_1/20$}{\includegraphics[width=0.49\textwidth]{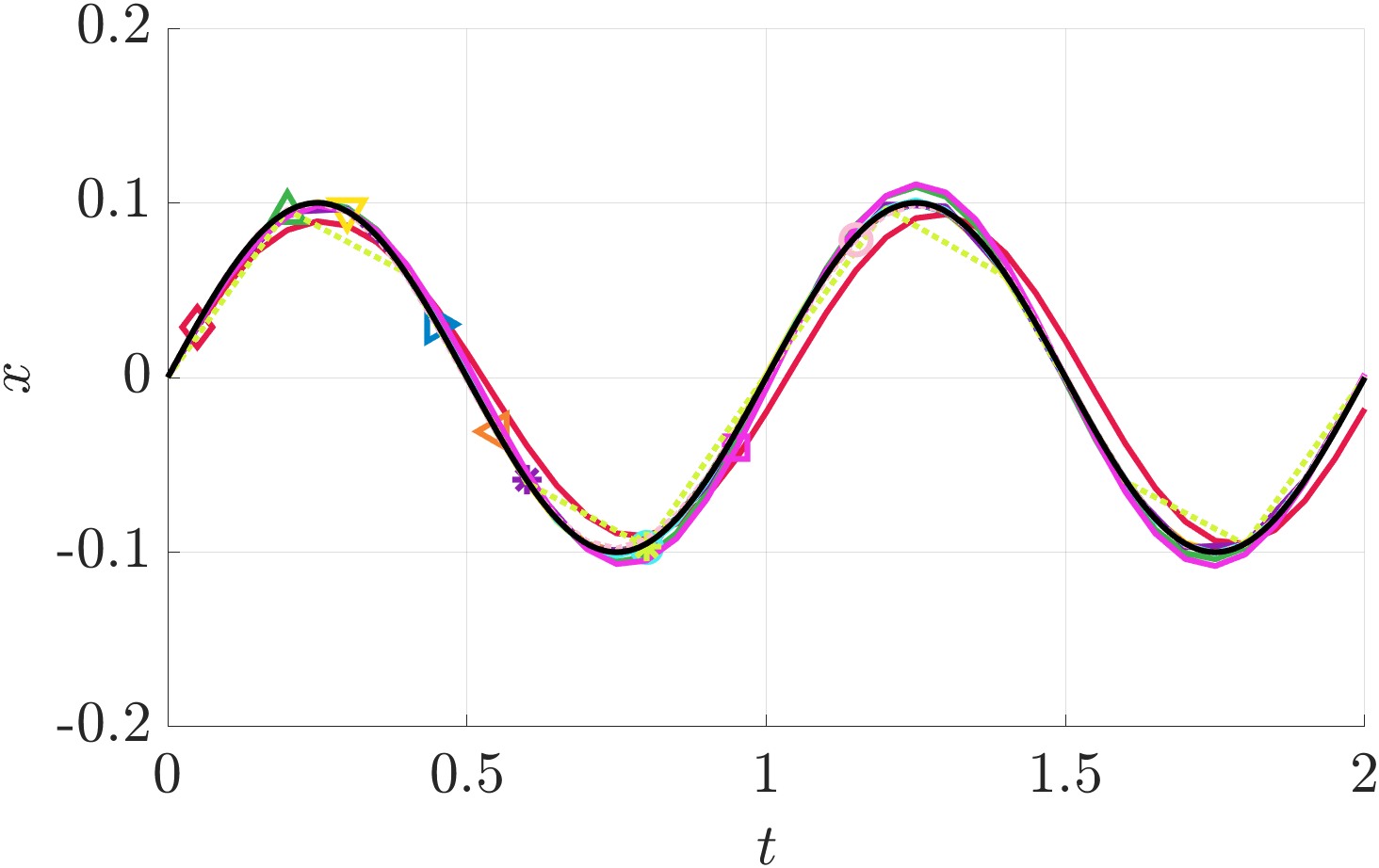}}
    \subcaptionbox{$\theta$ when $\Delta t=T_1/20$}{\includegraphics[width=0.49\textwidth]{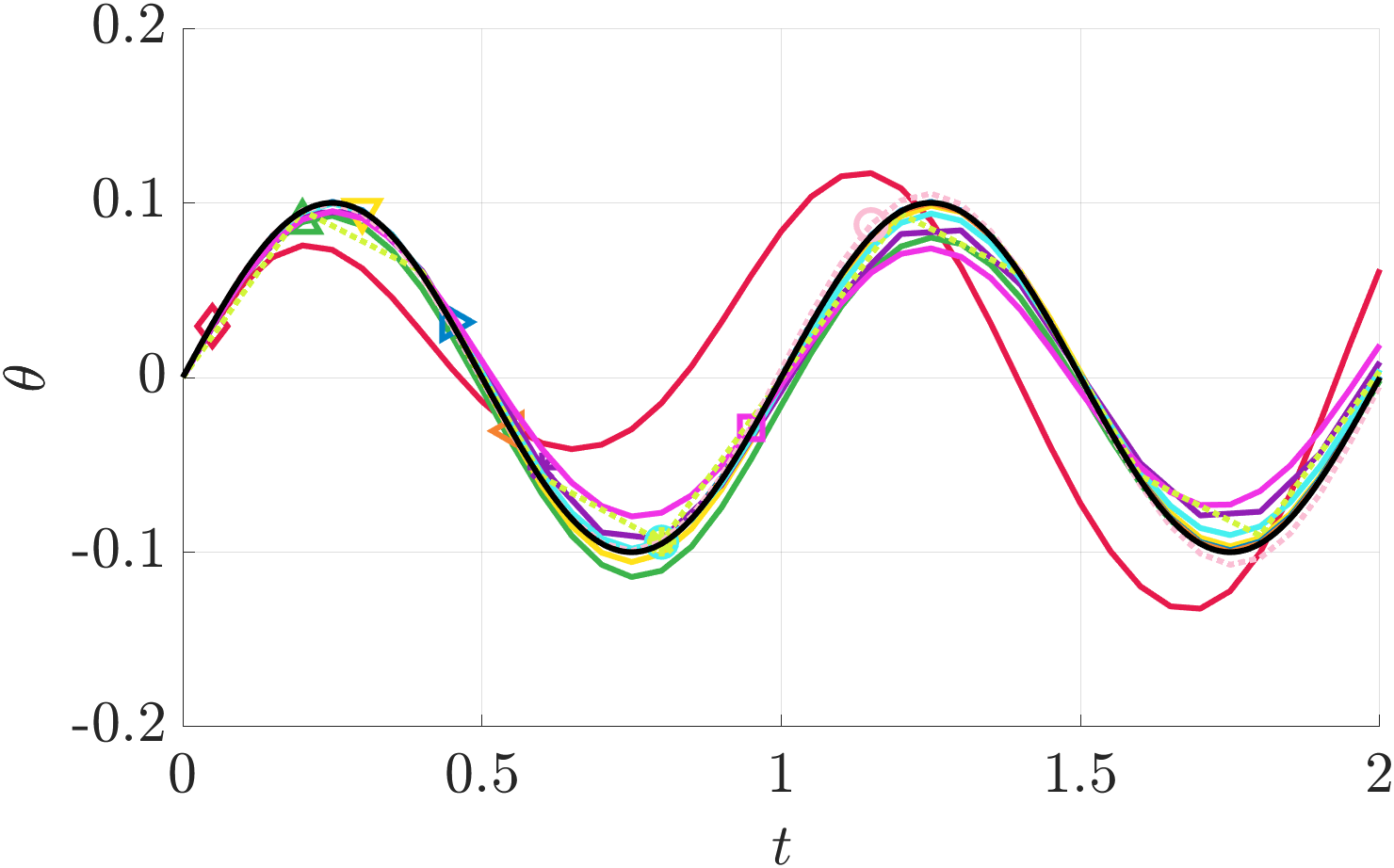}}
    \caption{Time history of the radial displacement $x$ and circumferential displacement $\theta$.}
    \label{fig:4thSpringPendulum_d}
\end{figure}

\begin{table}[!htbp]
    \caption{\label{tab:4thSpringPendulumSummary} Maximum errors and computational cost for various schemes using $\Delta t=T_1/10$ and $T_1/20$ over the time interval $[0,\,2T_1]$ in the spring-pendulum problem.}
    \centering 
    \resizebox{\linewidth}{!}{
    \begin{tabular}{l|c c| r r| c r| c r| c  r}
         & \multicolumn{2}{c|}{Avg. \#}& \multicolumn{2}{c|}{Comput.}& \multicolumn{6}{c}{Error (\%) \vspace{-2mm}}\\ 
         &\multicolumn{2}{c|}{of iter.}&  \multicolumn{2}{c|}{cost ($\mu s$)}& \multicolumn{2}{c|}{$\epsilon(\boldsymbol{u})$}& \multicolumn{2}{c|}{$\epsilon(\boldsymbol{v})$}& \multicolumn{2}{c}{$\epsilon(\boldsymbol{a})$}\\
 \multicolumn{1}{r|}{$\Delta t$}& $T_1/10$& $T_1/20$& $T_1/10$& $T_1/20$& $T_1/10$& $T_1/20$&  $T_1/10$&$T_1/20$& $T_1/10$&$T_1/20$\\
 Scheme&  &&  &&  &&  && &\\ \hline
         IMEX-BDF1-SAV & 1\phantom{.00} 
 &1\phantom{.000}
& 178 &326& 86.02 &45.22& 55.16 &29.29& 47.18&28.27\\
         IMEX-BDF2-SAV & 1\phantom{.00} 
 &1\phantom{.000}
& 167 &335& 36.10 &10.04& 24.99 &7.59& 37.81&11.16\\
         IMEX-BDF3-SAV & 1\phantom{.00} 
 &1\phantom{.000}
& 203 &323& 39.40 &2.91& 22.40 &1.93& 30.59&2.54\\
         IMEX-BDF4-SAV & 1\phantom{.00} 
 &1\phantom{.000}
& 176 &306& 12.68 &0.78& 10.54 &0.50& 17.48&0.71\\
         IMEX-BDF5-SAV & 1\phantom{.00} 
 &1\phantom{.000}
& 157 &301& 11.27 &0.27& \phantom{0}9.33 &0.16& 15.52&0.19\\
         Bathe         & 2.95
 &2.425
& 409 &743& 27.61 &9.14& 18.38 &5.59& 19.00&5.74\\
         TR & 2.8\phantom{0}
 &2\phantom{.000}
& 418 &622& 16.84 &4.92& 10.74 &3.00& 11.30&3.02\\
         Generalized-$\alpha$ & 3\phantom{.00}
 &2.5\phantom{00}
& 613 &1012& 41.15 &15.20& 31.70 &11.52& 48.35
&23.00\\
         Kim's $4^\textrm{th}$-order & 1\phantom{.00}
 &1\phantom{.000}
& *\phantom{0} &105& * &2.27& * &1.51& *&2.30\\
         CD (PU) & 1\phantom{00} &1\phantom{.000}
& 94 &182& 13.32 &4.25& 14.29 &6.18& 10.80&4.22\end{tabular}
    }
    \raggedright
    {\scriptsize
     * Unstable
     }
\end{table}

\newpage

\subsection{Multi-degree-of-freedom (MDOF) Duffing oscillator}

In this section, we solve a 20-DOF Duffing oscillator system as shown in Fig.~\ref{fig:4thMDOFDuffing}.
The Duffing oscillators are connected in series such that DOFs $i$ and $i-1$ are connected with a damper with coefficient $c_i$, linear hardening stiffness $k_{1,i}$, and cubic hardening stiffness $k_{3,i}$.
Conceptually, this problem can be thought of as a 20-story shear building with nonlinear springs.
The equation of motion at DOF $i$ of this system can be written as:
\begin{equation}
    \begin{array}{l}
        m_i \ddot{u}^i + \underbrace{c_i(\dot{u}^i-\dot{u}^{i-1}) + c_{i+1}(\dot{u}^i-\dot{u}^{i+1})}_{\text{Linear damping}} + \underbrace{k_{1,i}(u^i-u^{i-1}) + k_{1,i+1}(u^i-u^{i+1})}_{\text{Linear stiffness}}  \\
        + \underbrace{k_{3,i}(u^i-u^{i-1})^3 + k_{3,i+1}(u^i-u^{i+1})^3}_{\text{Nonlinear force}} = f^{ext,i} 
    \end{array}     
\end{equation}
We select the following values of the problem parameters: $m_i=1$, $c_i=0.3$, $k_{1,i}=1$ and $k_{3,i}=10$ for all $i=1,2,...,20$. 
External force, $f^{ext,i}=0$ for $i=1,2,...,19$, and $f^{ext,20}=\cos{\omega_p t}$, where $\omega_p=1$.
We impose zero initial displacement and velocity, $\boldsymbol{u}(0) = \boldsymbol{v}(0) = \boldsymbol{0}$.
During the simulation, the maximum value of pseudo-energy stored in the underlying undamped linear system is approximated as $\Psi_{max} \approx \frac{1}{2}\boldsymbol{u}_{max}^T\boldsymbol{K}\boldsymbol{u}_{max} \approx  40 $, where $\boldsymbol{u}_{max} \approx 2 [\boldsymbol{K}]^{-1} \max (\boldsymbol{f}^{ext})$, $\boldsymbol{K}$ is the linear stiffness matrix, and $\boldsymbol{f}^{ext}$ is external force vector.
Thus, we choose $\psi = 100\Psi_{max}=4000$ for the proposed IMEX-BDF$k$-SAV schemes.

\begin{figure}[!htbp]
    \centering
    {\includegraphics[width=\textwidth]{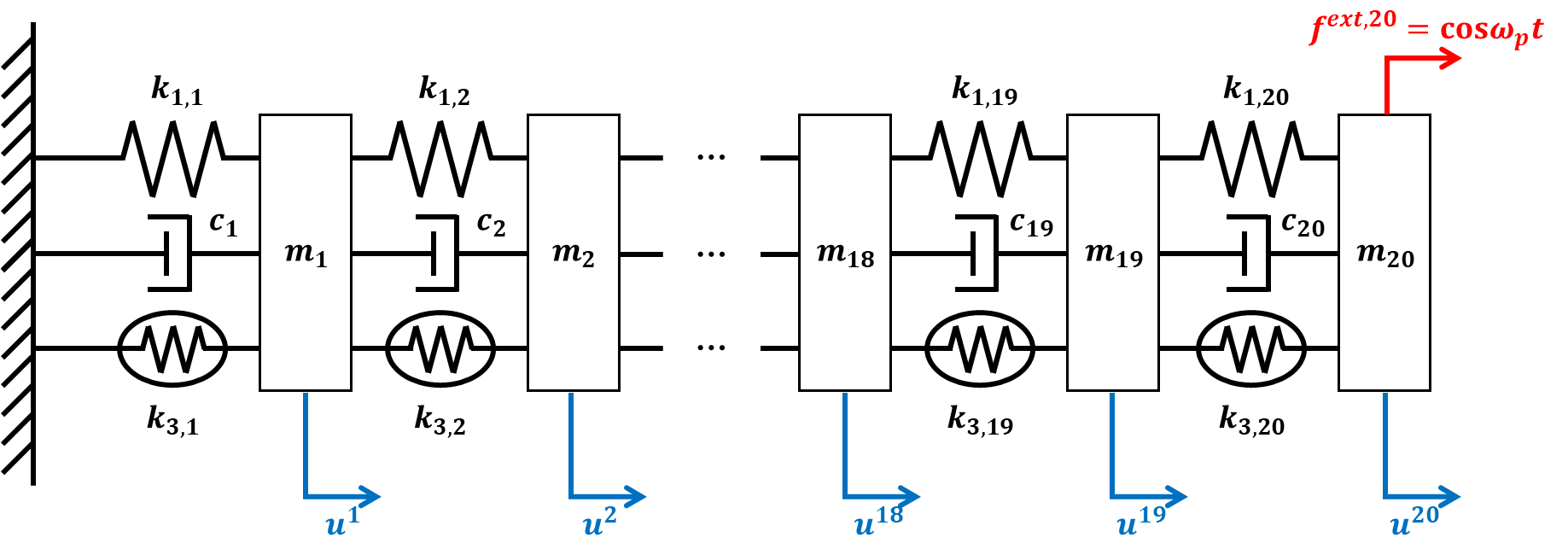}}
    \caption{A nonlinear system comprising 20 Duffing oscillators.}
    \label{fig:4thMDOFDuffing}
\end{figure}

To generate a reference solution to this problem, the standard Bathe scheme~\cite{bathe2005composite,bathe2007conserving}, Kim's $4^\textrm{th}$-order scheme and the RK4 method are used to solve this problem with a time-step of $10^{-5}$.
The maximum difference (error), as defined in Eq.~(\ref{eqn:4thErr_definition}), between the solutions obtained from these methods is found to be of the order of $10^{-10}$.
With this difference (error) in mind, the standard Bathe scheme with the time-step size of $10^{-5}$ is taken to be a reference solution for comparing errors among different time integration schemes for this problem.

Fig.~\ref{fig:4thMDOFDuffing_plot} shows time histories of all kinematic quantities at node 10 over the time duration $[0,50]$ for the proposed IMEX-BDF$k$-SAV, Bathe, TR, and generalized-$\alpha$, Kim's $4^\textrm{th}$-order, and CD schemes.
The time-step chosen for this problem is $n_{sub} \times \Delta t$, where $\Delta t = 0.2$.
Note that Kim's $4^\textrm{th}$-order scheme is unstable at this large time-step whereas the proposed IMEX-BDF$k$-SAV schemes are always unconditionally stable and they achieve up to $5^\textrm{th}$-order accuracy.

\begin{figure}[!htbp]
 \centering
 {\includegraphics[width=0.5\textwidth]{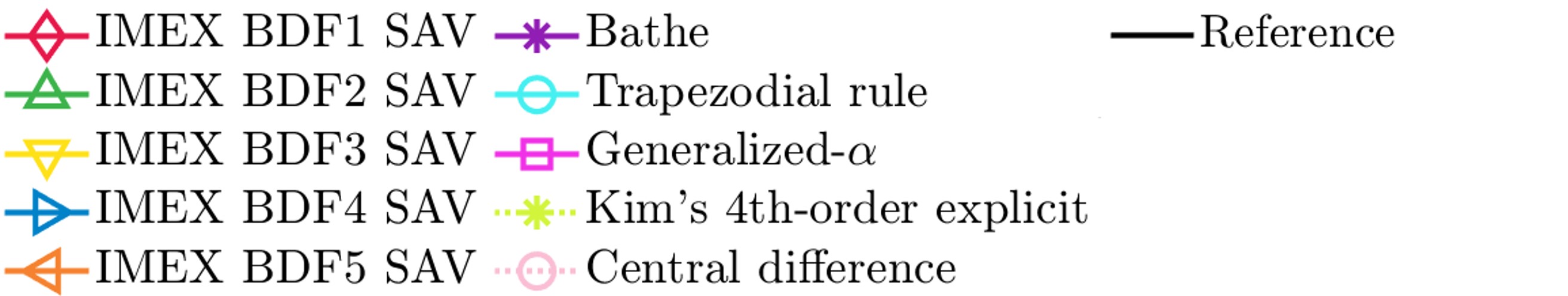}}
 \subcaptionbox{Displacement}{\includegraphics[width=0.8\textwidth]{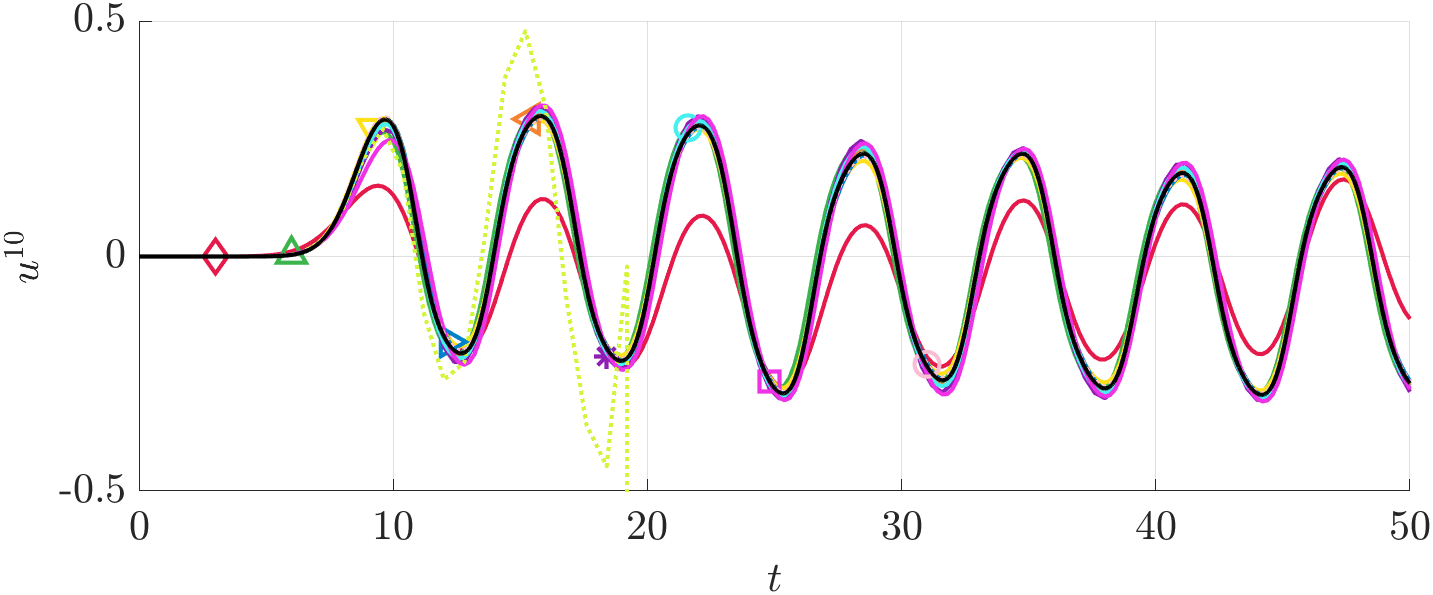}}
 \subcaptionbox{Velocity}{\includegraphics[width=0.8\textwidth]{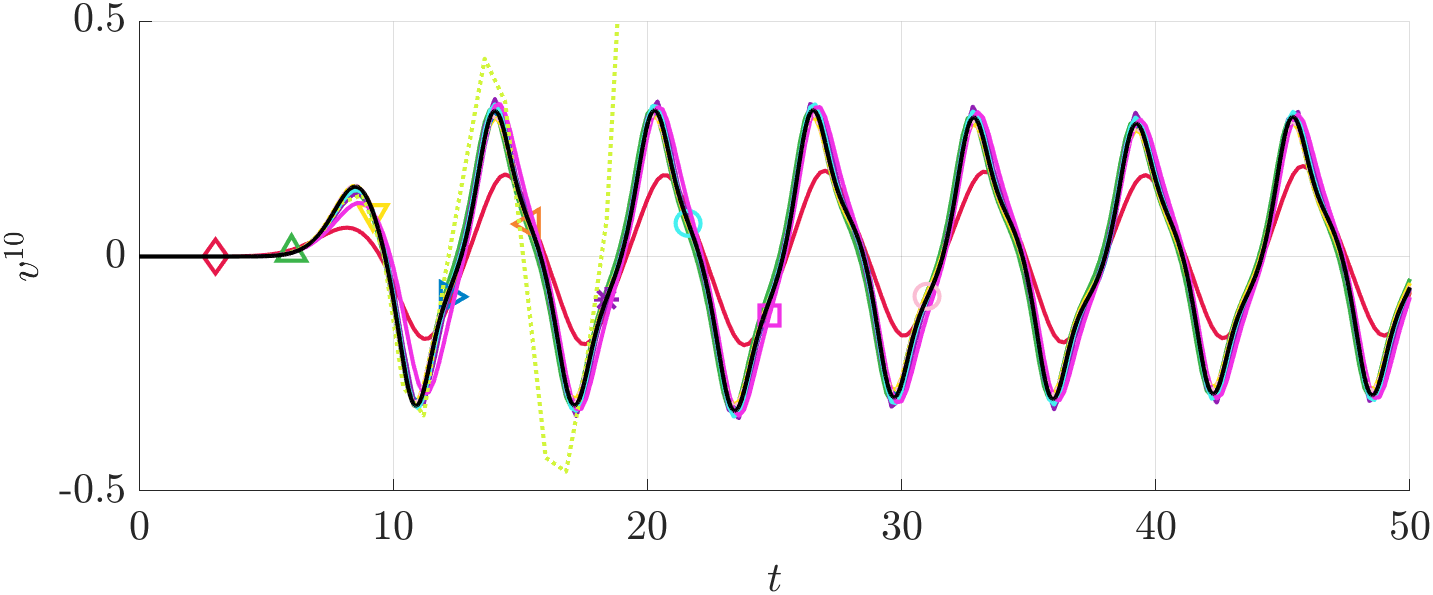}}
 \subcaptionbox{Acceleration}{\includegraphics[width=0.8\textwidth]{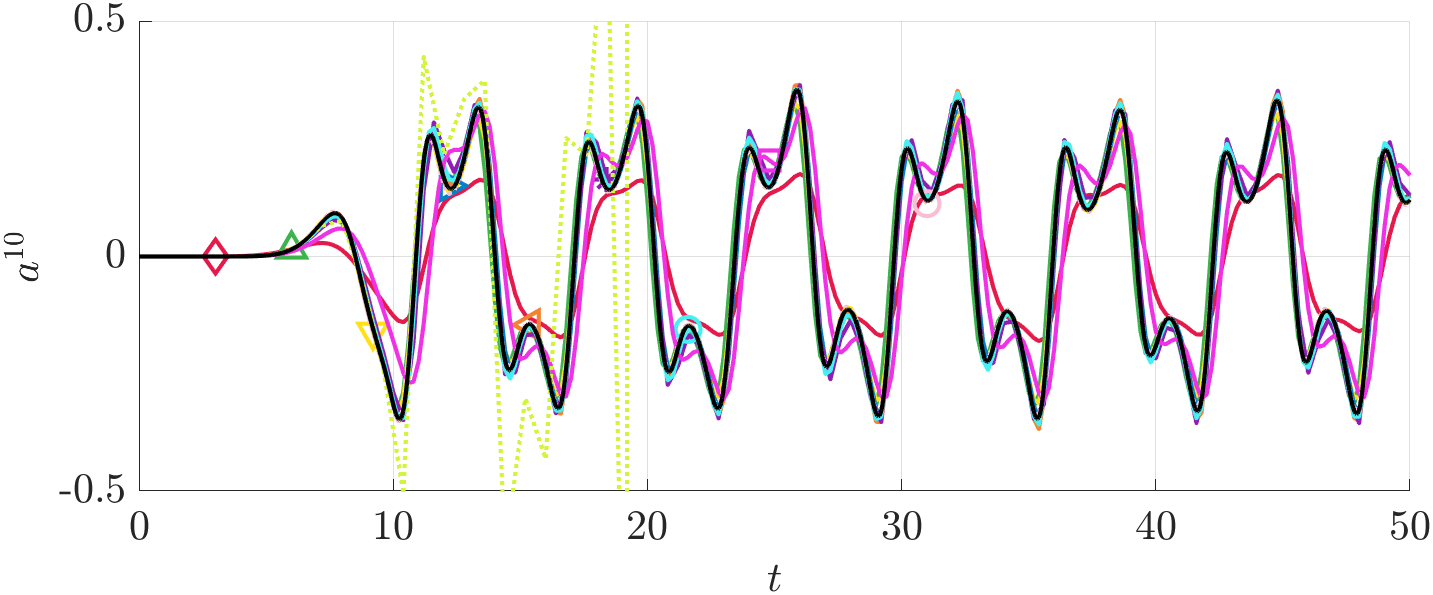}}
 \caption{Time histories of displacement, velocity, and acceleration at node 10 when $\Delta t=0.2$.}
 \label{fig:4thMDOFDuffing_plot}
\end{figure}

Fig.~\ref{fig:4thMDOFDuffing_error} shows the comparison between maximum errors and computational cost for different time integration schemes when $\Delta t$ is varied from 0.005 to 0.16 by doubling it successively 5 times.
Note that, for the same computational cost, explicit CD and Kim's $4^\textrm{th}$-order schemes yield more accurate results than the proposed IMEX-BDF2-SAV and IMEX-BDF4-SAV schemes, respectively.
However, explicit schemes are only conditionally stable and require a small time-step to compute stable solutions.
On the other hand, compared to $2^\textrm{nd}$-order implicit schemes, we note that the IMEX-BDF2-SAV scheme offers the best trade-off between computational cost and accuracy.
Thus, with the proposed IMEX-BDF$k$-SAV schemes, one essentially gets the benefit of unconditional stability, high accuracy and lowest computational cost.

\begin{figure}[!htbp]
    \centering
    \subcaptionbox{Displacement}{\includegraphics[width=0.49\textwidth]{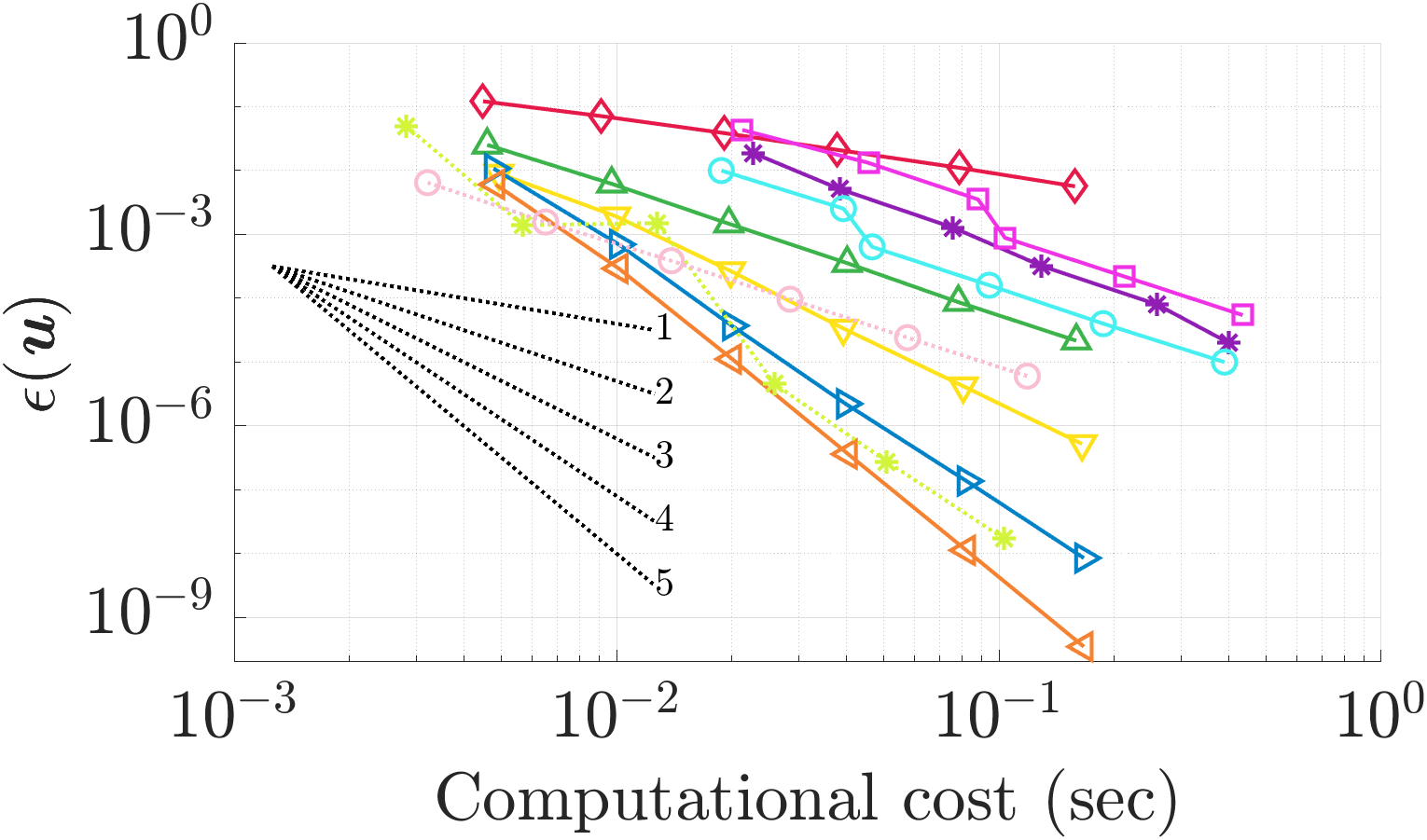}}
    \subcaptionbox{Velocity}{\includegraphics[width=0.49\textwidth]{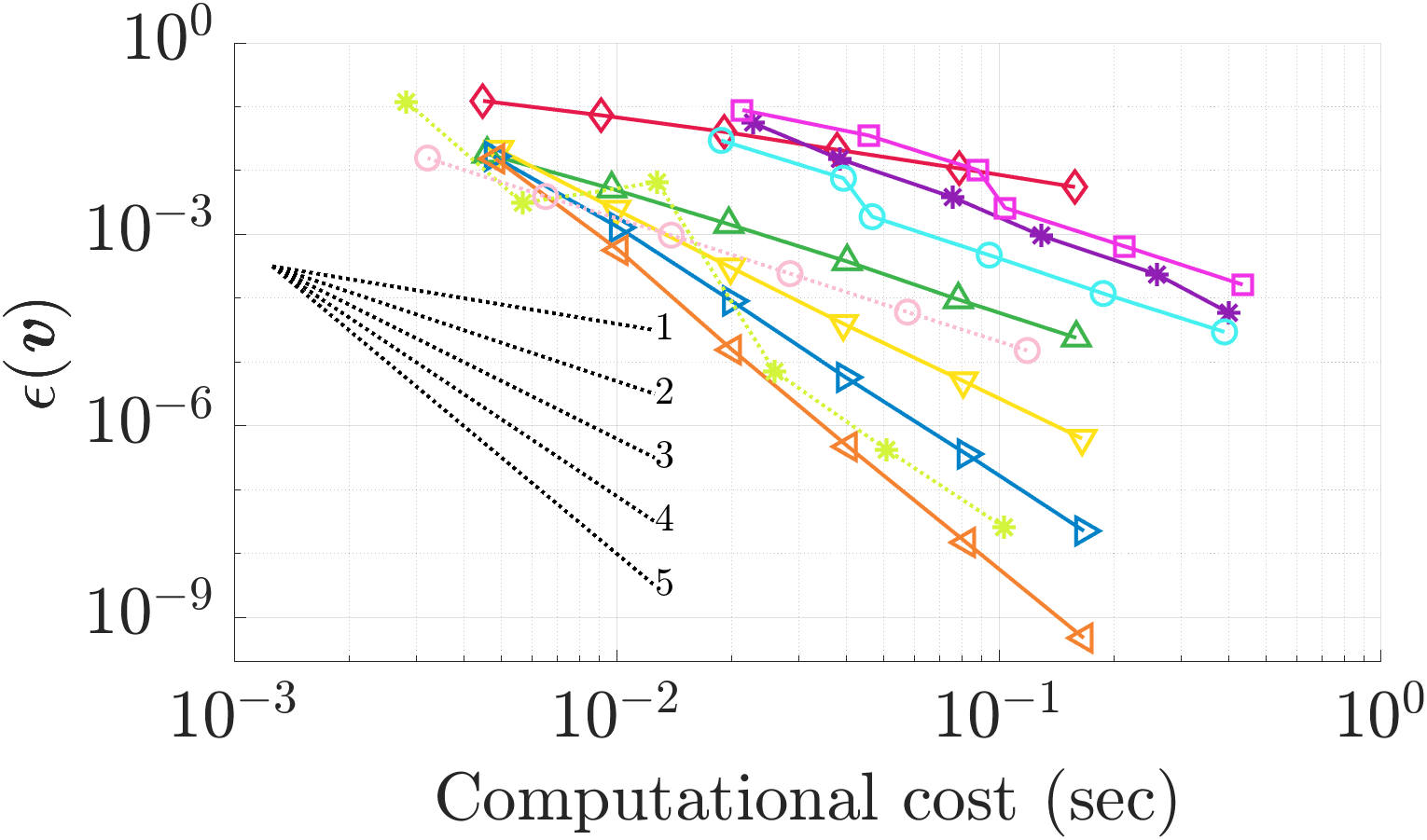}}
    \subcaptionbox{Acceleration}{\includegraphics[width=0.49\textwidth]{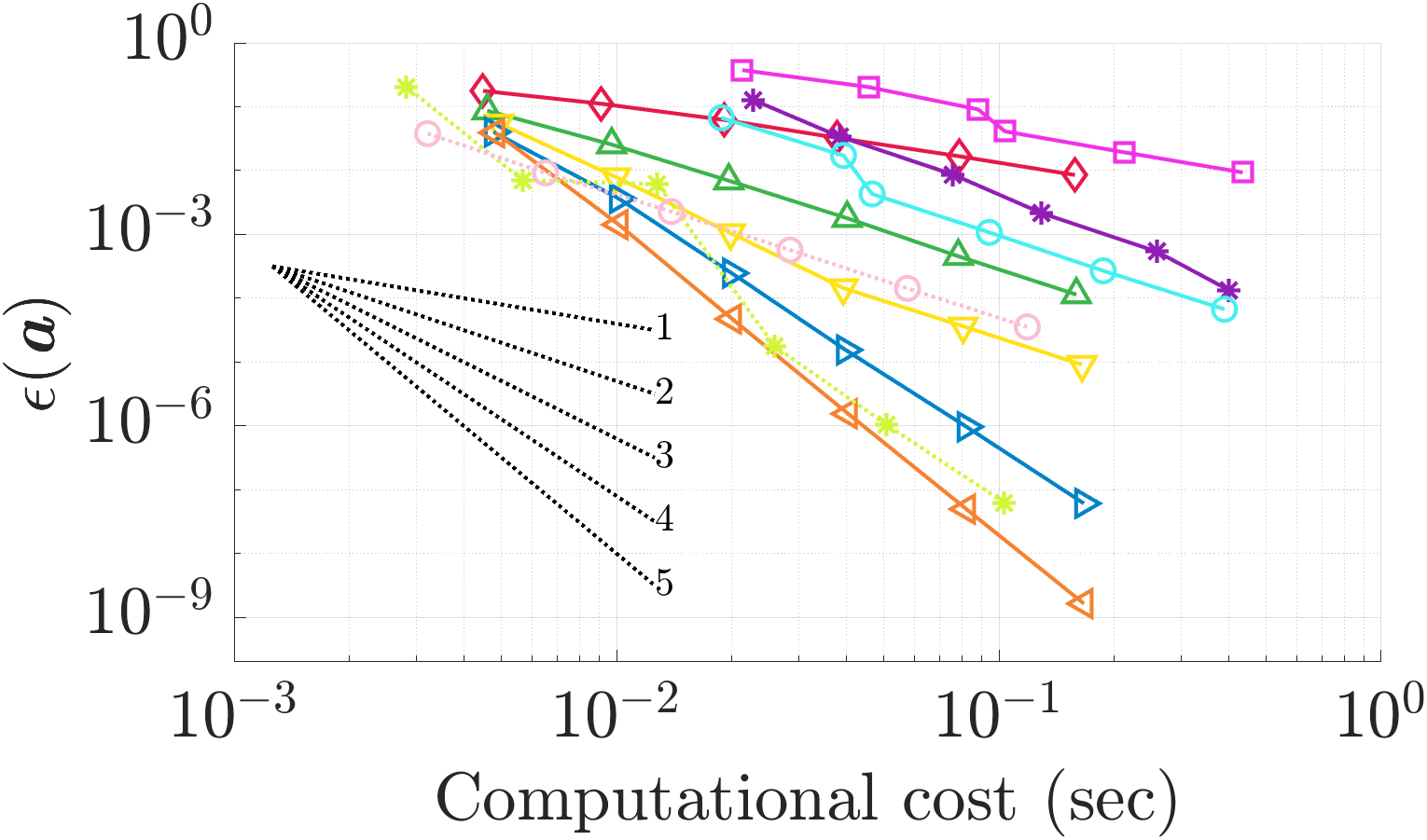}}
    {\includegraphics[width=0.25\textwidth]{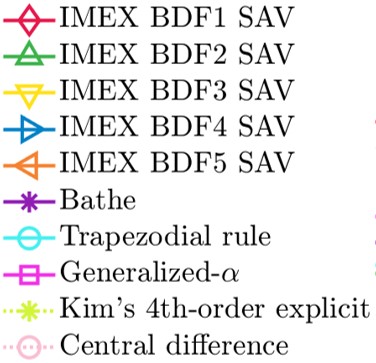}}
    \caption{Maximum errors in displacement, velocity, and acceleration.
    The 6 points on each curve correspond to values of $\Delta t \in \{0.005, 0.01, 0.02, 0.04, 0.08, 0.16\}$.
    }
    \label{fig:4thMDOFDuffing_error}
\end{figure}

\subsection{Nine-story building subject to El Centro earthquake}

As shown in Fig.~\ref{fig:4thNineStory}, we consider a 9-story building with a Buckling-Restrained Axial Damper (BRAD) on the first floor. 
This 9-story structure has been used as a benchmark problem to study seismically-excited buildings in Ref.~\cite{ohtori2004benchmark,bunting2016characterizing,maghareh2016adaptive}.
The geometry, material properties and boundary conditions are identical to those of the nine-story benchmark building north-south (N-S) moment-resisting frame in Ref.~\cite{ohtori2004benchmark}.
Initial displacement and velocity vectors are taken to be zero and the following external force is applied to the structure:
\begin{equation}
    \boldsymbol{f}^{ext}(t) = -\boldsymbol{M}\boldsymbol{\Gamma} \ddot{x}_g(t)
\end{equation}
where $\ddot{x}_g$ is the ground acceleration and $\boldsymbol{\Gamma}$ is a loading vector that describes the inertial effects due to the ground acceleration.
The input ground motion is selected as the N-S component recorded at the Imperial Valley Irrigation District substation in El Centro, California, during the Imperial Valley, California earthquake of May 18, 1940 (see Ref.~\cite{2803}).
The input ground motion in Ref.~\cite{2803} is recorded at uniformly spaced intervals of 0.02 sec.
Thus, when a time-instant of $t_i$ is not a multiple of 0.02 sec, the ground acceleration $\ddot{x}_g(t_i)$ is obtained by shape-preserving piecewise cubic interpolation.

\begin{figure}[!htbp]
    \centering
    {\includegraphics[width=0.95\textwidth]{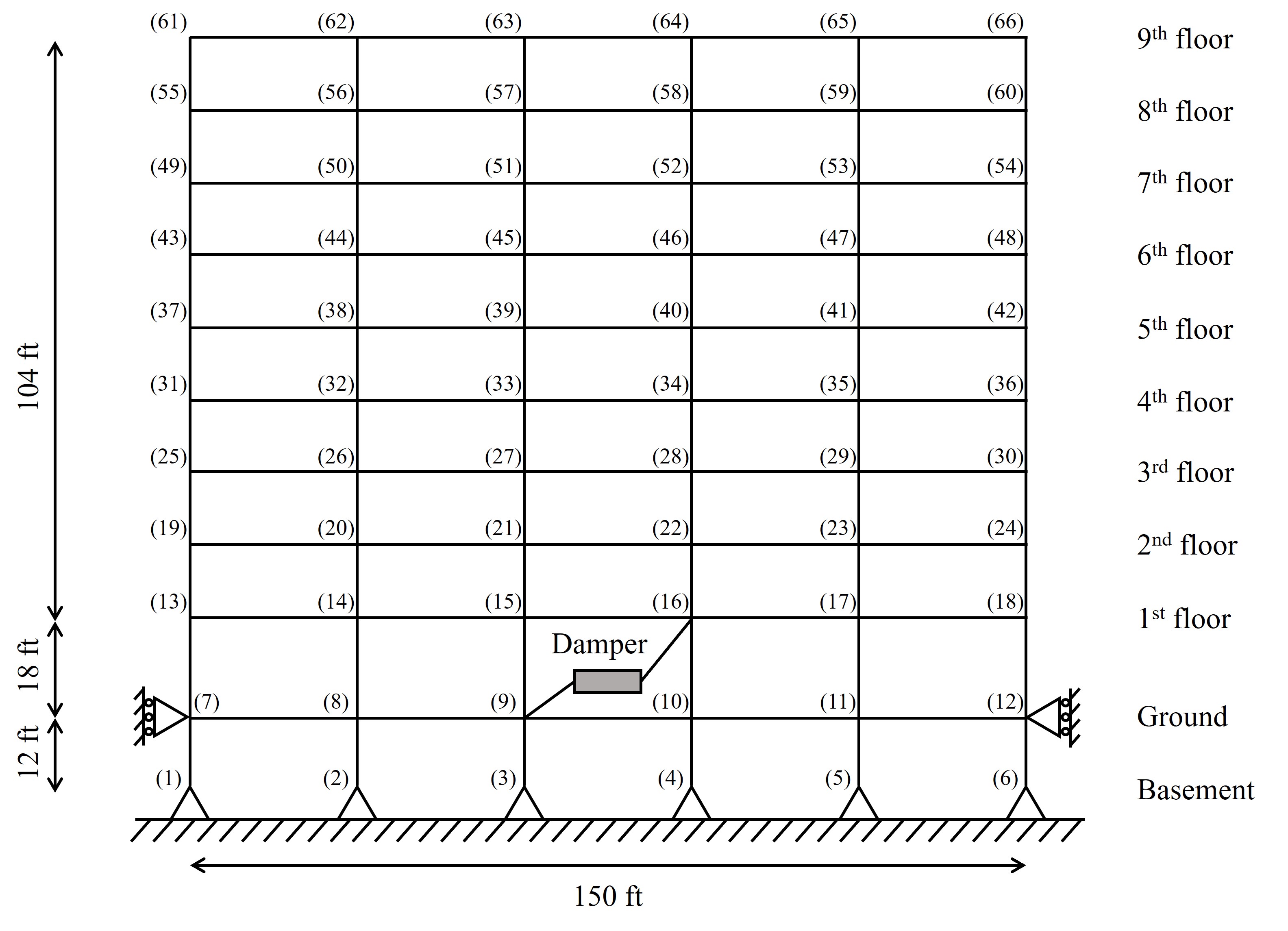}}
    \caption{A 9-story building with a Buckling-Restrained Axial Damper (BRAD).}
    \label{fig:4thNineStory}
\end{figure}

We use one 2-noded Euler-Bernoulli frame element within each column and beam for spatial discretization.
The mass matrix is obtained using the HRZ lumping scheme \cite{hinton1976note}. 
The corresponding first five natural frequencies are 0.443, 1.16, 1.99, 2.92, and 3.94 Hz, which are consistent with those in Ref.~\cite{ohtori2004benchmark}.
Mass- and stiffness-proportional Rayleigh damping coefficients are selected such that the first and fifth modal damping ratios are 0.02: $\zeta_1=\zeta_5=0.02$.
For modeling the BRAD, we use the extended Bouc–Wen model \cite{sireteanu2010identification}, where the restoring force $F(t)$ and hysteretic displacement $z(t)$ are given as:
\begin{align}
    &F(t) = \alpha k u(t) + (1-\alpha)k z(t) \\
    &\dot{z}(t) = \dot{u}(t) \left[ A - \beta |z(t)|^p - \gamma |z(t)|^q sgn(z(t)\dot{y}(t)) \right]
\end{align}
where $sgn$ denotes the sign function, and $\alpha$, $k$, $A$, $\beta$, $\gamma$, $p$, and $q$ are model parameters.
In this study, the following model parameters are chosen: $A=1$, $\alpha = 0.015$, $p=4.6$, $q=0.925$, $k = \varphi F_u / u_u$, $\beta = \hat{\beta}/|u_u|^p$, and $\gamma = \hat{\gamma}/|u_u|^q$ where $u_u=0.02$ m, $F_u=2\times 10^5$ N, $\varphi = 7.756$, $\hat{\beta}=-9.856$, and $\hat{\gamma}=7.185$ (see Ref.~\cite{sireteanu2010identification}).

In the absence of linear and nonlinear damping, we estimate the maximum pseudo-energy as $\Psi_{max} \approx \frac{1}{2}\boldsymbol{u}_{max}^T\boldsymbol{K}\boldsymbol{u}_{max} \approx  12$ MJ, where $\boldsymbol{u}_{max} \approx 2 [\boldsymbol{K}]^{-1} \max (|\boldsymbol{f}^{ext}|)$.
Similar to other examples, $\psi$ for the proposed IMEX-BDF$k$-SAV schemes is selected as $100\Psi_{max} = 1.2$ GJ.

A reference solution is obtained by solving the problem with the RK4 scheme and with Kim's $4^\textrm{th}$-order scheme using a small time-step of $10^{-6}$ sec.
The error (or difference) between these two solutions is found to be of the order of $10^{-13}$ in displacement.
The solution from the RK4 scheme is taken as the reference solution for computation of errors in the following discussion.

\begin{figure}[!htbp]
    \centering
    \subcaptionbox{Time history of displacement \label{fig:4th9Story_sol_a}}
    {\includegraphics[width=0.99\textwidth]{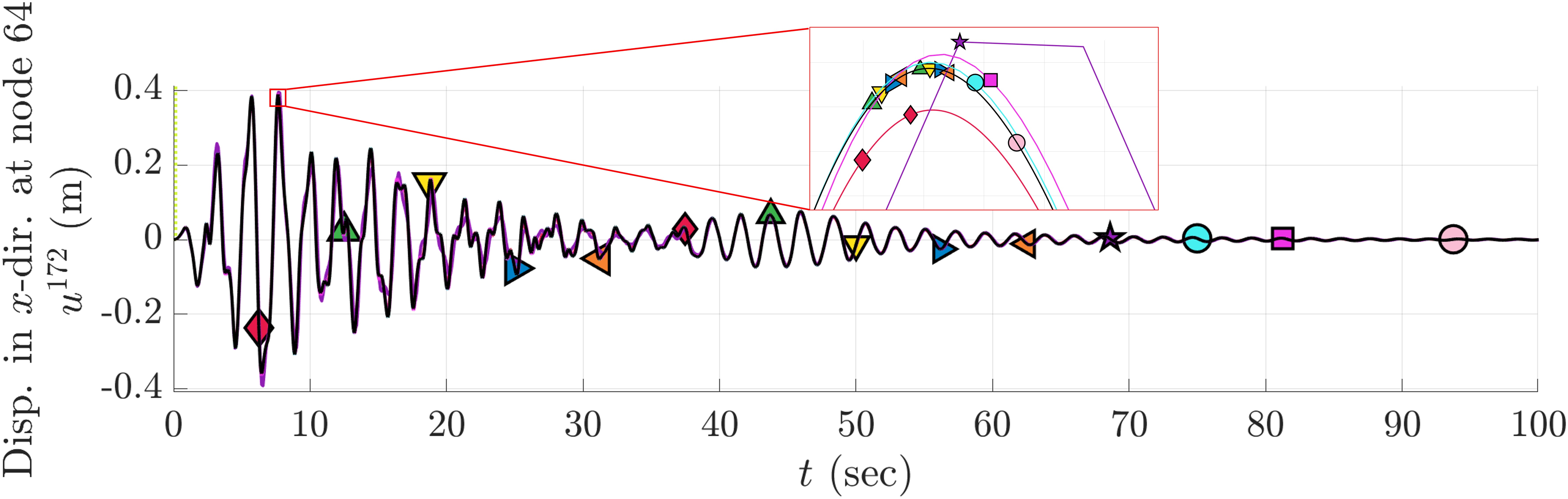}}
    \subcaptionbox{Time history of local instantaneous error \label{fig:4th9Story_sol_b}}
    {\includegraphics[width=0.99\textwidth]{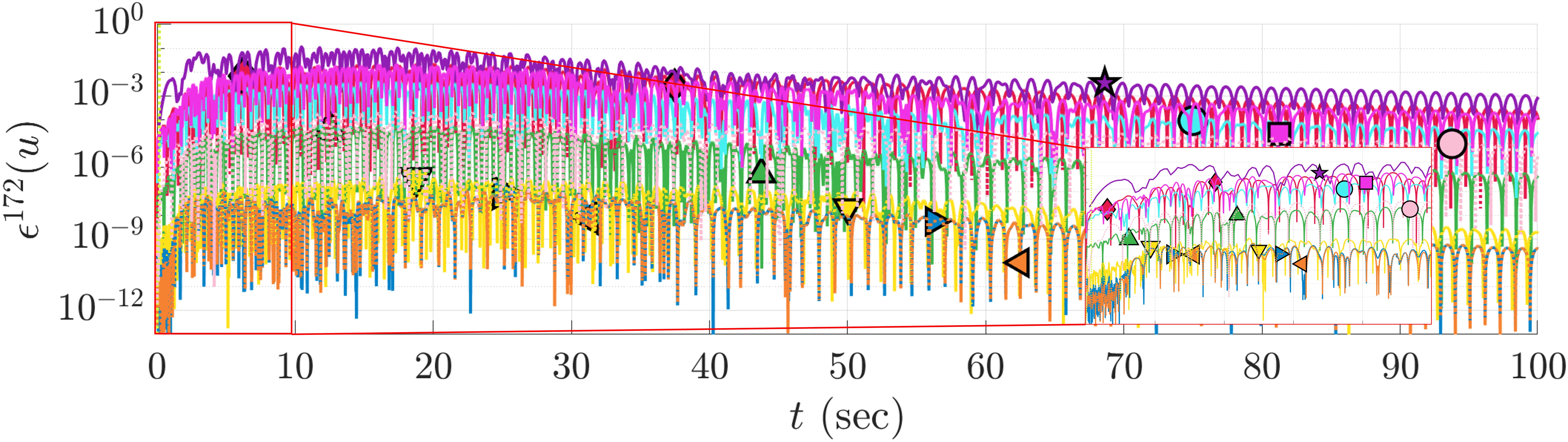}}
    {\includegraphics[width=0.45\textwidth]{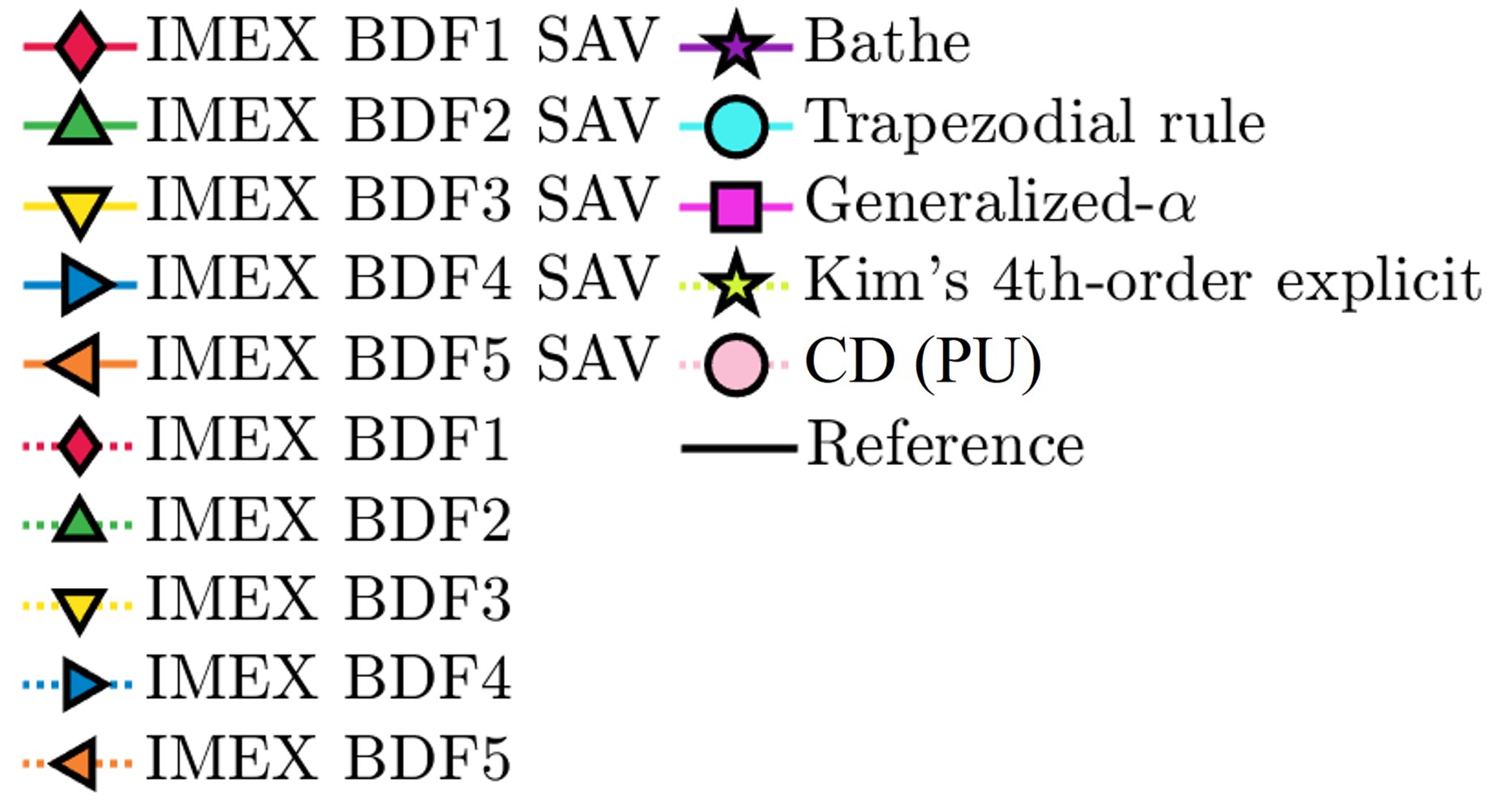}}
    \subcaptionbox{Maximum error vs. Computational cost for different time-steps \label{fig:4th9Story_sol_c}}
    {\includegraphics[width=0.49\textwidth]{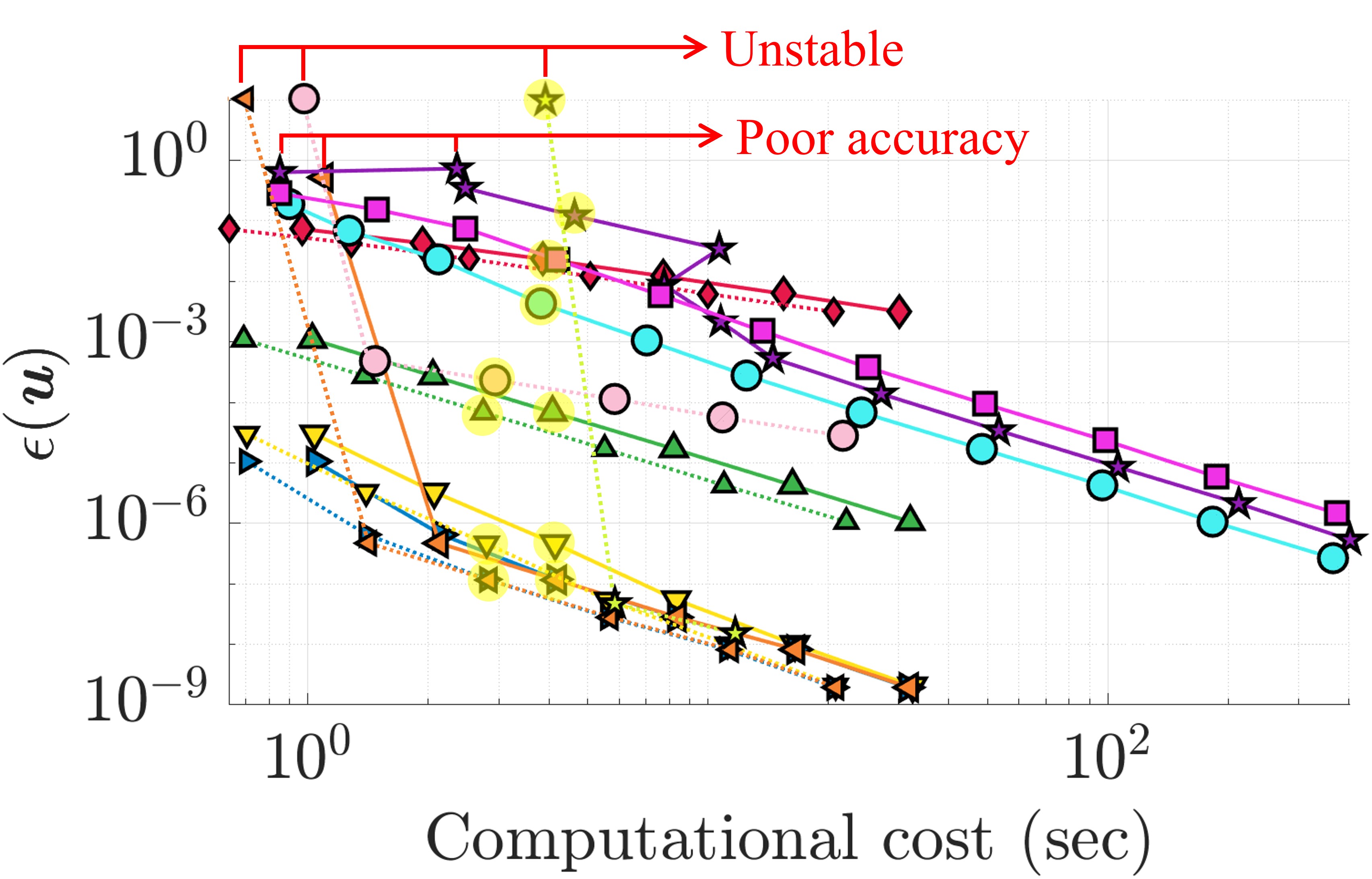}}
    \caption{(a) Time history of displacement in $x$-direction at node 64 for different time-steps (corresponding points in sub-figure~(c) are highlighted), (b) its local instantaneous error $\epsilon^{172}(d)$, (c) maximum error in displacement at any DOF for different time integration schemes.
    In sub-figure~(c), the smallest time-step for each scheme considered is $n_{sub} \times \Delta t$ where $\Delta t = 0.0001$ sec and the remaining points on each curve are obtained by successively doubling this time-step. 
    Note that the conventional IMEX-BDF5 and the central difference schemes are unstable for time-steps $\Delta t \ge 0.0032$ sec and Kim's $4^\textrm{th}$-order explicit scheme is unstable for time-step $(4 \times \Delta t) \ge (4\times0.0004$ sec$)$.}
    \label{fig:4th9Story_sol}
\end{figure}

\begin{figure}[!htbp]
    \centering
    {\includegraphics[width=0.99\textwidth]{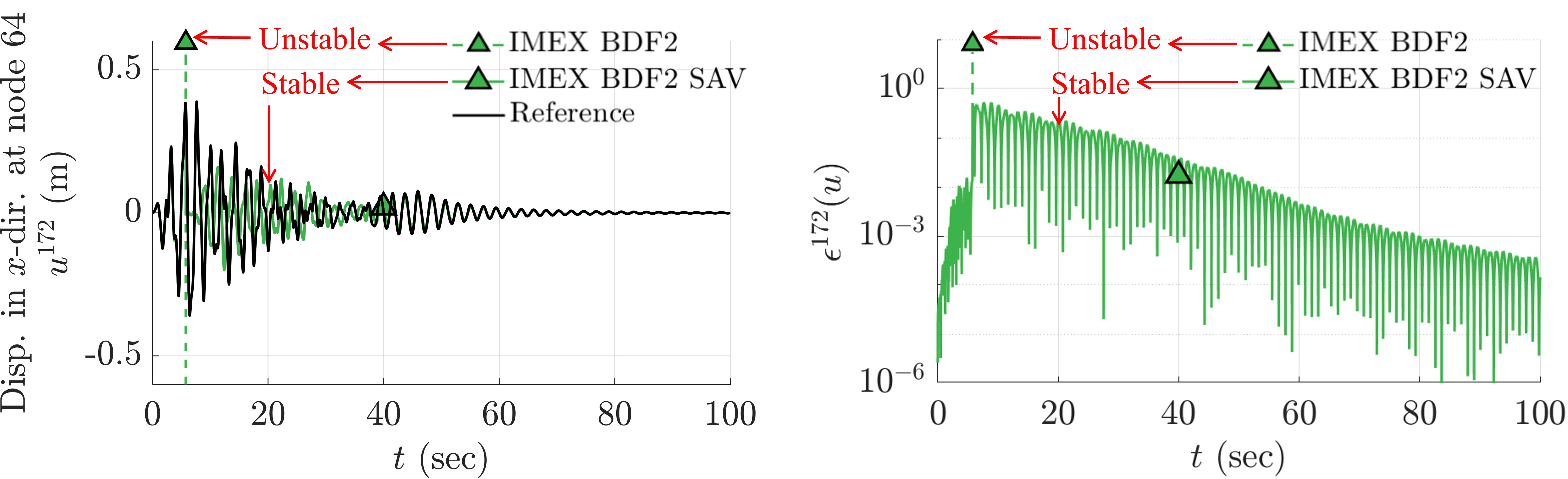}}    
    \caption{Time history of displacement and local instantaneous error obtained from IMEX-BDF2 and IMEX-BDF2-SAV schemes for $\Delta t = 0.02$ sec. Note the unconditional stability of IMEX-BDF2-SAV scheme when the underlying IMEX-BDF2 scheme becomes unstable.}
    \label{fig:4th9Story_sol_compare_IMEX}
\end{figure}

\begin{table}[!htbp]
    \caption{\label{tab:4th9StoryCost} Summary of time-steps, computational costs and maximum errors for different time-integration schemes presented in Fig.~\ref{fig:4th9Story_sol}}
    \centering 
    \resizebox{\linewidth}{!}{
    \begin{tabular}{l|c cc ||l |c cc}
 Scheme & $\Delta t$&  Comput.
&Error& Scheme & $\Delta t$& Comput.
&Error \vspace{-2mm}\\
 & &  cost (sec)
&$\epsilon(\boldsymbol{u})$& & & cost (sec)&$\epsilon(\boldsymbol{u})$\\ \hline
         IMEX-BDF1-SAV& 0.0008&  3.867389 
&$2.3 \times 10^{-2}$& Bathe& 0.0512& 4.639049 
&$1.2\times 10^{-1}$\\
 IMEX-BDF2-SAV & 0.0008&  4.086772 
&$6.8\times 10^{-5}$& TR& 0.0128& 3.823075 
&$4.3\times10^{-3}$\\
         IMEX-BDF3-SAV & 0.0008&  4.147435 
&$4.6 \times10^{-7}$& Generalized-$\alpha$ & 0.0128& 4.202484 
&$2.2\times10^{-2}$\\
         IMEX-BDF4-SAV & 0.0008&  4.166472 
&$1.2\times10^{-7}$& Kim's $4^\textrm{th}$ explicit& 0.0004& *&*\\
         IMEX-BDF5-SAV & 0.0008&  4.208458 
&$1.2\times10^{-7}$& CD (PU)& 0.0008& 2.936435 &$2.3\times10^{-4}$\\
 IMEX-BDF1& 0.0008&  2.532614 
&$2.3 \times 10^{-2}$& & & &\\
         IMEX-BDF2& 0.0008&  2.741602 
&$6.8 \times 10^{-5}$& & & &\\
         IMEX-BDF3& 0.0008&  2.797186 
&$4.6\times10^{-7}$& & & &\\
         IMEX-BDF4& 0.0008&  2.806063 
&$1.2\times10^{-7}$& & & &\\
         IMEX-BDF5& 0.0008&  2.842503 &$1.2\times10^{-7}$& & & &\\\end{tabular}    
         }
         \\
         \raggedright
    {\scriptsize
     * Unstable
     }
\end{table}

In Fig.~\ref{fig:4th9Story_sol}, sub-figure (a) presents the time history of displacement in the $x$-direction at the node 64 (DOF 172), sub-figure (b) shows the time-history of error, and sub-figure (c) shows a comparison of the computational cost and accuracy of the different time integration schemes.
In Figs.~\ref{fig:4th9Story_sol}(a) and \ref{fig:4th9Story_sol}(b), the time-steps for the various schemes presented are chosen such that the computational cost of all the schemes is similar --
around 3.5 seconds, as highlighted in Fig.~\ref{fig:4th9Story_sol}(c).
The specific time-steps chosen and associated computational costs of all the schemes considered are listed in Table~\ref{tab:4th9StoryCost}.
Fig.~\ref{fig:4th9Story_sol}(b) shows that for all schemes, local instantaneous error is higher in the initial part of the simulation when the earthquake excitation is higher and it steadily decreases with time.
Fig.~\ref{fig:4th9Story_sol}(c) shows that the conventional IMEX-BDF5 and the central difference schemes are unstable for time-steps $\Delta t \ge 0.0032$ sec and Kim's $4^\textrm{th}$-order explicit scheme is unstable for time-step $(n_{sub} \times \Delta t) \ge (n_{sub}\times0.0004$ sec$)$ (where $n_{sub} = 4$).
Further, from Table \ref{tab:4th9StoryCost}, we note that the proposed IMEX-BDF$k$-SAV schemes give stable and progressively accurate solutions as $k$ increases from 1 to 5.
One exception to this statement is that the IMEX-BDF5-SAV scheme results in poor accuracy for $\Delta t \ge 0.0032$ sec.
Note that for this time-step, the underlying IMEX-BDF5 scheme is actually unstable.
Upon further investigation, it was found that when the underlying IMEX-BDF$k$ scheme becomes unstable, the associated IMEX-BDF$k$-SAV scheme results in poor accuracy for some duration of the simulation, but it remains unconditionally stable.
For problems with a non-zero forcing function, the proposed IMEX-BDF$k$-SAV schemes are also able to recover from the local temporal inaccuracy and give stable and accurate solutions for the remainder of the simulation.
For example, Fig.~\ref{fig:4th9Story_sol_compare_IMEX} shows that for $\Delta t = 0.02$ sec (a relatively large time-step for this problem), the IMEX-BDF2 scheme becomes unstable about 5 seconds into the simulation.
At this point, we see a sharp rise in error of the corresponding IMEX-BDF2-SAV scheme which steadily decreases over time.
Finally, we observe that the proposed IMEX-BDF$k$-SAV scheme with $k \ge 2$ provides the most accurate results when compared to other implicit and explicit schemes for the computational cost.

\section{Concluding remarks}
\label{sec:4th4}
In this study, we present a SAV stabilization of IMEX BDF$k$ schemes for structural dynamics, and its stability and error analysis.
Unlike existing SAV schemes, the SAV stabilization is designed to solve second-order ODEs, such as general nonlinear problems in structural dynamics.
In the proposed IMEX-BDF$k$-SAV schemes, linear terms are handled implicitly while nonlinear terms are handled explicitly.
Thus, there is no need for iterative algorithms, such as Picard iteration and Newton-Raphson iteration, to solve nonlinear terms. 
Unlike conventional IMEX methods, which are conditionally stable, the proposed IMEX-BDF$k$-SAV schemes employ a novel scalar auxiliary variable to maintain unconditional stability.
In addition, the proposed schemes are able to achieve $k^{\textrm{th}}$-order accuracy.

A rigorous stability analysis is conducted to show that the proposed schemes are unconditionally stable for both linear and nonlinear dynamics.
Moreover, we carry out a detailed truncation error analysis.
From the error analysis, we show that the optimal value of parameter $\beta^{(k)}$, given by Eq.~(\ref{eqn:4thbetaPara}), is about half of its value in existing IMEX-BDF$k$-SAV schemes for other dynamical systems governed by first-order ODEs.
This leads to a lower energy bound for unconditional stability.
We also prove that the proposed IMEX-BDF$k$-SAV schemes with these values preserve the local truncation errors of the conventional BDF$k$ schemes for linear structural dynamics.
Convergence tests are conducted to validate these theoretical findings.

We illustrate the performance of the proposed IMEX-BDF$k$-SAV schemes by solving various nonlinear problems: a simple pendulum, a spring-pendulum, a MDOF Duffing oscillator, a 9-story building with a nonlinear damper subjected to an earthquake excitation.
We compare the numerical results obtained using the proposed schemes and several implicit (Bathe, TR, and generalized-$\alpha$) and explicit (Kim's $4^{\textrm{th}}$-order and CD) time integration schemes using different measures of error.
These numerical examples demonstrate that the proposed IMEX-BDF$k$-SAV schemes yield accurate solutions at a low computational cost while maintaining unconditional stability.

\section*{Acknowledgments}
This work is supported by the Collaborative Research CPS Co-Designed Control and Scheduling Adaptation for Assured Cyber-Physical System Saftey and Performance through NSF CNS-2229136.

\begin{appendices}
\counterwithin{equation}{section} 
\renewcommand{\theequation}{\thesection\arabic{equation}} 

\section{Local truncation errors in forced responses}
\label{ap:4thLTEforced}
With initial displacement $u_0$ and velocity $v_0$ given, the LTEs of the proposed IMEX-BDF$k$-SAV schemes for forced responses (i.e. $f^{ext}=p_0 \sin(\omega_f t)$) are obtained as follows. 
Note that LTEs in displacements and velocities are found to be of the order of $\Delta t^{k+1}$: $\tau_u \propto O(\Delta t^{k+1})$ and $\tau_v \propto O(\Delta t^{k+1})$ for $1 \le k \le 5$:

i) $k=1$:
\begin{equation}
    \begin{bmatrix}
        \tau_{u} \\ \tau_{v}
    \end{bmatrix} \propto
    -\frac{\omega_0^2\Delta t^2}{2}
    \begin{bmatrix}
        u_0 (1+2\zeta\kappa) \\
        v_0 (1-4\zeta^2-2\zeta/\kappa)
    \end{bmatrix} 
    +\frac{p_0 \omega_f \Delta t^2}{2}
    \begin{bmatrix}
        0 \\
        1
    \end{bmatrix} 
    + O(\Delta t^3)
\end{equation}

ii) $k=2$:
\begin{equation}
    \begin{array}{ll}
    \begin{bmatrix}
        \tau_{u} \\ \tau_{v}
    \end{bmatrix} \propto &
    -\dfrac{2\omega_0^3\Delta t^3}{9}
    \begin{bmatrix}
        u_0 (-2\zeta+(1-4\zeta^2)\kappa) \\
        v_0 (-4\zeta(1-2\zeta^2)-(1-4\zeta^2)/\kappa)
    \end{bmatrix} \\
    & +\dfrac{2 p_0 \omega_f \Delta t^3}{9}
    \begin{bmatrix}
        1 \\
        -2 \omega_0 \zeta
    \end{bmatrix} + O(\Delta t^4)
    \end{array}
\end{equation}

iii) $k=3$:
\begin{equation}
    \begin{array}{ll}
    \begin{bmatrix}
        \tau_{u} \\ \tau_{v}
    \end{bmatrix} \propto &
    \dfrac{3\omega_0^4\Delta t^4}{22}\begin{bmatrix}
        u_0 (1-4\zeta^2+4\zeta(1-2\zeta^2)\kappa) \\
        v_0 (1-12\zeta^2+16\zeta^4-4\zeta(1-2\zeta^2)/\kappa)
    \end{bmatrix} \\
    &-
    \dfrac{3 p_0 \omega_f \Delta t^4}{22}\begin{bmatrix}
        2 \omega_0 \zeta \\
        \omega_0^2 (1 - 4 \zeta^2) + \omega_f^2
    \end{bmatrix}
    + O(\Delta t^5)
    \end{array}
\end{equation}

iv) $k=4$:
\begin{equation}
    \begin{array}{ll}
    \begin{bmatrix}
        \tau_{u} \\ \tau_{v}
    \end{bmatrix} \propto &
    \dfrac{12\omega_0^5\Delta t^5}{125}
    \begin{bmatrix} 
        u_0 (-4\zeta(1-2\zeta^2)+(1-12\zeta^2+16\zeta^4)\kappa) \\
        v_0 (-2\zeta(3-16\zeta^2+16\zeta^4)-(1-12\zeta^2+16\zeta^4)/\kappa)
    \end{bmatrix} \\
    & -\dfrac{12 p_0 \omega_f \Delta t^5}{125}
    \begin{bmatrix} 
        \omega_0^2 (1 -4 \zeta^2) +\omega_f^2 \\
        -2 \omega_0 \zeta(2\omega_0^2 (1 -2\zeta^2)+\omega_f^2)
    \end{bmatrix} 
    + O(\Delta t^6)
    \end{array}
\end{equation}

v) $k=5$:
\begin{equation}
    \begin{array}{ll}
    \begin{bmatrix}
        \tau_{u} \\ \tau_{v}
    \end{bmatrix} \propto &
    -\dfrac{10\omega_0^6\Delta t^6}{137}
    \begin{bmatrix}
        u_0 (1-12\zeta^2+16\zeta^4+2\zeta(3-16\zeta^2+16\zeta^4)\kappa) \\
        v_0 (1-24\zeta^2+80\zeta^4-64\zeta^6-2\zeta(3-16\zeta^2+16\zeta^4)/\kappa)
    \end{bmatrix} \\
    &+\dfrac{10 p_0 \omega_f \Delta t^6}{137}
    \begin{bmatrix}
        -2\omega_0 \zeta (2\omega_0^2 (1-2\zeta^2)+\omega_f^2) \\
         \omega_0^4 (1 - 12\zeta^2 + 16\zeta^4) + \omega_0^2\omega_f^2 ( 1 - 4\zeta^2) + \omega_f^4
    \end{bmatrix} + O(\Delta t^7)
    \end{array}
\end{equation}
where $\kappa = v_0/(\omega_0 u_0)$.

\end{appendices}

\bibliography{mybibfile}

\end{document}